\documentclass[aos,preprint]{imsart}
\RequirePackage[OT1]{fontenc}
\RequirePackage{amsthm,amsmath}
\RequirePackage[numbers]{natbib}
\usepackage{pictexwd}
\usepackage{amssymb,graphicx,url}
\usepackage{color}
\usepackage{epstopdf}
\usepackage[normalem]{ulem}

\startlocaldefs    
\def\a{\alpha}
\def\b{\beta}
\def\R{\mathbb R}
\def\dd{\Delta}

\def\d{\delta}

\def\K{{\cal K}}

\def\E{{\mathbb E}}

\def\G{{\mathbb G}}

\def\P{{\mathbb P}}
\def\g{\gamma}
\def\l{\lambda}
\def\labda1{\lambda_1}
\def\labda2{\lambda_2}

\def\e{\varepsilon}
\def\f{\phi}
\def\t{\tau}

\def\s{\sigma}

\def\comment#1{\relax}

\def\=in{\mathop{\rm =}}

\newtheorem{theorem}{Theorem}[section]

\newtheorem{lemma}{Lemma}[section] 
\newtheorem{definition}{Definition}[section]

\newtheorem{remark}{Remark}[section]

\numberwithin{equation}{section}
\theoremstyle{plain}
\setattribute{journal}{name}{}
\endlocaldefs

\usepackage{amsmath,here,amsfonts,latexsym,sym1,graphicx,here}
\usepackage{amsthm,curves}
\usepackage{graphicx}
\usepackage{caption}
\usepackage{subcaption}
\usepackage{float}

\def\P{{\mathbb P}}

\def\G{{\mathbb G}}

\def\ee{\epsilon}

\begin{document}
\begin{frontmatter}
\title{Current status linear regression}
\runtitle{current status regression}

\begin{aug}
\author{\fnms{Piet} \snm{Groeneboom}\corref{}\ead[label=e1]{P.Groeneboom@tudelft.nl}
\ead[label=u1,url]{http://dutiosc.twi.tudelft.nl/\textasciitilde pietg/}}
and\ \
\author{\fnms{Kim} \snm{Hendrickx}\ead[label=e2]{kim.hendrickx@uhasselt.be}
\ead[label=u2,url]{http://www.uhasselt.be/fiche_en?voornaam=Kim&naam=HENDRICKX}}
\runauthor{P.\ Groeneboom and K.\ Hendrickx}
\affiliation{Delft University of Technology and Hasselt University}

\address{Delft University of Technology, Mekelweg 4, 2628 CD Delft,
	The Netherlands.\\ 
	\printead{e1}} 

\address{\\} 

\address{Hasselt University, I-BioStat, Agoralaan, B­3590 Diepenbeek, Belgium.\\ 
		\printead{e2}}
\end{aug}

\begin{abstract}
We construct $\sqrt{n}$-consistent and asymptotically normal estimates for the finite dimensional regression parameter in the current status linear regression model, which do not require any smoothing device and are based on maximum likelihood estimates (MLEs) of the infinite dimensional parameter. We also construct estimates, again only based on these MLEs, which are arbitrarily close to efficient estimates, if the generalized Fisher information is finite. This type of efficiency is also derived under minimal conditions for estimates based on smooth non-monotone plug-in estimates of the distribution function. Algorithms for computing the estimates and for selecting the bandwidth of the smooth estimates with a bootstrap method are provided.  The connection with results in the econometric literature is also pointed out.
\end{abstract}



\end{frontmatter}

\section{Introduction}
\label{section:intro}
\setcounter{equation}{0}
Investigating the relationship between a response variable $Y$ and one or more explanatory variables is a key activity in statistics. Often encountered in regression analysis however, are situations where a part of the data is not completely observed due to some kind of censoring. In this paper we focus on modeling  a linear relationship when the response variable is subject to interval censoring type I, i.e. instead of observing the response $Y$, one only observes whether or not $Y \le T$ for some random censoring variable $T$, independent of $Y$. This type of censoring is often referred to as the current status model and arises naturally, for example, in animal tumorigenicity experiments (see e.g. \cite{Finkelstein:85} and \cite{Finkelstein:86}) and in HIV and AIDS studies (see e.g. \cite{Shiboski:1992}). Substantial literature has been devoted to regression models with current status data including the proportional hazard model studied in \cite{huang:94}, the accelerated failure time model proposed by \cite{Rabinowitz:95} and the proportional odds regression model of \cite{Rossini:96}. 

The regression model we want to study is the semi-parametric linear regression model $Y = \b_0'X + \varepsilon$, where the error terms are assumed to be independent of $T$ and $X$ with unknown distribution function $F_0$. This model  is closely related to the binary choice model type, studied in econometrics (see e.g. \cite{cosslett:83,cosslett:07},  \cite{klein_spady:93} and \cite{sherman:05}), where, however, the censoring variable $T$ is  degenerate, i.e. $P(T = 0) =1$, and observations are of the type $(X_i,1_{\{Y_i\le 0\}})$. In the latter model, the scale is not identifiable, which one usually solves by adding a constraint on the parameter space such as setting the length of $\beta$ or the first coefficient equal to one.

Our model of interest is parametrized by the finite dimensional regression parameter $\b_0$ and the infinite dimensional nuisance parameter $F_0$ that contains $\b_0$ as one of its arguments. A similar bundled parameter problem was studied by \cite{Ding:11}, where the authors first provide a  framework for the distributional theory of problems with bundled parameters and next prove their theory for efficient estimation in the linear regression model with right censored data. A spline based estimate of the nuisance parameter is proposed.

Although it is indeed tempting to think that some kind of smoothing is needed, like the splines in \cite{Ding:11} or the kernel estimates in the econometric literature for the binary choice model, where even higher order kernels are used (see, e.g., \cite{klein_spady:93}), a maximum rank correlation estimate which does not use any smoothing has been introduced in \cite{han:87}, and this estimator has been proved to be $\sqrt{n}$-consistent and asymptotically normal in  \cite{sherman:93}. However, the latter estimate does not attain the efficiency bounds and one wonders whether it is possible to construct simple discrete estimates of this type and achieve the efficiency bounds. It is not clear how the maximum rank correlation estimate in \cite{han:87} could be used to this end, and we therefore turn to estimators depending on maximum likelihood estimators for the nuisance parameter.

The profile maximum likelihood estimator (MLE) of $\b_0$ was proved to be consistent in \cite{cosslett:83} but nothing is known about its asymptotic distribution, apart from its consistency and upper bounds for its rate of convergence. It remains an open question whether or not the profile MLE of $\b_0$ is $\sqrt{n}$-consistent. \cite{murphy:99} derived an $n^{1/3}$-rate for the profile MLE; we show that without any smoothing it is possible to construct  estimates, based on the MLE for the distribution function $F$ for fixed $\b$, that converge at $\sqrt{n}$-rate to the true parameter. We note, however, that the estimator we propose, based on the nonparametric MLE for $F$ for fixed $\b$, is {\it not} the profile MLE for $\b_0$. The estimator is a kind of hybrid estimator, which is based on the argmax MLE for $F$ for  fixed $\b$, but defined as the zero of a non-smooth score function as a function of $\b$. So we have the remarkable situation that finding the estimate $\hat\beta_n$ as the root of a score equation based on the MLEs $\hat F_{n,\b}$, can be proved to give $\sqrt{n}$-consistent estimates of $\b_0$, in contrast with the argmax approach, using profile likelihood, for which we even still do not know whether it is $\sqrt{n}$-consistent. We go somewhat deeper into this matter in the discussion section of this paper.

A general theoretical framework for semi-parametric models when the criterion function is not smooth is developed in \cite{chen:2003}. The proposed theory is less suited for our score approach since the authors assume existence of a uniform consistent estimator for the infinite dimensional regression parameter with convergence rate not depending on the finite dimensional regression parameter of interest. In the current status linear regression model we have to estimate $\b_0$ and $F_0$ simultaneously, as a consequence the convergence rate of the estimator for $F_0$ depends on the convergence rate of the estimator for $\b_0$, the parameters $\b_0$ and $F_0$ are bundled and therefore we cannot apply their theory.

\cite{murphy:99} considers efficient estimation for the current status model with a 1-dimensional regression parameter $\b$ via a penalized maximum likelihood estimator under the conditions that $F_0$ and $u\mapsto E_{\b}(X|T-\b X=u)$ are three times continuously differentiable and that the data only provide information about a part of the distribution function $F_0$, where $F_0$ stays away from zero and 1.
\cite{zhang:98} proposes an estimation equation for $\b$, derived from an inequality on the conditional covariance between $X$ and $\dd$ conditional on $T - \b'X$, and use a U-statistics representation, involving summation over many indices.
\cite{Shen2000} considered an estimator based on a random sieved likelihood, but the expression for the efficient information (based on the generalized Fisher information) in this paper seems to be different from what we and the authors mentioned above obtain for this expression.

Approaches to $\sqrt{n}$-consistent and efficient estimation of the regression parameters in the binary choice model were considered by \cite{klein_spady:93}  and \cite{cosslett:07} among others. For a derivation of the efficient information $\tilde\ell_{\b_0,F_0}^2$, defined by 
\begin{align}
\label{efficient_score}
& \tilde\ell_{\b,F}(t,x,\d)= \left\{\E(X|T-\b'X=t-\b' x)-x\right\}f(t-\b' x)\nonumber\\
&\qquad\qquad\qquad\qquad\qquad\qquad\qquad\cdot\left\{\frac{\d}{F(t-\b' x)}-\frac{1-\d}{1-F(t-\b' x)}\right\},
\end{align}
where we assume $f(t-\b' x)>0$, we refer to \cite{cosslett:87}  for the binary choice model, and next to \cite{Huang:93} and \cite{murphy:99} for the current status regression model.

As mentioned above, the condition that the support of the density of $T-\beta_0'X$ is strictly contained in an interval $D$ for all $\beta$ and that $F_0$ stays strictly away from $0$ and $1$ on $D$ is used in \cite{murphy:99}. This condition is also used in \cite{Huang:93} and \cite{Shen2000}. The drawback of the assumption is that we have no information about the whole distribution $F_0$. It also goes against the usual conditions made for the current status model, where one commonly assumes that the observations provide information over the whole range of the distribution one wants to estimate. We presume that this assumption is made for getting the Donsker properties to work and to avoid truncation devices that can prevent the problems arising if this condition is not made, such as unbounded score functions and ensuing numerical difficulties. Examples of truncation methods can be found in \cite{cosslett:07} and \cite{klein_spady:93} among others where the authors consider truncation sequences that converge to zero with increasing sample size. We show that it is possible to estimate the finite dimensional regression parameter $\b_0$ at $\sqrt n-$rate based on a fixed truncated subsample of the data where the truncation area is determined by the quantiles of the infinite dimensional nuisance parameter estimator. 

The paper is organized as follows. The model, its corresponding log likelihood and a truncated version of the log likelihood are introduced in Section \ref{section:model}. In this section, we also discuss the advantages of a score approach over the maximum likelihood characterization. The behavior of the MLE for the distribution function $F_0$ in case $\b$ is not equal to $\b_0$ is studied in Section \ref{section:MLE}. We first construct in Section \ref{section:estimates}, based on a score equation, a $\sqrt{n}-$consistent but inefficient estimate of the regression parameter based on the MLE of $F_0$ and show how an estimate of the density, based on the MLE, can be used to extend the estimate of the regression parameter to an estimate with an asymptotic variance that is arbitrarily close to the information lower bound.

Next, we give the asymptotic behavior of a plug-in estimator which is obtained by a score equation derived from the truncated log likelihood in case a second order kernel estimate for the distribution function $F_0$ is considered. We show that the latter estimator is $\sqrt{n}$-consistent and asymptotically normal with an asymptotic variance that is arbitrarily (determined by the truncation device) close to the information lower bound, just like the estimator based on the MLE we discussed in the preceding paragraph.

The estimation of an intercept term, that originates from the mean of the error distribution, is outlined in Section \ref{section:intercept}. Section \ref{section:computation} contains details on the computation of the estimates together with the results of our simulation study; a bootstrap method for selecting a bandwidth parameter is also given. A discussion of our results is given in Section \ref{section:discussion}. Section \ref{section:Appendix} contains the derivation of the efficient information given in (\ref{efficient_score}). The proofs of the results given in this paper are worked out in the supplemental article \cite{GroeneboomHendrickx-supplement}.

\section{Model description}
\label{section:model}
Let $(T_i, X_i, \dd_i), i = 1, \ldots,n $ be independent and identically distributed observations from $(T, X,\dd)=(T, X,1_{\{Y \le T\}})$. We assume that $Y$ is modeled as
\begin{align}
\label{model}
Y=\b_0' X  + \e,
\end{align}
where $\b_0$ is a $k$-dimensional regression parameter in the parameter space $\Theta$ and $\e$ is an unobserved random error, independent of $(T,X)$ with unknown distribution function $F_0$. We assume that the distribution of $(T,X)$ does not depend on $(\b_0, F_0)$ which implies that the relevant part of the log likelihood for estimating $(\b_0, F_0)$ is given by: 
\begin{align} 
\label{log_likelihood3}
l_n(\b,F)&=\sum_{i=1}^n\left[\dd_i\log F(T_i-\b' X_i)+(1-\dd_i)\log\{1-F(T_i- \b'X_i)\}\right] \nonumber
\\
&=\int\left[\d\log F(t-\b'x)+(1-\d)\log\{1-F(t-\b'x)\}\right]\,d\P_n(t,x,\d),
\end{align}
where $\P_n$ is the empirical distribution  of the $(T_i,X_i,\dd_i)$. We will denote the probability measure of $(T,X, \dd)$ by $P_0$.  
We define the truncated log likelihood  $l_n^{(\ee)}(\b, F)$ by
\begin{equation}
\label{truncated_likelihood}
\int_{F(t-\b' x)\in[\ee,1-\ee]}\left[\d\log F(t-\b'x)+(1-\d)\log\{1-F(t-\b'x)\}\right]\,d\P_n(t,x,\d),
\end{equation} 
where $\ee\in(0,1/2)$ is a truncation parameter. 
Analogously, let,
\begin{align}
\label{first_score_function}
\psi_n^{(\ee)}(\b,F) =\int_{F(t-\b' x)\in[\ee,1-\ee]} \phi(t,x,\d)\{\d - F(t-\b'x) \}\,d\P_n(t,x,\d),
\end{align}
define the truncated score function for some weight function $\phi$.  In this paper, we consider estimates of $\b_0$, derived by the idea of solving a score equation $\psi_n^{(\ee)}(\b,\hat F_\b)= 0$ where $\hat F_\b$ is an estimate of $F$ for fixed $\b$. A motivation of the score approach is outlined below,  we have three reasons for using the score function characterization instead of the argmax approach for the estimation of $\b_0$.
 \begin{enumerate}
 	\item[(i)] Our simulation experiments indicate that, even if the profile MLE would be $\sqrt{n}$-consistent, its variance is clearly bigger than the other estimates we propose. 
 	\item[(ii)] The characterization of $\hat\b_n$ as the solution of a score equation
 	\begin{align*}
 	\psi_n\left(\b,\hat F_{n,\b}\right)=0,
 	\end{align*}
 	(see, e.g., (\ref{first_score_function})), where $\hat F_{n,\b}$ is the MLE for fixed $\b$ maximizing the log likelihood defined in (\ref{log_likelihood3}) over all $F \in {\cal{F}}$ = \{$F:\R\to [0,1]: F$ is a distributuion function\}, gives us freedom in choosing the function $\psi_n$ of which we try to find the root $\hat\b_n$. Smoothing techniques can be used but are not necessary to obtain $\sqrt{n}$-convergence of the estimate.
 	
 	In this paper, we first choose a function $\psi_n$, which produces a $\sqrt{n}$-consistent and asymptotically normal estimate of $\b_0$, and does not need any smoothing device. Just like the Han maximum correlation estimate, this estimate does not attain the efficiency bound, although the difference between its asymptotic variance and the efficient asymptotic variance is rather small in our experiments. More details are given in Section \ref{section:computation}.
 	
 	Next we choose a function $\psi_n$ which gives (only depending on our truncation device) an asymptotic variance which is arbitrarily close to the efficient asymptotic variance. In this case we need an estimate of the density of the error distribution and are forced to use smoothing in the definition of $\psi_n$. The estimate, although efficient in the sense we use this concept in our paper, is not necessarily better in small samples, though.
 	\item[(iii)] The ``canonical'' approach to proofs that argmax estimates of $\b_0$ are $\sqrt{n}$-consistent has been provided by \cite{sherman:93}. His Theorem 1 says that $\|\hat\b_n-\b_0\|=O_p(n^{-1/2})$, where $\|\cdot\|$ denotes the Euclidean norm, if $\hat\b_n$ is the maximizer of $\Gamma_n(\b)$, with population equivalent $\Gamma(\b)$ and
 	\begin{itemize}
 		\item[(a)] there exists a neighborhood $N$ of $\b_0$ and a constant $k>0$ such that
 		$$
 		\Gamma(\b)-\Gamma(\b_0)\le-k\|\b-\b_0\|^2,
 		$$
	 	for $\b\in N$, and
 		\item[(b)] uniformly over $o_p(1)$ neighborhoods of $\b_0$,
 		\begin{align*}
 		&\Gamma_n(\b)-\Gamma_n(\b_0)\\
 		&=\Gamma(\b)-\Gamma(\b_0)+O_p\left(\|\b-\b_0\|/\sqrt{n}\right)+o_p\left(\|\b-\b_0\|^2\right)+O_p\left(n^{-1}\right).
 		\end{align*}
 	\end{itemize} 
If we try to apply this to the profile MLE $\hat\b_n$, it is not clear that an expansion of this type will hold. We seem to get inevitably an extra term of order $O_p(n^{-2/3})$ in (b), which does not fit into this framework. On the other hand, in the expansion of our score function $\psi_n$, we get that this function is in first order the sum of a term of the form
 	$$
 	\psi'(\b_0)(\b-\b_0)
 	$$
where $\psi'$ is the matrix, representing the total derivative of the population equivalent score function $\psi$, and a term $W_n$ of order $O_p(n^{-1/2})$, which gives:
 	$$
 	\hat\b_n-\b_0\sim -\psi'(\b_0)^{-1}W_n=O_p(n^{-1/2}),
 	$$
and here extra terms of order $O_p(n^{-2/3})$ do not hurt. The technical details are elaborated in the proofs of our main result given in the supplemental article \cite{GroeneboomHendrickx-supplement}.
 \end{enumerate}

Before we formulate our estimates, we first describe in Section \ref{section:MLE} the behavior of the MLE $\hat F_{n,\b}$ for fixed $\b$. Throughout the paper, we illustrate our estimates by a simple simulated data example; we consider the model $Y_i= 0.5 X_i+\e_i$, where the $X_i$ and $T_i$ are independent Uniform$(0,2)$ and where the $\e_i$ are independent random variables with density $f(u)=384(u-0.375)(0.625-u)1_{[0.375,0.625]}(u)$ and independent of the $X_i$ and $T_i$. Note that the expectation of the random error $E(\e) = 0.5$, our linear model contains an intercept, $\E(Y_i| X_i= x_i) = 0.5 + 0.5 x_i$.

\begin{remark}
{\rm
We chose the present model as a simple example of a model for which the (generalized) Fisher information is finite. This Fisher information  easily gets infinite. For example, if $F_0$ is the uniform distribution on $[0,1]$ and $X$ and $T$ (independently) also have uniform distributions on $[0,1]$ and $\b=1/2$, the Fisher information for estimating $\b$ is given by:
\begin{align*}
\int_{u=0}^{1/2} \frac{(x-1/2)^2}{u(1-u)}\,dx\,du+\int_{u=1/2}^1 \frac{\{x-(1-u)\}^2}{u(1-u)}\,dx\,du=\infty.
\end{align*}
We observed in simulations with the uniform distribution that $n$ times the variance of our estimates (using $\ee = 0$) steadily decreases with increasing sample size $n$, suggesting a faster than $\sqrt n$-convergence for the estimate in this model. The theoretical framework for estimation of models with infinite Fisher information falls beyond the scope of this paper.
So we chose a model where the ratio $f_0(x)^2/\{F_0(x)\{1-F_0(x)\}$ stays bounded near the boundary of its support by taking a rescaled version of the density $6x(1-x)1_{[0,1]}(x)$ for $f_0$. Note that, if the Fisher information is infinite, our theory still makes sense for the truncated version:
\begin{align*}
&\int_{F_0(u)\in[\ee,1-\ee]}\int_{x=0}^1 \frac{(x-\E\{X|T-X/2=u\})^2f_0(u)^2}{F_0(u)\{1-F_0(u)\}}f_{X|T-X/2}(x|u)f_{T-X/2}(u)\,dx\,du,
\end{align*}
corresponding to our truncation of the log likelihood and the score function in the sequel. For completeness we included the derivation of the Fisher information in the Appendix. These calculations provide more insight in the information loss when one moves from a parametric model where $F_0$ is known to our semi-parametric model with unknown $F_0$.
}
\end{remark}

\section{Behavior of the maximum likelihood estimator}
\label{section:MLE}
For fixed $\b$, the MLE $\hat F_{n, \b}$ of $l_n(\b,F)$ is a piecewise constant function with jumps at a subset of $\{T_i - \b'X_i: i= 1, \ldots, n\}$. Once we have fixed the parameter $\b$, the order statistics on which the MLE is based are the order statistics of the values $U_1^{(\b)}=T_1-\b' X_1,\dots,U_n^{(\b)}=T_n-\b' X_n$ and the values of the corresponding $\dd_i^{(\b)}$. The MLE can be characterized as the left derivative of the convex minorant of a cumulative sum diagram consisting of the points $(0,0)$ and
\begin{align}
\label{cusum}
 \left(i, \sum_{j=1}^{i}\dd_{(j)}^{(\b)} \right) \qquad i= 1,\ldots, n,
\end{align}
where $\dd_{(i)}^{(\b)}$ corresponds to the $i$th order statistic of the $T_i - \b'X_i$ (see e.g. Proposition 1.2 in \cite{GrWe:92} on p.\, 41). We have:
\begin{align*}
&\P\left\{\dd_i^{(\b)}=1\bigm|U_i^{(\b)}=u\right\}=\int F_0(u+(\b-\b_0)'x)f_{X|T-\b'X}(x|T-\b'X=u)\,dx.
\end{align*}
Hence, defining
\begin{align}
\label{def_F_beta}
F_{\b}(u)=\int F_0(u+(\b-\b_0)'x)f_{X|T-\b'X}(x|T-\b'X=u)\,dx,
\end{align}
we can consider the $\dd_i^{(\b)}$ as coming from a sample in the ordinary current status model, where the observations are of the form $\bigl(U_i^{(\b)},\dd_i^{(\b)}\bigr)$, and where the observation times have density $f_{T-\b' X}$ and where $\dd_i^{(\b)}=1$ with probability $F_{\b}(U_i^{(\b)})$ at observation $U_i^{(\b)}$.

\begin{remark}	{\rm
		Assume that  $T$ and $X$ are continuous random variable, then we can write
		\begin{align*}
		F_{\b}'(u)&=\int f_0(u+(\b-\b_0)'x)f_{X|T-\b'X}(x|u)\,dx\\
		&\qquad\qquad+\int F_0(u+(\b-\b_0)'x)\frac{\partial}{\partial u}f_{X|T-\b'X}(x|u)\,dx.
		\end{align*}
		Integration by parts on the second term yields
		\begin{align*}
		&\int F_0(u+(\b-\b_0)'x)\frac{\partial}{\partial u}f_{X|T-\b'X}(x|u)\,dx\\
		&\qquad\qquad\qquad =-(\b-\b_0)'\int f_0(u+(\b-\b_0)'x)\frac{\partial}{\partial u}F_{X|T-\b'X}(x|u)\,dx.
		\end{align*}
		This implies
		\begin{align*}
		F_{\b}'(u)	&=\int f_0(u+(\b-\b_0)'x)\Bigl\{f_{X|T-\b'X}(x|u)\\
		&\qquad\qquad\qquad\qquad\qquad\qquad-(\b-\b_0)'\frac{\partial}{\partial u}F_{X|T-\b'X}(x|u)\Bigr\}\,dx.
		\end{align*}
		Assuming that $u\mapsto f_{X|T-\b'X}(x|u)$ stays away from zero on the support of $f_0$, this implies by a continuity argument that $F_{\b}$ is monotone increasing on the support of $F'_{\b}$ for $\b$ close to $\b_0$.
		
		Also note that we get from the fact that $F_0$ is a distribution function with compact support:
		$$
		\lim_{u\to-\infty}F_{\b}(u)=0\qquad\text{ and }\qquad \lim_{u\to\infty}F_{\b}(u)=1.
		$$
		So we may assume that $F_{\b}$ is a distribution function for $\b$ close to $\b_0$. If $X$ is discrete, a similar argument can be used to show that $F_{\b}$ is a distribution function for $\b$ close to $\b_0$ under the assumption that $u\mapsto P(X=x|T-\b'X=u)$ stays away from zero on the support of $f_0$.
	}
\end{remark}

A picture of the MLE $\hat F_{n,\b}$, based on the values $T_i-\b X_i$, and the corresponding function $F_{\b}$ for the model used in our simulation experiment, is shown in Figure \ref{fig:F_{beta}} and compared with $F_0$. Note that $F_{\b}$ involves both a location shift and a change in shape of $F_0$.

\begin{figure}[!ht]
	
	\begin{subfigure}[b]{0.43\textwidth}
		\centering
		\includegraphics[width=\textwidth]{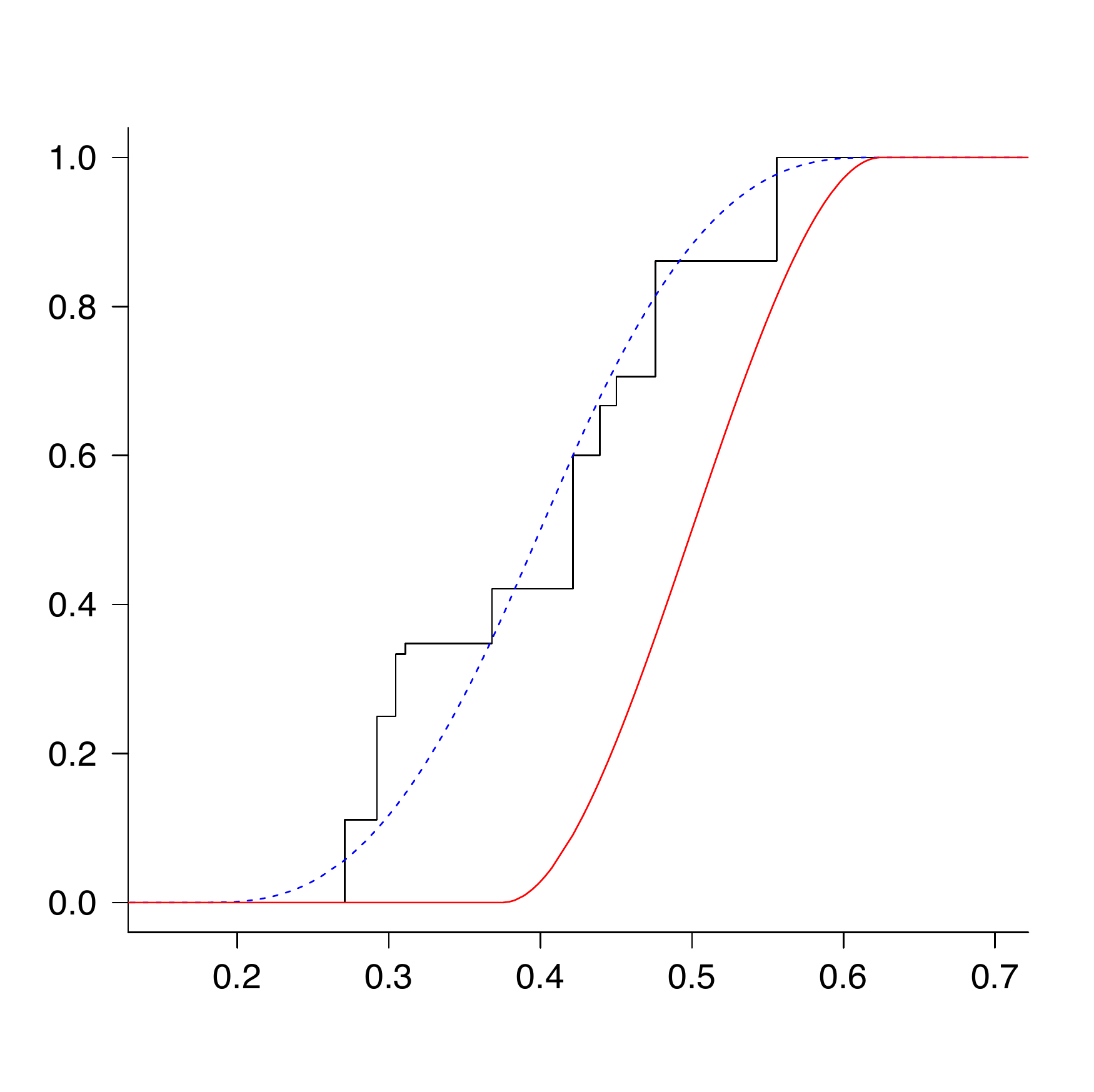}
		\caption{}
	\end{subfigure}
	\hspace{0.5cm}
	\begin{subfigure}[b]{0.43\textwidth}
		\includegraphics[width=\textwidth]{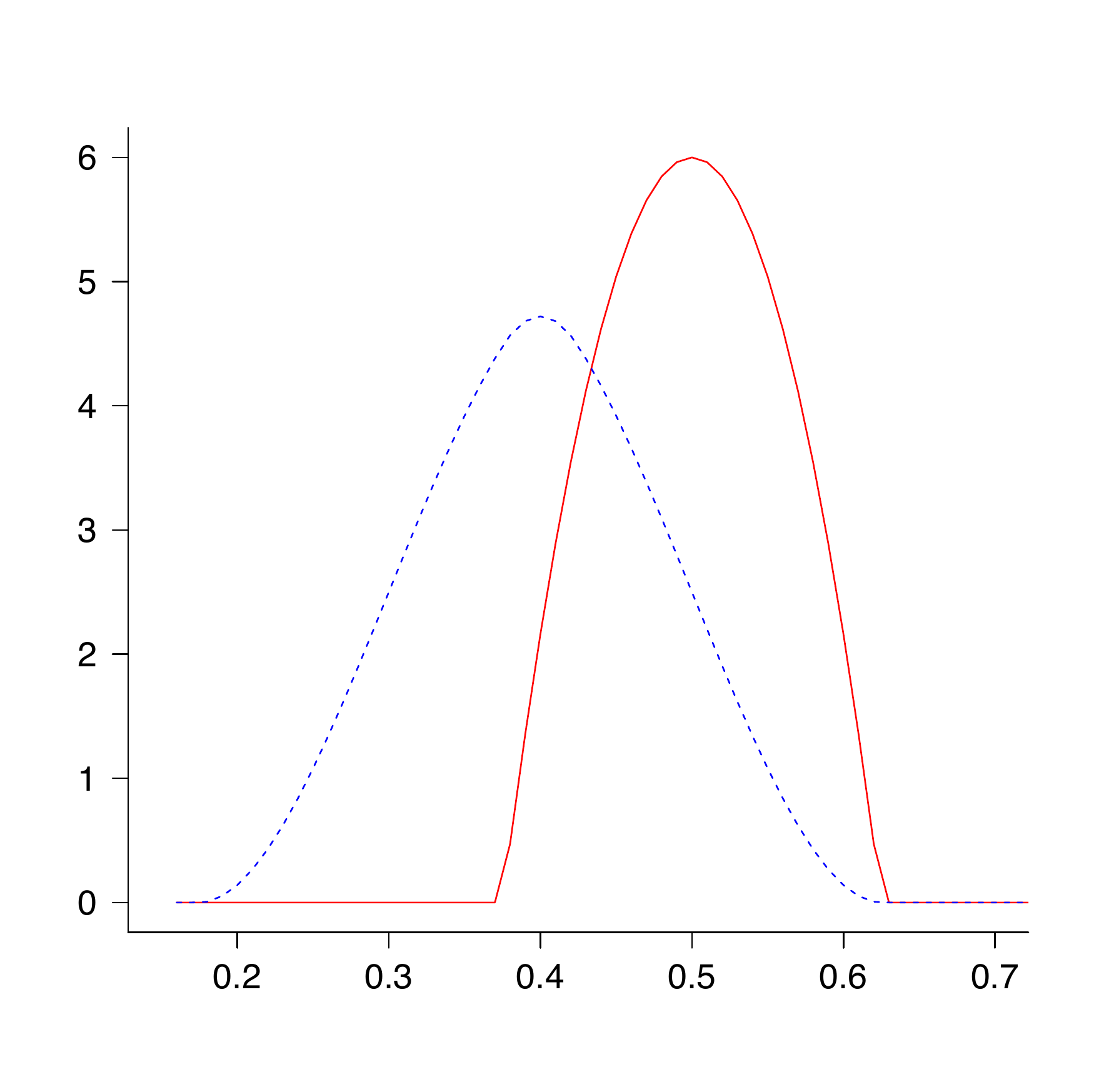}
		\caption{}
	\end{subfigure}
	\caption{ The real $F_0$ (red, solid), the function $F_{\beta}$ for $\b=0.6$ (blue, dashed) and the MLE $\hat F_{n,\b}$ (step function),  for a sample of size $n=1000$. (b) The real $f_0$ (red, solid) and the function $F_\b'$ for $\b=0.6$ (blue, dashed). }
	\label{fig:F_{beta}}
\end{figure}

For fixed $\b$ in a neighborhood of $\b_0$ we can now use standard theory for the MLE from current status theory. The following assumptions are used.

\begin{itemize}
\item[(A1)]  The parameter $\b_0=(\b_{0,1},\dots,\b_{0,k})\in\R^k$ is an interior point of $\Theta$ and the parameter space $\Theta$ is a compact convex set.
\item [(A2)]  $F_{\b}$ has a strictly positive continuous derivative, which stays away from zero on $A_{\ee',\b}\stackrel{\text{\rm def}}=\{u:F_{\b}(u)\in[\ee',1-\ee']\}$ for all $\b \in \Theta$, where $\ee'\in(0,\ee)$.
\item[(A3)] The density $u\mapsto f_{T-\b'X}(u)$ is continuous and also staying away from zero on $A_{\ee',\b}$ for all $\b \in \Theta$, where $A_{\ee',\b}$ is defined as in (A2).
\end{itemize}

\begin{remark}
{\rm Note that the truncation is for the interval $[\ee,1-\ee]$, but that we need conditions (A2) and (A3) to be satisfied for the slightly bigger interval $[\ee',1-\ee']$.
}
\end{remark}

\begin{lemma}
	\label{lemma:MLE_misspecified}
	If Assumptions (A1), (A2) and (A3) hold, then:
	\begin{enumerate}
\item[(i)]
\begin{align*}
\sup_{\b\in\Theta}\int\left\{\hat F_{n,\b}(t-\b'x)-F_{\b}(t-\b'x)\right\}^2\,dG(t,x)=O_p\left(n^{-2/3}\right).		
\end{align*}
\item[(ii)] $$
\P\left(\lim_{n\to \infty} \sup_{\b \in \Theta,\,u \in A_{\ee',\b}} |\hat F_{n,\b}(u)- F_\b(u)| =0 \right)=1,
$$
where $\ee'$ is chosen as in condition (A2).
\end{enumerate}
\end{lemma}

\begin{proof}
Part (i) is proved in the supplementary material. Using the continuity and monotonicity of $F_\b$ the second result follows from (i).
\end{proof}

We first show in Section \ref{subsection:MLE} that it is possible to construct $\sqrt{n}$-consistent estimates of $\b_0$ derived from a score function $\psi_n^{(\ee)}(\b,\hat F_{n,\b})$ without requiring any smoothing in the estimation process. In Section \ref{subsection:MLE2} we look at $\sqrt{n}$-consistent and efficient score-estimates based on the MLE $\hat F_{n,\b}$, using a weight function $\phi$ that incorporates the estimate $\int K_h(u-y)\,d\hat F_{n,\b}(y)$ of the density $f_0(u)=F_0'(u)$. An efficient estimate of $\b_0$ derived by a score function based on kernel estimates for the distribution function, is considered in Section \ref{subsection:plugin}. The latter estimate does not involve the behavior of the MLE $\hat F_{n,\b}$.
\section{$\sqrt{n}$-consistent estimation of the regression parameter}
\label{section:estimates}

\subsection{A simple estimate based on the MLE $\hat F_{n,\b}$, avoiding any smoothing}
\label{subsection:MLE}
We consider the function $\psi_{1,n}^{(\ee)}$, defined by:
\begin{align}
\label{score1}
\psi_{1,n}^{(\ee)}(\b)&\stackrel{\mbox{\rm\small def}}{=}\int_{\hat F_{n,\b}(t-\b'x)\in[\ee,1-\ee]} x\bigl\{\d - \hat F_{n,\b}(t-\b'x)\bigr\}\,d\P_n(t,x,\d),
\end{align}
where $\hat F_{n,\b}$ is the MLE based on the order statistics of the values $T_i-\b' X_i, i = 1,\ldots,n$. The function $\psi_{1,n}^{(\ee)}$ has finitely many different values, and we look for a ``crossing of zero'' to define our estimate. For simplicity we first treat the one-dimensional case.
\begin{figure}[!ht]
	\begin{subfigure}[b]{0.43\textwidth}
		\centering
		\includegraphics[width=\textwidth]{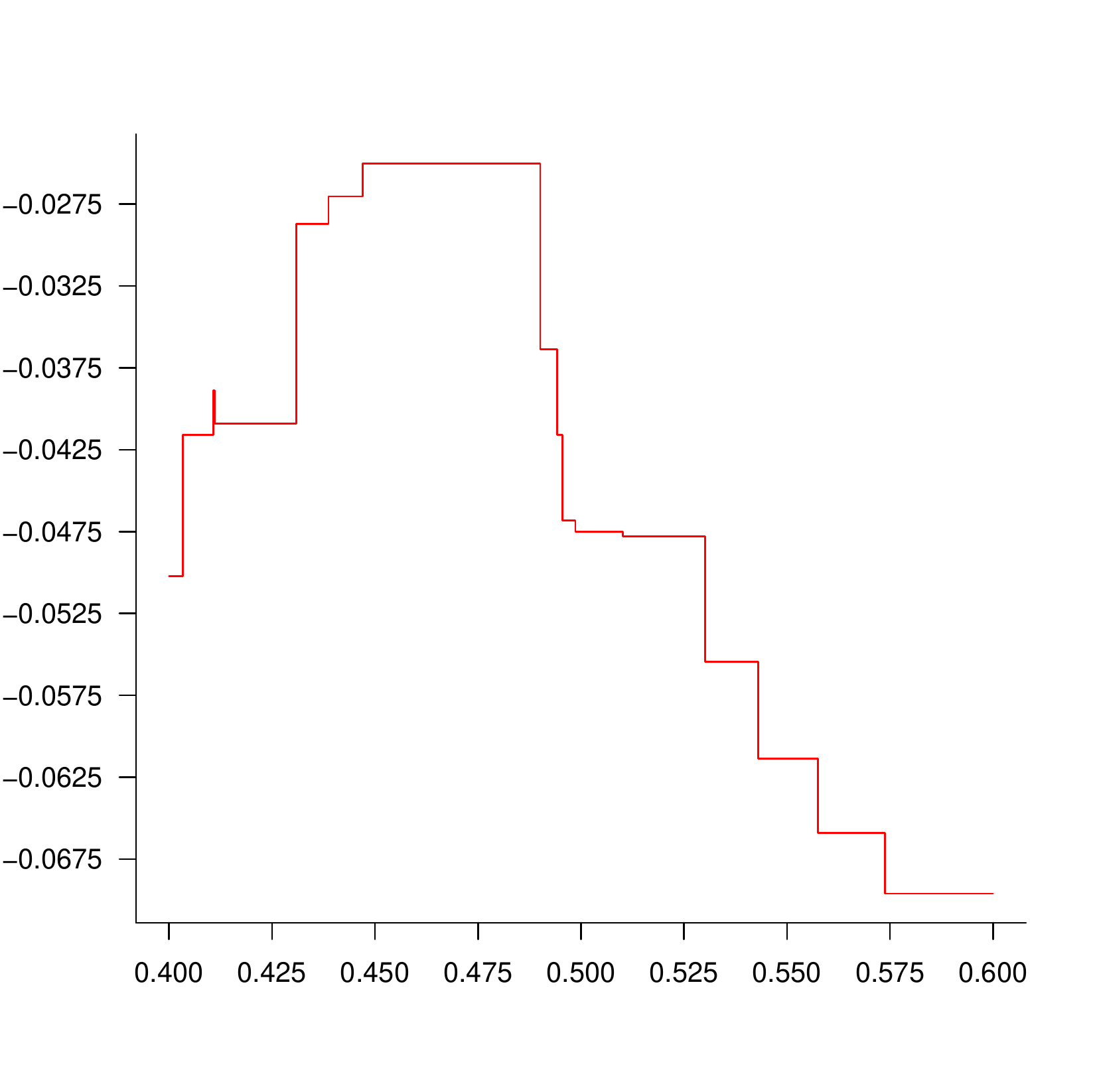}
	\end{subfigure}
	\hspace{0.5cm}
	\begin{subfigure}[b]{0.43\textwidth}
		\includegraphics[width=\textwidth]{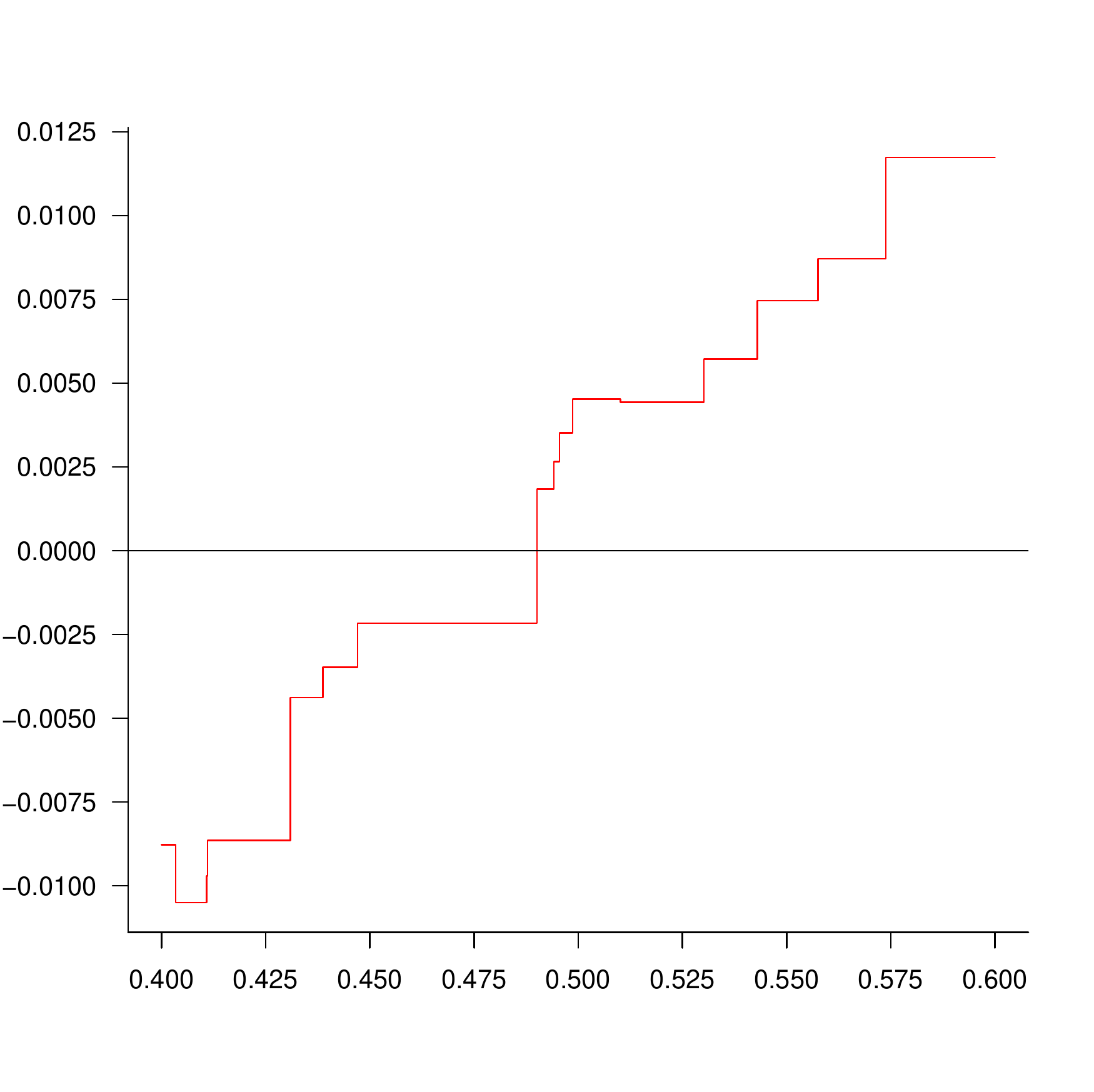}
	\end{subfigure}
	\caption{The truncated profile log likelihood $l_n^{(\ee)}$ for the MLE $\hat F_{n,\b}$ (left panel) and the score function $\psi_{1,n}^{(\ee)}$ (right panel) as a function of $\b$ for a sample of size $n=100$ and $\ee=0.001$.}
	\label{fig:MLE1}
\end{figure}

Figure \ref{fig:MLE1} gives a picture of the function $\psi_{1,n}^{(\ee)}$ as a function of $\b$, drawn as a right-continuous piecewise constant function on a grid of 100 points. Note that this function can have at most $n!$ different values, for all permutations of the numbers $1,\dots,n$.
We would like to define the estimate $\hat \b_n$ by:
\begin{align}
\label{estimator1a}
\psi_{1,n}^{(\ee)}(\hat \b_n) = 0,
\end{align}
but it is clear that we cannot hope to achieve that due to the discontinuous nature of the score function $\psi_{1,n}^{(\ee)}$.  
We therefore introduce the following definition:
\begin{definition}[zero-crossing]
\label{zero-crossing}
We say that $\b_*$ is a crossing of zero of a function $C: \Theta \to \R: \b \mapsto C(\b)$ if each open neighborhood of $\b_*$ contains points $\b_{1}, \b_{2} \in \Theta$ such that $ C(\b_{1})C(\b_{2}) \le 0$. 

We say that a function $\tilde C: \Theta \to \R^d: \b \mapsto \tilde C(\b) = (\tilde C_1(\b),\ldots \tilde C_d(\b))'$ has a crossing of zero at a point $\b_*$ if $\b_*$ is a crossing of zero of each component $ \tilde C_j: \Theta \to \R, j=1\ldots,d$.

\end{definition} 

We define our estimator $\hat \b_n$ as a crossing of zero  of $\psi_{1,n}^{(\ee)}$. Figure \ref{fig:MLE1} shows a crossing of zero at a point $\beta$ close to $\beta_0=0.5$. If the number of dimensions $d$ exceeds one, then a crossing of zero can be thought of as a point $\b_* \in \Theta$ such that each component of the score function $\psi_{1,n}^{(\ee)}$ passes through zero in $\b = \b_*$. Before we state the asymptotic result of our estimator in Theorem \ref{theorem:method1}, we give in Lemma \ref{lemma:population_score} below some interesting properties of the population version of the score function.

\begin{lemma}
	\label{lemma:population_score}
	Let $\psi_{1,\ee}$ be defined by,
	\begin{align}
	\label{score_population}
	\psi_{1,\ee}(\b) &= \int_{F_\b(t-\b'x) \in [\ee,1-\ee]} x \{\d - F_\b(t-\b'x)\}\,dP_0(t,x,\d),
	\end{align}
	 and define
	$$\E_{\ee,\b}(w(T,X,\dd)) = \E\left(\,1_{\{F_\b(t-\b'x) \in[\ee,1-\ee]\}} w(T,X,\dd)\right),$$
	 for functions $w$, then $\psi_{1,\ee}(\b_0) = 0$ and for each $\b \in \Theta$ we have
	\begin{align*}
	&(i) \qquad \psi_{1,\ee}(\b) = \E_{\ee,\b}\left[ \text{\rm Cov}\left( \dd, X | T-\b'X \right)\right] \\ 
	& (ii) \qquad (\b -\b_0)'\E_{\ee,\b}\left[\text{\rm Cov}\left( \dd, X | T-\b'X \right) \right]\ge 0 \text{ for all } \b \in \Theta, 
	\end{align*}
	and  $\b_0$ is the only value such that (ii) holds.
	The derivative of $\psi_{1,\ee}$ at $\b = \b_0$ is given by,
	\begin{align}
	&\psi_{1,\ee}'(\b_0) = \E_{\ee,\b_0}\left[f_0(T-\b_0'X) \text{\rm Cov}(X| T-\b_0'X)\right]. \label{derivative_score}
	\end{align}
\end{lemma}
The proof of Lemma \ref{lemma:population_score} is given in the supplementary material. An illustration of the second result (ii) is given in Figure \ref{fig:cov-condition}, this property is used in the proof of consistency of our estimator $\hat \b_n$ given in the supplementary material, as the first part of Theorem \ref{theorem:method1}.

\begin{figure}[!ht]
	\includegraphics[width=0.5\textwidth]{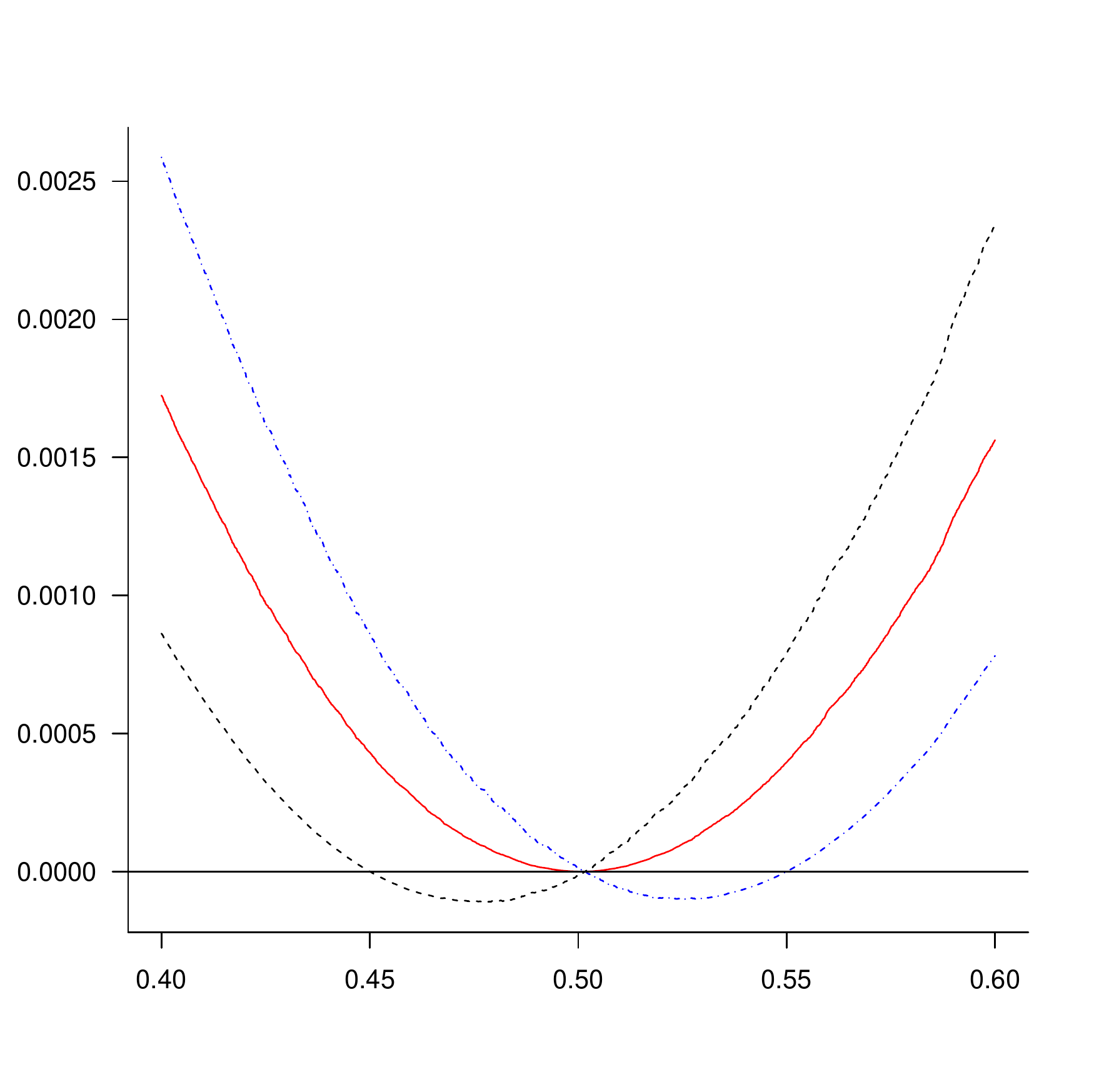}
	\caption{The function $\b \mapsto (\b -\b_*)' \int_{ F_{\b}(t-\b'x)\in[\ee,1-\ee]} x\bigl\{\d-F_{\b}(t-\b'x)\bigr\}\,dP_0(t,x,\d)$ as a function of $\b$, with $\b_* = 0.45$ (black,dashed), $\b_* = \b_0 = 0.50$ (red,solid) and $\b_* = 0.55$ (blue,dashed-dotted) for $\ee=0.001$.}
	\label{fig:cov-condition}
\end{figure}

The following assumptions are also needed for the asymptotic normality results of our estimators.

\begin{itemize}
	\item[(A4)]  The function $F_\b$ is twice continuously differentiable on the interior of the support $S_\b$ of $f_\b =F'_\b$ for $\b \in \Theta$.
	\item[(A5)] The density $f_{T-\b'X}(u)$ of $T-\b'X$ and the conditional expectations $\E\{X|T-\b'X=u\}$ and $\E\{XX'|T-\b'X=u\}$ are twice continuously differentiable functions w.r.t.\ $u$, except possibly at a finite number of points. The functions $\b\mapsto f_{T-\b'X}(v)$, $\b\mapsto \E\{X|T-\b'X=v\}$ and $\b\mapsto \E\{XX'|T-\b'X=v\}$ are continuous functions, for $v$ in the definition domain of the functions and for $\b\in \Theta$. The density of $(T,X)$ has compact support.
\end{itemize}

\begin{theorem}
	\label{theorem:method1}
	Let Assumptions (A1)-(A5) be satisfied and suppose that the covariance $\text{\rm Cov}(X,F_0(u+(\b-\b_0)'X)|T-\b'X=u)$ is not identically zero for $u$ in the region $A_{\ee,\b}$, for each $\b\in \Theta$. Moreover, let $\hat \b_n$ be defined by a crossing of zero of $\psi_{1,n}^{(\ee)}$.
	Then:
	\begin{enumerate}
	\item[(i)]
	\text{\rm[Existence of a root]} For all large $n$, a crossing of zero $\hat\b_n$ of $\psi_{1,n}^{(\ee)}$ exists with probability tending to one.
	\item[(ii)]
	\text{\rm[Consistency]}
	\begin{align*}
	\hat\b_n \stackrel{p}{\rightarrow}\b_0,\qquad n\to\infty.
	\end{align*}
	\item[(iii)]
	\text{\rm[Asymptotic normality]}	
	$\sqrt{n}\bigl\{\hat\b_n-\b_0\bigr\}$ is asymptotically normal with mean zero and variance $A^{-1}BA^{-1}$,  where
	\begin{align*}
	&A=\E_{\ee}\Bigl[ f_0(T- \b_0'X)\,\text{\rm Cov}(X|T-\b_0'X)\Bigr],
	\end{align*}
	and
	\begin{align*}
	B&=\E_{\ee}\Bigl[ F_0(T- \b_0'X)\{1- F_0(T- \b_0'X)\}\text{\rm Cov}(X|T-\b_0'X)\Bigr],
	\end{align*}
	defining $ \E_\ee(w(T,X,\dd)) = \E\,1_{\{F_0(t-\b_0'x) \in[\ee,1-\ee]\}} w(T,X,\dd)$ for functions $w$ and assuming that $A$ is non-singular.
	\end{enumerate}
\end{theorem}

\begin{remark}
{\rm Note that $\text{\rm Cov}(X,F_0(u+(\b-\b_0)'X)|T-\b'X=u)$ is not identically zero for $u$ in the region $\{u:\ee\le F_{\b}(u)\le 1-\ee\}$ if the conditional distribution of $X$, given $T-\b'X=u$, is non-degenerate for some $u$ in this region if $F_0$ is strictly increasing on $\{u:\ee\le F_{\b}(u)\le 1-\ee\}$.
}
\end{remark}

The proof of Theorem \ref{theorem:method1} is given in the supplemental article \cite{GroeneboomHendrickx-supplement}. A picture of the truncated profile log likelihood $l_n^{(\ee)}(\b, \hat F_{n,\b})$ and the score function $\psi_{1,n}^{(\ee)}(\b)$ for $\b$ ranging from 0.45 to 0.55 is shown in Figure \ref{fig:MLE1}. Note that, since the MLE $\hat F_{n,\b}$ depends on the ranks of the $T_i-\b'X_i$, both curves are piecewise constant where jumps are possible if the ordering in $T_i-\b'X_i$ changes when $\b$ changes. Due to the discontinuous nature of the profiled log likelihood and the score function, the estimators are not necessary unique. The result of Theorem \ref{theorem:method1} is valid for any $\hat \b_n$ satisfying 
Definition \ref{zero-crossing}.

\subsection{Efficient estimates involving the MLE $\hat F_{n,\b}$}
\label{subsection:MLE2}

Let $K$ be a probability density function with derivative $K'$ satisfying
\begin{itemize}
	\item[(K1)] The probability density $K$ has support [-1,1], is twice continuously differentiable and symmetric on $\R$.
\end{itemize}
Let $h>0$ be a smoothing parameter and $K_h$ respectively $K'_h$ be the scaled versions of $K$ and $K'$ respectively,  given by
\begin{align*}
K_h(\cdot) = h^{-1 }K\left(h^{-1}(\cdot)\right) \quad \text{and} \quad K'_h(\cdot) = h^{-2 }K'\left(h^{-1}(\cdot)\right).
\end{align*}
The triweight kernel is used in the simulation examples given in the remainder of the paper.
Define the density estimate,
\begin{align}
\label{density_estimate}
 f_{nh,\b}(t-\b'x) = \int K_h(t-\b'x -w)\,d\hat F_{n,\b}(w).
\end{align}
We consider
\begin{align}
\label{score2}
\psi_{2,nh}^{(\ee)}(\b)\stackrel{\text{\small def}}= \nonumber& \int_{\hat F_{n,\b}(t-\b'x)\in[\ee,1-\ee]}x f_{nh,\b}(t-\b'x)\\
&\qquad\qquad\qquad\cdot\frac{\d -\hat F_{n,\b}(t-\b'x)}{\hat F_{n,\b}(t-\b'x)\{1-\hat F_{n,\b}(t-\b'x)\}}\,d\P_n(t,x,\d),
\end{align}
and let, analogously to the first estimator defined in the previous section, $\hat \beta_n$ be the estimate of $\b_0$, defined by a zero-crossing of the score function $\psi_{2,nh}^{(\ee)}$.

\begin{theorem}
	\label{theorem:method2}
	Suppose that the conditions of Theorem \ref{theorem:method1} hold and that the function $F_\b$ is three times continuously differentiable on the interior of the support $S_\b$. Let $\hat\b_n$ be defined by a zero-crossing of $\psi_{2,nh}^{(\ee)}$. Then, as $n\to\infty$, and $h\asymp n^{-1/7}$, 
	\begin{enumerate}
	\item[(i)]
	\text{\rm[Existence of a root]} For all large $n$, a crossing of zero $\hat \b_n$ of $\psi_{2,nh}^{(\ee)}$ exists with probability tending to one.
	\item[(ii)]
	\text{\rm[Consistency]}
	\begin{align*}
	\hat\b_n \stackrel{p}{\rightarrow}\b_0,\qquad n\to\infty.
	\end{align*}
	\item[(iii)]
	\text{\rm[Asymptotic normality]}	
	$\sqrt{n}\bigl\{\hat\b_n-\b_0\bigr\}$ is asymptotically normal with mean zero and variance $I_{\ee}(\b_0)^{-1}$, where
	\begin{equation}
	\label{I(beta_0)}
	I_{\ee}(\b_0)=\E_\ee\left\{  \frac{ f_0(T- \b_0'X)^2\,\text{\rm Cov}(X|T-\b_0'X)}{F_0(T- \b_0'X)\{1-F_0(T- \b_0'X)\}} \right\},
	\end{equation}
	 which is assumed to be non-singular.
	 \end{enumerate}
\end{theorem}

A picture of the score function $\psi_{2,nh}^{(\ee)}(\b)$ is shown in Figure \ref{fig:MLE2}. Note that the range on the vertical axis is considerably larger than the range on the vertical axis of the corresponding score function $\psi_{1,n}^{(\ee)}(\b)$.

\begin{figure}[!ht]
	\begin{subfigure}[b]{0.43\textwidth}
		\centering
		\includegraphics[width=\textwidth]{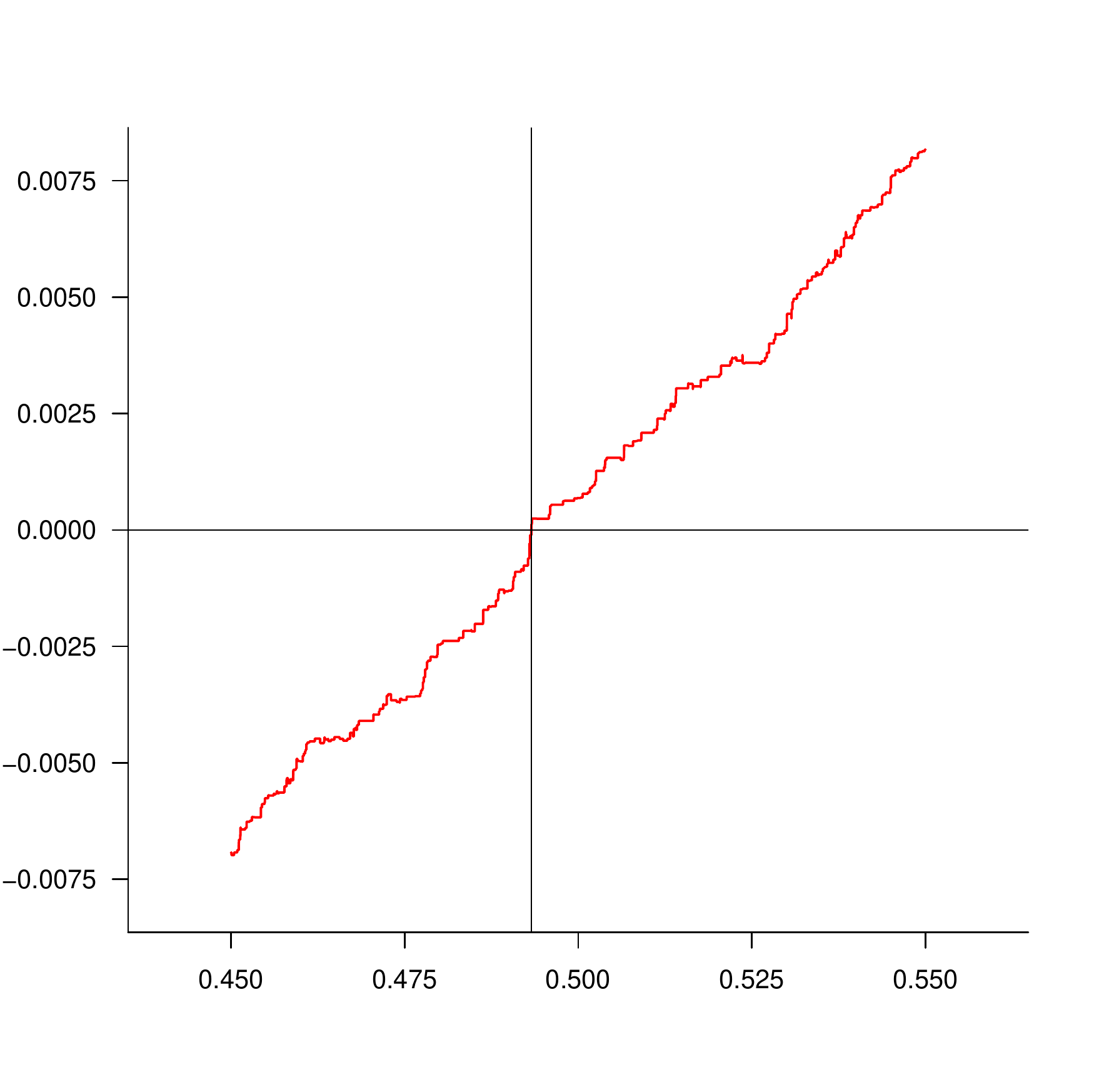}
	\end{subfigure}
	\hspace{0.5cm}
	\begin{subfigure}[b]{0.43\textwidth}
			\includegraphics[width=\textwidth]{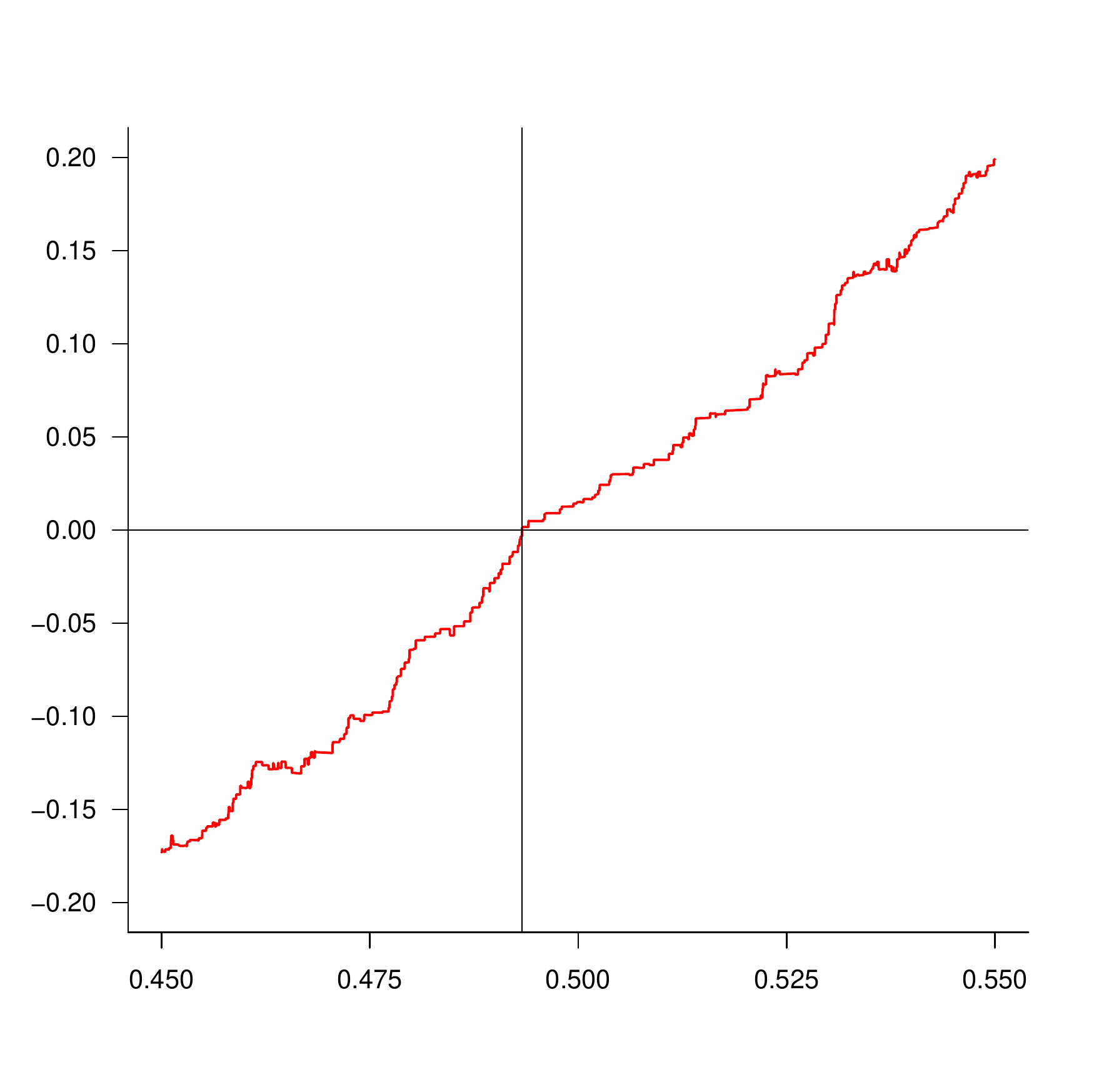}
	\end{subfigure}
	\caption{The score functions $\psi_{1,n}^{(\ee)}$ (left panel) and $\psi_{2,nh}^{(\ee)}$ (right panel) as a function of $\b$ for a sample of size $n=1000$ with $\ee=0.001$ and $h=0.5n^{-1/7}$.}
\label{fig:MLE2}
\end{figure}

\subsection{Efficient estimates not involving the MLE $\hat F_{n,\b}$ }
\label{subsection:plugin}
Define the plug-in estimate
\begin{equation}
\label{plug_in_estimate}
F_{nh,\b}(t-\b'x)=\frac{\int \d K_h(t-\b'x -u+ \b'y)\,d\P_n(u,y,\d)}{\int K_h(t-\b'x -u+ \b'y)\,d\G_n(u,y)},
\end{equation}
where $\G_n$ is the empirical distribution function of the pairs $(T_i,X_i)$ and where $K_h$ is a scaled version of a probability density function $K$, satisfying condition (K1); the probability measure of $(T,X)$ will be denoted by $G$. The plug-in estimates are not necessarily monotone but we show in Theorem \ref{th:monotonicity} that $F_{nh,\b}$ is monotone with probability tending to one as $n \rightarrow \infty$ and $\b \rightarrow \b_0.$ Another way of writing $F_{nh,\b}$ is in terms of ordinary sums. Let
\begin{align}
\label{g1}
g_{nh, 1,\b}(t-\b' x)=\frac1{n}\sum_{j=1}^n\dd_j K_h(t-\b' x -T_j+ \b' X_j),
\end{align}
and,
\begin{align}
\label{g}
g_{nh,\b}(t-\b' x)=\frac1{n}\sum_{j=1}^n K_h(t-\b' x -T_j+\b' X_j),
\end{align}
then,
$$
F_{nh,\b}(t-\b' x)=\frac{g_{nh,1,\b}(t-\b' x)}{g_{nh,\b}(t-\b' x)} = \frac{\sum_{j=1}^n \dd_jK_h(t-\b' x-T_j+\b' X_j)}{\sum_{j=1}^n K_h(t-\b' x-T_j+ \b' X_j)},
$$ 
in which we recognize the Nadaraya-Watson statistic. One could also omit the diagonal term $j=i$ in the sums above when estimating $F_{nh,\b}(T_i-\b' X_i)$ which is often done in the econometric literature (see e.g. \cite{hardle:93}). In our computer experiments however, this gave an estimate of the distribution function which had a more irregular behavior than the estimator with the diagonal term included.

If we replace $F$ in (\ref{truncated_likelihood}) by $F_{nh,\b}$, the truncated log likelihood becomes a function of $\b$ only.
Although the log likelihood has discontinuities if we consider the lower and upper boundaries $F_{nh,\b}^{-1}(\ee)$ and $F_{nh,\b}^{-1}(1-\ee)$ of the integral also as a function of $\b$, an asymptotic representation of the partial derivatives of the truncated log likelihood is given by the score function,
\begin{align}
\label{score_equation}  
&\psi_{3,nh}^{(\ee)}(\b) \stackrel{\text{\small def}}{=}\int_{F_{nh,\b}(t-\b'x)\in[\ee,1-\ee]}\partial_\b F_{nh,\b}(t-\b'x) \nonumber\\
& \qquad\qquad\qquad\qquad\cdot\frac{\d-F_{nh,\b}(t-\b'x)}{F_{nh,\b}(t-\b'x)\{1-F_{nh,\b}(t-\b'x)\}}\,d\P_n(t,x,\d),
\end{align}
where the partial derivative of the plug-in estimate $F_{nh,\b}(t-\b' x)$, given by (\ref{plug_in_estimate}), w.r.t.\ $\b$ has the following form:
\begin{align}
\label{partial_derivative}
&\partial_\b F_{nh,\b}(t-\b' x)=\frac{\int(y-x)\{\d-F_{nh,\b}(t-\b' x)\}K_h'(t-\b' x-u+\b' y)\,d\P_n(u,y,\d)}{g_{nh,\b}(t-\b' x)}\,,
\end{align}
where $g_{nh,\b}(t-\b' x)$ is defined in (\ref{g}). 
We define the plug-in estimator $\hat \b_n$ of $\b_0$ by
\begin{align}
\label{estimator3}
\psi_{3,nh}^{(\ee)}(\hat\b_n) = 0.
\end{align}
A picture of the truncated log likelihood $l_n^{(\ee)}(\b,F_{nh,\b})$ and score function $\psi_{3,nh}^{(\ee)}(\b)$ for the plug-in method is shown in Figure \ref{fig:plugin}. Since $F_{nh,\b}(t-\b'x)$ is continuous, we no longer need to introduce the concept of a zero-crossing to ensure existence of the estimator and we can work with the zero of the score function $\psi_{3,nh}^{(\ee)}(\b)$.
\begin{figure}[!ht]
	\begin{subfigure}[b]{0.43\textwidth}
		\centering
		\includegraphics[width=\textwidth]{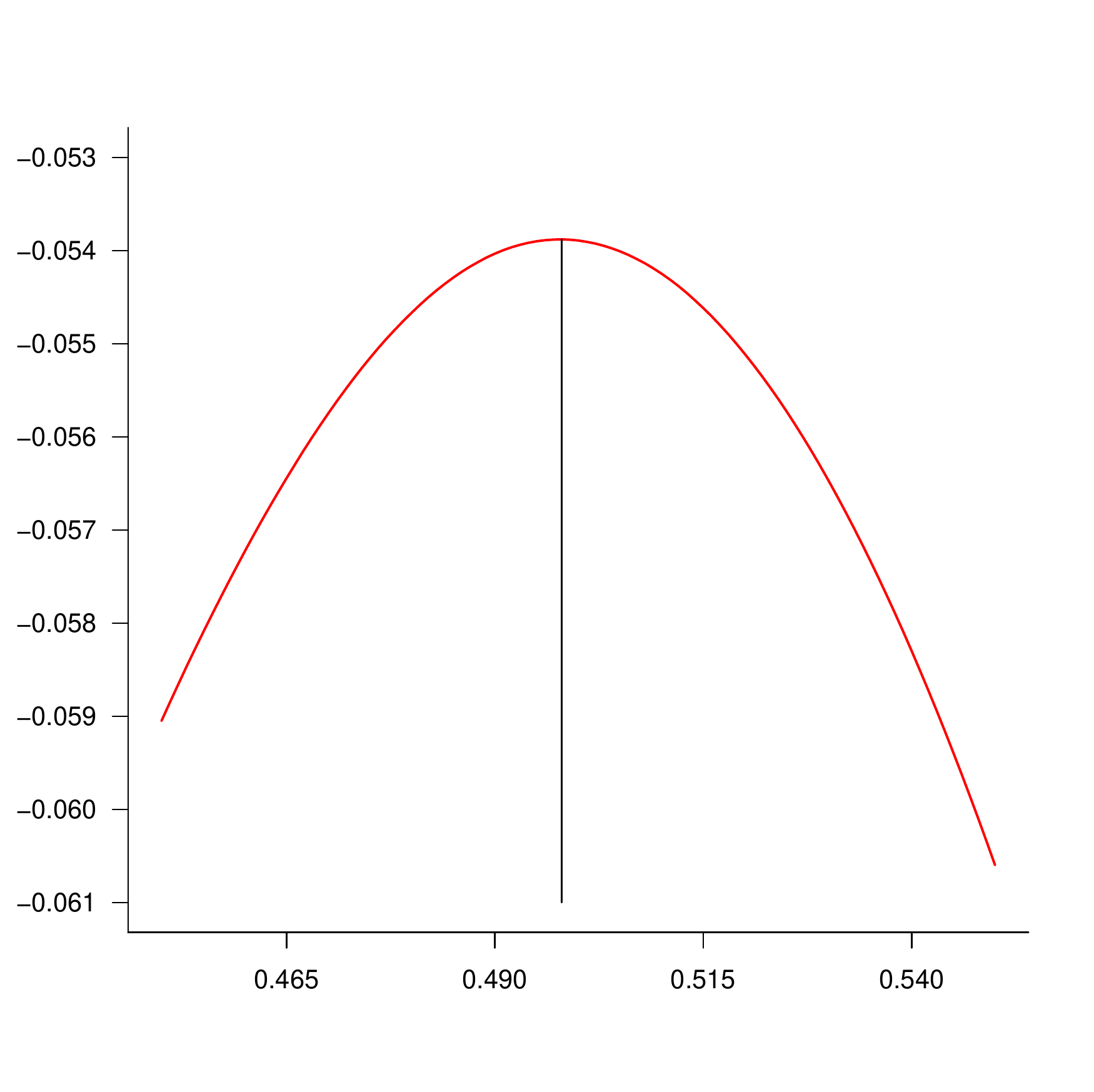}
	\end{subfigure}
	\hspace{0.5cm}
	\begin{subfigure}[b]{0.43\textwidth}
		\includegraphics[width=\textwidth]{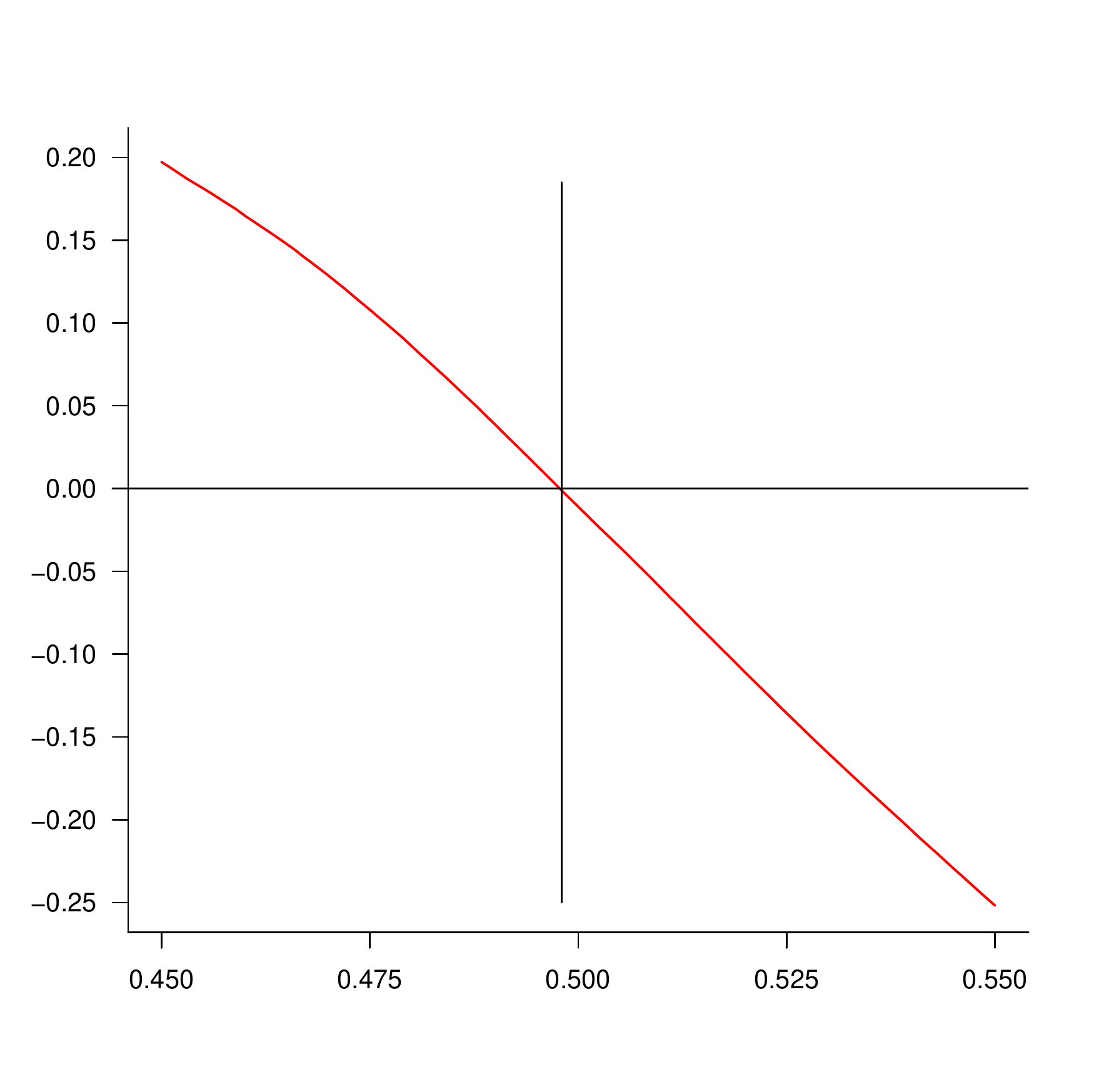}
	\end{subfigure}
	\caption{The truncated profile log likelihood $l_n^{(\ee)}$ for the plug-in $F_{nh,\b}$ (left panel) and the score function $\psi_{3,nh}^{(\ee)}$ (right panel) as a function of $\b$ for a sample of size $n=1000$ with $\ee=0.001$ and $h=0.5n^{-1/5}$.}
	\label{fig:plugin}
\end{figure}


Our main result on the plug-in estimator is given below.

\begin{theorem}
	\label{th:asymptotic_normality}
	If Assumptions (A1)-(A5) hold and 
	\begin{align}
	\label{identifiability_plugin}
	&-\left(\b-\b_0\right)' \int_{F_{\b}(t-\b'x)\in[\ee,1-\ee]}\partial_{\b}F_{\b}(t-\b'x)\frac{F_0(t-\b_0'x) -F_{\b}(t-\b'x)}{F_{\b}(t-\b'x)\{1-F_{\b}(t-\b'x)\}}dG(t,x),
	\end{align}
	is nonzero  for each $\b \in \Theta$ except for $\b = \b_0$,
	then, for $\hat\b_n$ being the plug-in estimator introduced above, as $n\to\infty$, and $h\asymp n^{-1/5}$,
	\begin{enumerate}
	\item[(i)]
	\text{\rm[Existence of a root]} For al large $n$ a point $\hat\b_n$, satisfying (\ref{estimator3}), exists with probability tending to one.
	\item[(ii)]
	\text{\rm[Consistency]}
	\begin{align*}
	\hat\b_n \stackrel{p}{\rightarrow}\b_0,\qquad n\to\infty.
	\end{align*}  
	\item[(iii)]
	\text{\rm[Asymptotic normality]}	
	$\sqrt{n}\bigl\{\hat\b_n-\b_0\bigr\}$ is asymptotically normal with mean zero and variance $I_{\ee}(\b_0)^{-1}$ where $I_{\ee}(\b_0)$, defined in (\ref{I(beta_0)}), is assumed to be non-singular.
	\end{enumerate}
\end{theorem}	

\begin{remark}
\label{remark:Hoeffding}
\rm{ 
Note that using an expansion in $\b - \b_0$, we can write $\partial_\b F_{\b}(t-\b' x)$ as, 
\begin{align*}
&\int(y-x)f_0(t- \b_0'x + (\b-\b_0)'(y-x) )f_{X|T- \b'X}(y|T-\b'X=t-\b'x)dy\\
&\qquad+\int F_0(t- \b_0' x + (\b-\b_0)'(y-x) ) \partial_\b f_{X|T-\b' X}(y|T-\b' X=t-\b' x)\,dy\\
&=f_0(t- \b'x) \E\left\{X-x|T- \b'X=t- \b'x\right\} + O( \b - \b_0)
\end{align*}	
so that the integral defined in (\ref{identifiability_plugin}) can be approximated by,
\begin{align*}
	&-\left(\b-\b_0\right)' \int_{F_{\b}(u)\in[\ee,1-\ee]}f_0(u)\E\left\{X-x|T- \b 'X=u\right\}\\
	&\qquad\qquad\qquad\qquad\qquad\cdot\frac{F_0(u+ (\b - \b_0)'x) -F_{\b}(u)}{F_{\b}(u)\{1-F_{\b}(u)\}}f_{X|T-\b'X}(x|u)\,dx\,du\\
	& = \int_{F_{\b}(u)\in[\ee,1-\ee]} \frac{f_0(u)\,\text{\rm Cov}(( \b - \b_0)'X,F_0(u+(\b-\b_0)'X)|T-\b'X=u)}{F_{\b}(u)\{1-F_{\b}(u)\}}\,du,
\end{align*}
which is positive by the monotonicity of $F_0$. (See also the discussion in \cite{zhang:98} about this covariance and the proof of Lemma \ref{lemma:population_score} given in the supplemental material.)
A crucial property of the covariance used here, showing that the covariance is nonnegative, goes back to a representation of the covariance in \cite{Hoeffding:40}, which can easily be proved by an application of Fubini's theorem:
\begin{align*}
EXY-EXEY=\int\{\P(X\ge s,Y\ge t)-\P(X\ge s)\P(Y\ge t)\}\,ds\,dt.
\end{align*}
if $XY,X$ and $Y$ are integrable. Figure \ref{fig:identifiability_plugin} shows the integral in (\ref{identifiability_plugin})
for our simulation model for $\b \in [0.45,0.55]$ and illustrates that this integral is strictly positive except for $\b = \b_0$, which is a crucial property for the proof of the consistency of the plug-in estimator given in the supplemental article \cite{GroeneboomHendrickx-supplement}.  }
\end{remark}
\begin{figure}[!ht]
	\includegraphics[width=0.5\textwidth]{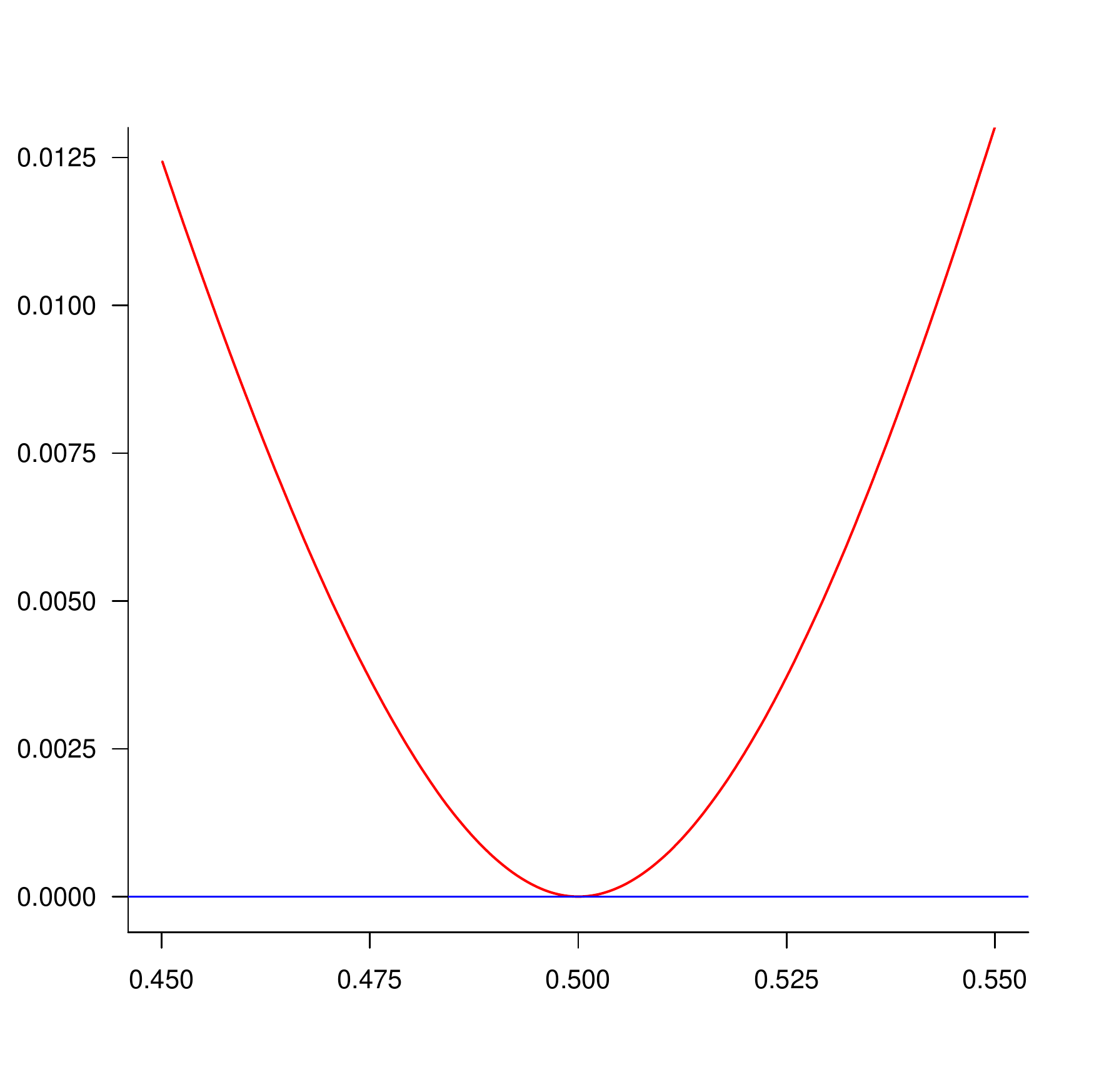}
	\caption{The integral defined in (\ref{identifiability_plugin}), as a function of $\b$, with $\ee = 0.001$.}
	\label{fig:identifiability_plugin}
\end{figure}

Section \ref{subsec:roadmap} contains a road map of the proof of Theorem \ref{th:asymptotic_normality}, the proof itself is given in the supplemental article \cite{GroeneboomHendrickx-supplement}.
We also have the following results for the plug-in estimate.
\begin{theorem} 
	\label{th:monotonicity}
	Let the conditions of Theorem \ref{th:asymptotic_normality} be satisfied, then we have on each interval $I$ contained in the support of $f_\b$ and for each $\b \in \Theta$
	\begin{align*}
	P\left\{F_{nh,\b} \text{ is monotonically increasing on } I\right\} \stackrel{p}\longrightarrow 1.
	\end{align*}
\end{theorem}
The proof of Theorem \ref{th:monotonicity} follows from the asymptotic monotonicity of the plug-in estimate in the classical current status model (without regression parameters) and is proved in the same way as Theorem 3.3 of \cite{piet_geurt_birgit:10}.

\begin{theorem}
\label{th:alternative_MLE-expansion}
Let the conditions of Theorem \ref{th:asymptotic_normality} be satisfied.  Then, for $\hat \b_n$ being the plug-in estimator of $\b_0$,
\begin{align*}
\sqrt{n}(\hat\b_n-\b_0)&= \frac{I_{\ee}(\b_0)^{-1}}{\sqrt n} \sum_{i \in J_{F_0}} f_0(T_i-\b_0'X_i)\{\E(X_i|T_i-\b_0'X_i)-X_i\}
\\
&\qquad\qquad\qquad\qquad\qquad \cdot \frac{\dd_i-F_0(T_i-\b_0' X_i)}{F_0(T_i-\b_0' X_i)\{1-F_0(T_i-\b_0' X_i)\}} +o_p(1).
\end{align*}
where 	$J_H = \{i : \ee \leq H(T_i-\b_0'X_i)\leq 1 -\ee\}$ for some function $H$.
\end{theorem}
The representation of Theorem \ref{th:alternative_MLE-expansion} plays an important role in determining the variance of smooth functionals, of which the intercept $\a=\int u\,d F_0(u)$  is an example. The proof of Theorem \ref{th:alternative_MLE-expansion} is given in the supplemental article \cite{GroeneboomHendrickx-supplement}. A similar representation holds for the estimators defined in Theorem \ref{theorem:method1} and Theorem \ref{theorem:method2} (see the proofs of Theorem \ref{theorem:method1} and \ref{theorem:method2} respectively given in the supplemental article).

\begin{remark} {\rm
		The plug-in method also suggests the use of U-statistics. By straightforward calculations, we can write the score function defined in (\ref{score_equation}) as
		\begin{align}
		\label{U-statistic}
		&\psi_{3,nh}^{(\ee)}(\b)  \nonumber\\
		&= \frac1{n^2}\sum_{i \in J_{F_{nh,\b}}}\frac{\frac{\partial}{\partial\b}F_{nh,\b}(T_i-\b' X_i)\bigl\{\dd_i-F_{nh, \b}(T_i-\b' X_i)\bigr\}}{F_{nh,\b}(T_i-\b' X_i)\{1-F_{nh,\b}(T_i-\b' X_i)\}}\nonumber
		\end{align}
		\begin{align}
		&=  \frac1{n^2}\sum_{i \in J_{F_{nh,\b}}}\sum_{j\ne i}\frac{\dd_i\dd_j(X_j-X_i) K_h'(T_i-\b' X_i -T_j+\b' X_j)}{g_{nh,1,\b}(T_i-\b' X_i)} \nonumber\\
		&\quad+\frac1{n^2}\sum_{i \in J_{F_{nh,\b}}}\sum_{j\ne i}\frac{(1-\dd_i)(1-\dd_j)(X_j-X_i) K_h'(T_i-\b' X_i -T_j+\b' X_j)}{g_{nh,0,\b}(T_i-\b' X_i)}\nonumber\\
		& \quad-\frac1{n^2}\sum_{i \in J_{F_{nh,\b}}}\sum_{j\ne i}\frac{(X_j-X_i) K_h'(T_i-\b' X_i -T_j+\b' X_j)}{g_{nh,\b}(T_i-\b' X_i)} 
		\end{align}
		where $g_{nh,0,\b}=g_{nh,\b}-g_{nh,1,\b}$, see (\ref{g1}) and (\ref{g}).
		Each of the three terms on the right-hand side of (\ref{U-statistic}) can be rewritten in terms of a scaled second order U-statistics. A proof based on U-statistics requires lengthy and tedious calculations which are avoided in the current approach for proving Theorem \ref{th:asymptotic_normality}.  The representation given in Theorem \ref{th:alternative_MLE-expansion} also indicates that the U-statistics representation does not give the most natural approach to the proof of asymptotic normality and efficiency of $\hat\b_n$. For these reasons, we do not further examine the results on U-statistics.
	}
\end{remark}
\begin{remark}
	{\rm We propose the bandwidths $h\asymp n^{-1/7}$ resp. $h\asymp n^{-1/5}$ in Theorem \ref{theorem:method2} resp. Theorem \ref{th:asymptotic_normality}, which are the usual bandwidths with ordinary second order kernels for the estimates of a density resp. distribution function. 
		Unfortunately, various advices are given in the literature on what smoothing parameters one should use.  \cite{klein_spady:93} has fourth order kernels and uses bandwidths between the orders $n^{-1/6}$ and $n^{-1/8}$ for the estimation of $F$. Note that the use of fourth order kernels needs the associated functions to have four derivatives in order to have the desired bias reduction. \cite{cosslett:07}  advises a bandwidth $h$ such that $n^{-1/5}\ll h\ll n^{-1/8}$, excluding the choice $h\asymp n^{-1/5}$. Both ranges are considerably large and exclude our bandwidth choice $h\asymp n^{-1/5}$. 	\cite{murphy:99} considers a penalized maximum likelihood estimator where the penalty parameter $\l_n$ satisfies $
		1/\l_n=O_p\left(n^{2/5}\right)$ and $\l_n^2=o_p\left(n^{-1/2}\right)$. 
		Translated into bandwidth choice (using $h_n\asymp\sqrt{\l_n}$), the conditions correspond to: $	n^{-1/5}\lesssim h\ll n^{-1/8}$, suggesting that their conditions do allow the choice $h\asymp n^{-1/5}$ for estimating the distribution function. 
	}
\end{remark}
\subsubsection{Road map of the proof  of Theorem \ref{th:asymptotic_normality}}
\label{subsec:roadmap}
The older proofs of a result of this type always used second derivative calculations. As convincingly argued in \cite{vaart:98}, proofs of this type should only use first derivatives and that is indeed what we do. The limit function $F_\b$ of the estimates for $F_0$ when $\b \ne \b_0$ plays a crucial role in our proofs. 
We first prove the consistency of the plug-in estimate $\hat \b_n$.
Next, we use a Donsker property for the functions representing the score function and prove that the integral w.r.t.\ $d\P_n$ of this score function is
$$
o_p\left(n^{-1/2}+\hat\b_n-\b_0\right),
$$
and that the integral w.r.t.\ $dP_0$ is asymptotically equivalent to
$$
-(\hat\b_n-\b_0)I_{\ee}(\b_0),
$$
where $I_{\ee}(\b_0)$ is the generalized Fisher information, given by (\ref{I(beta_0)}). Combining these results gives Theorem \ref{th:asymptotic_normality}. Very essential in this proof are $L_2$-bounds on the distance between the functions  $F_{nh,\b}$ to its limit $F_\b$ for fixed $\b$ and  on the $L_2$-distance between the first partial derivatives $\partial_\b F_{nh,\b}$ and $\partial_\b F_{\b}$. If the bandwidth $h\asymp n^{-1/5}$, the first $L_2$-distance is of order $n^{-2/5}$ and the second distance is of order $n^{-1/5}$, allowing us to use the Cauchy-Schwarz inequality on these components. Here we use a result in \cite{nickl:15} on $L_2$ bounds of derivatives of kernel density estimates.

In Section \ref{section:intercept} we discuss the estimation of an efficient estimate of the intercept term in regression model (\ref{model}) using the plug-in estimates $\hat \b_n$ and $F_{nh,\hat \b_n}$.

\subsection{Truncation}
We introduced a truncation device in order to avoid unbounded score functions and numerical difficulties. If one starts with the {\it efficient} score equation or an estimate thereof, the solution sometimes suggested in the literature, is to add a constant $c_n$, tending to zero as $n\to\infty$, to the factor $F(t-\b' x)\{1-F(t-\b' x)\}$ which inevitably will appear in the denominator. This is done in, e.g. \cite{zhang:98}; similar ideas involving a sequence $(c_n)$ are used in \cite{klein_spady:93} and \cite{cosslett:07}. 
	
In contrast with the usual approaches to truncation, which imply the selection of a suitable sequence $c_n$, we do not consider a vanishing truncation sequence but work with a subsample of the data depending on the $\ee$ and $(1-\ee)$ quantiles of the distribution function estimate for small but fixed $\ee \in (0,1/2)$. This simple device in (\ref{truncated_likelihood}) moreover implies keeping the characterizing properties of the MLE (see Proposition 1.1 on p.\,39 of \cite{GrWe:92}) which are lost when a vanishing sequence is considered. It is perhaps somewhat remarkable that we can, instead of letting $\ee\downarrow0$, fix $\ee>0$ and still have consistency of our estimators; on the other hand, the estimate proposed by \cite{murphy:99} is also identified via a subset of the support of the distribution $F_0$.

Although the truncation area depends on $\b$, we show in the supplemental article \cite{GroeneboomHendrickx-supplement} (see the proof of Theorem \ref{theorem:method1}) that the population version of the score function, given by
\begin{align}
\label{population}
\psi_\ee(\b) = \int_{F_\b(t-\b'x) \in [\ee,1-\ee]} \f(t,x,\d)  \{\d - F_\b(t-\b'x)\}dP_0(t,x,\d),
\end{align}
has a derivative at $\b = \b_0$ that only involves the derivative of the integrand in (\ref{population}), but does not involve terms arising from the truncation limits appearing in the integral. Using the truncation in the argmax maximum log likelihood approach would not lead to a derivative of the population version of the log likelihood which ignores the boundaries and therefore this truncation is less suited for argmax estimators. 

A drawback of our fixed truncation parameter approach is that we get a truncated Fisher information. The resulting estimates are therefore not efficient in the classical sense of efficiency but the difference between the efficient variance and almost (determined by the size of $\ee$) efficient variance is rather small in our simulation models. We also tried to program the fully efficient estimators proposed by \cite{zhang:98} and compared its performance to the performance of our almost-efficient estimators. The comparison showed that our estimates perform better in finite samples. Moreover, the estimates by \cite{zhang:98} involve several kernel density estimates, resulting in a very large computation time compared to our simple estimates (involving 5 double summations over the data points).

Moreover, the usual conditions in the theory of estimation of $F_0$ under current status and, more generally, interval censored data are that $F_0$ corresponds to a distribution with compact support. Otherwise, certain variances easily get infinite, and similarly, the Fisher information in our model can easily become infinite. Truncating by keeping the quantiles between $\ee$ and $1-\ee$ avoids difficulties in this case and allows us to apply the theory which presently has been developed for the current status model.

Note that the score function defined in (\ref{score1}) does not contain a factor $F(t-\b' x)$ or $1-F(t-\b' x)$ in the denominator. For simplicity of the proofs, we still impose the truncation area, since the classical results for the current status model are derived under the assumption that the density $f_0$ is bounded away from zero . We conjecture however that the result of Theorem \ref{theorem:method1} remains valid when taking $\ee = 0$.

\section{Estimation of the intercept}
\label{section:intercept}
We want to estimate the intercept
\begin{equation}
\label{def_intercept}
\a=\int u\,d F_0(u).
\end{equation}
We can take the plug-in estimate $\hat\b_n$ of $\b_0$, by using a bandwidth of order $n^{-1/5}$ and the score procedure, as before.
However, in estimating $\a$, as defined by (\ref{def_intercept}), we have to estimate $F_0$ with a smaller bandwidth $h$, satisfying $h\ll n^{-1/4}$ to avoid bias, for example $h\asymp n^{-1/3}$. The matter is discussed in \cite{cosslett:07}, p.\ 1253.

We have the following result of which the proof can be found in the supplemental article \cite{GroeneboomHendrickx-supplement}.

\begin{theorem}
	\label{th:intercept}
	Let the conditions of Theorem \ref{th:asymptotic_normality} be satisfied, and let $\hat\b_n$ be the $k$-dimensional estimate of $\b_0$ as obtained by the score procedure, described in Theorem \ref{th:asymptotic_normality}, using a bandwidth of order $n^{-1/5}$. Let $ F_{nh,\hat\b_n}$ be a plug-in estimate of $F_0$, using $\hat\b_n$ as the estimate of $\b_0$, but using a bandwidth $h$ of order $n^{-1/3}$ instead of $n^{-1/5}$. Finally, let $\hat\a_n$ be the estimate of $\a$, defined by
	$$
	\int u\,dF_{nh,\hat\b_n}(u).
	$$
	Then $\sqrt{n}(\hat\a_n-\a)$ is asymptotically normal, with expectation zero and variance
	\begin{align}
	\label{intercept_variance}
	\s^2\stackrel{\text{\small def}}=a(\b_0)'I_{\ee}(\b_0)^{-1}\,a(\b_0)+\int\frac{F_0(v)\{1-F_0(v)\}}{f_{T-\b_0'X}(v)}\,dv,
	\end{align}
	where $a(\b_0)$ is the $k$-dimensional vector, defined by
	\begin{align*}
	\label{covar_vector1}
	a(\b_0)=\int \E\{X|T-\b_0'X=u\}f_0(u)\,du,
	\end{align*}
	and $I_{\ee}(\b_0)$ is defined in (\ref{I(beta_0)}).
\end{theorem}

\begin{remark}
	{\rm We choose the bandwidth of order $n^{-1/3}$ for specificity, but other choices are also possible. We can in fact choose $n^{-1/2}\ll h\ll n^{-1/4}$. The bandwidth of order $n^{-1/3}$ corresponds to the automatic bandwidth choice of the MLE of $F_0$, also using the estimate $\hat\b_n$ of $\b_0$.
	}
\end{remark}

\begin{remark}
	{\rm
		Note that the variance corresponds to the information lower bound for smooth functionals in the binary choice model, given in \cite{cosslett:07}. The second part of the expression for the variance on the right-hand side of (\ref{intercept_variance}) is familiar from current status theory,  see e.g. (10.7), p.\ 287 of \cite{piet_geurt:14}.
	}
\end{remark}
Instead of considering the plug-in estimate, we could also consider the estimates described in Theorem \ref{theorem:method1} and Theorem \ref{theorem:method2}. After having determined an estimate $\hat\b_n$ in this way, we next estimate $\a$ by
\begin{equation}
\label{hat_alpha}
\hat\a_n=\int x\,d\hat F_{n,\hat\b_n}(x),
\end{equation}
where $\hat F_{n,\hat\b_n}$ is the MLE corresponding to the estimate $\hat\b_n$. The theoretical justification of this approach can be proved using the asymptotic theory of smooth functionals given in \cite{piet_geurt:14}, p.286. Using the MLE $\hat F_{n,\hat \b_n}$ instead of the plug-in  $F_{nh,\hat \b_n}$ as an estimate of the distribution function $F_0$, avoids the selection of a bandwidth parameter for the intercept estimate. We discuss in the next section how the bandwidth can be selected by the practitioner in a real data sample.

\section{Computation and simulation}
\label{section:computation}

The computation of our estimates is relatively straightforward in all cases. 
For the score-estimates defined in Sections \ref{subsection:MLE} and \ref{subsection:MLE2},  we first compute the MLE for fixed $\b$ by the so-called ``pool adjacent violators'' algorithm for computing the convex minorant of the ``cusum diagram'' defined in (\ref{cusum}). If the MLE has been computed for fixed $\b$, we can compute the density estimate $f_{nh}$.   The estimate of $\b_0$ is then determined by a root-finding algorithm such as Brent's method. Computation is very fast.  For the plug-in estimate, we simply compute the estimate $F_{nh,\b}$ as a ratio of two kernel estimators for fixed $\b$ and then compute the derivative w.r.t. $\b$. Next we use again a root-finding algorithm to determine the zero of the corresponding score function.

Some results from the simulations of our model are available in Table \ref{table:beta_estimation}, which contains the mean value of the estimate, averaged over $N = 10,000$ iterations, and $n$ times the variance of the estimate of $\b_0 = 0.5$ (resp. $\alpha_0 = 0.5$)  for the different methods described above, as well as for the classical MLE of $\b_0$,  for different sample sizes $n$ and a truncation parameters $\ee=0.001$. 
We took the bandwidth $h=0.5n^{-1/7}$ for the efficient score-estimate of Section \ref{subsection:MLE2}. The bandwidth $h=0.5n^{-1/5}$ for the plug-in estimate of Section \ref{subsection:plugin} was chosen based on an investigation of the mean squared error (MSE) for different choices of $c$ in $h=cn^{-1/5}$. Details on how to choose the bandwidth in practice are given in Section \ref{subsec: bandwidth}.  The true asymptotic values for the variance of $\sqrt{n}(\hat\b_n -\b_0)$ in our simulation model, obtained via the inverse of the Fisher information $I_{\ee}(\b_0)$, are 0.151707 without truncation and 0.158699 for $\ee=0.001$ and 0.17596 for $\ee=0.01$.  We advise to use a truncation parameter $\ee$ of 0.001 or smaller in practice.  The variance defined in Theorem \ref{theorem:method1} for $\ee = 0.001$ is 0.193612. The lower bounds for the variance of the intercept are 0.257898 for the simple score method and 0.222984 for the efficient methods. Our results show slow convergence to these bounds.

Table \ref{table:beta_estimation} shows that the efficient score and the efficient plug-in methods perform reasonably well. A drawback of the plug-in method however is the long computing time for large sample sizes, whereas the computation for the MLE is fast even for the larger samples. Note  moreover that the plug-in estimate is only asymptotically monotone whereas the MLE is monotone by definition. All our proposed estimates perform better than the classical MLE, the log likelihood for the MLE has moreover a rough behavior, with a larger chance that optimization algorithms might calculate a local maximizer instead of the global maximizer. 

The performance of the score estimates is worse than the performance of the plug-in estimates for small sample sizes but increases considerably when the sample size increases. Although the asymptotic variance of the first score-estimator of Section \ref{subsection:MLE} is larger than the (almost, determined by the truncation parameter $\ee$) efficient variance, the results obtained with this method are noteworthy seen the fact that no smoothing is involved in this simple estimation technique.

Table \ref{table:beta_estimation} does not provide strong evidence of the $\sqrt{n}$-consistency of the classical MLE, but we conjecture that the MLE is indeed $\sqrt{n}$-consistent but not efficient. 
Considering the drawbacks of the classical MLE, we advise the use of the plug-in estimate for small sample sizes and the use of the score estimates, based on the MLE, for larger sample sizes, for estimating the parameter $\b_0$. We finally suggest to estimate the parameter $\a_0$ via the MLE corresponding to this $\b_0$-estimate, avoiding in this way the bias problem for the kernel estimates of $\a_0$.

\begin{table}[!ht]
	\centering
	\caption{\rm The mean value of the estimate and $n$ times the variance of the estimates of $\b_0$ and $\a_0$ for different methods, $h_{\b} = 0.5n^{-1/7}$  (for the efficient score method), $h_{\b}=0.5n^{-1/5}$ and $h_{\a} =0.75n^{-1/3}$  (for the plug-in method), $\ee= 0.001$ and $N = 10000$.}
	\label{table:beta_estimation}
	\vspace{0.5cm}
	\scalebox{0.93}{
	\begin{tabular}{|ll|cc|cc| cc| cc|}
		\hline
		&& \multicolumn{2}{|c|}{Score-1} & \multicolumn{2}{c|}{Score-2}  & \multicolumn{2}{c|}{ Plugin}  &\multicolumn{2}{c|}{MLE}\\
		\cline{3-10}
 		 &$n$  & mean & $n\times$var  & mean &$n\times$var  & mean & $n\times$var & mean & $n\times$var \\
		\hline
		$\b$ &$100$ & 0.500212 & 0.364558 & 0.502247 & 0.410449 &  0.499562& 0.245172  & 0.489690 & 0.307961\\
		 & 500  & 0.499845 & 0.221484 & 0.499825 & 0.230178 & 0.498857 & 0.191857& 0.499315 & 0.228335\\
		 & 1000 & 0.499982 & 0.211608 & 0.500353 & 0.208102 & 0.499502 & 0.192223& 0.499937 & 0.228420\\
	     & 5000 & 0.499901 & 0.195294 & 0.499964 & 0.184807 & 0.500314 & 0.181421 & 0.499933 & 0.239898\\
		 & 10000 & 0.499988& 0.191115 & 0.499985 & 0.172758 & 0.500120 & 0.172043& 0.499994 & 0.227222\\
		 & 20000 & 0.500038 &  0.187616 & 0.500023 & 0.169762 & 0.500096 & 0.174197& 0.499952 & 0.238400\\
		\hline
		$\a$ &$100$ & 0.511937 & 0.468415& 0.509679 & 0.515638 & 0.495709 & 0.332949& 0.523103 & 0.425614 \\
		& 500  & 0.502258 & 0.293585 & 0.502506 & 0.287576 & 0.498932 & 0.254040& 0.502514 & 0.304540\\
		& 1000  & 0.500839 & 0.284958 & 0.500616 & 0.262684 & 0.498385 & 0.270085& 0.500937 & 0.300201\\
		& 5000  & 0.500345 & 0.262566 & 0.500316 & 0.244892 & 0.501597 & 0.241294 & 0.500270 & 0.303754\\
		& 10000 & 0.500127 & 0.256983 & 0.500134 & 0.232973 & 0.501680 & 0.245993 & 0.500076 & 0.289905\\
		& 20000 & 0.500020 & 0.250720 & 0.500042 & 0.230901 & 0.501660 & 0.244042& 0.500101 & 0.302824\\
		\hline
	\end{tabular}}
\end{table}

\subsection{Bandwidth selection}
\label{subsec: bandwidth}
In this section we discuss the bandwidth selection for the plug-in estimate. A similar idea can be used for the selection of the bandwidth used for the second estimate defined in Section \ref{subsection:MLE2}. We define the optimal constant $c_{opt}$ in $h=cn^{-1/5}$ as the minimizer of  $MSE$,
\begin{align*}
c_{opt} = \arg \min_c MSE(c) = \arg \min_c E_{\b_0}(\hat\b_{n,h_c} - \b_0)^2 ,
\end{align*}
where $\hat\b_{n,h_c}$ is the estimate obtained when the constant $c$ is chosen in the estimation method. A picture of the Monte Carlo estimate of $MSE$ as a function of $c$ is shown for the plug-in method in Figure \ref{fig:mse_plugin}, where we estimated $MSE(c)$ on a grid $c =$ 0.01, 0.05, 0.10, $\cdots$, 0.95, for a sample size $n=1000$ and truncation parameter $\ee = 0.001$ by a Monte Carlo experiment with $N = 1000$ simulation runs,
\begin{align}
\label{Monte_carlo_MSE}
\widehat{MSE(c)} =N^{-1}\sum_{j=1}^N ( \hat\beta_{n,h_c}^{j} - \beta_0)^2,
\end{align}
where $\hat\beta_{n,h_c}^{j}$ is the estimate of $\b_0$ in the $j$-th simulation run, $j= 1,\ldots,N$. 

Since $F_0$ and $\b_0$ are unknown in practice, we cannot compute the actual $MSE$. We use the bootstrap method proposed by \cite{SenXu2015} to obtain an estimate of $MSE$. Our proposed estimate $F_{nh,\b}$ of the distribution function $F_0$ satisfy the conditions of Theorem 3 in \cite{SenXu2015} and the consistency of the bootstrap is guaranteed. Note that it follows from \cite{kosorok:08} and \cite{sen_mouli_woodroofe:10} that naive bootstrapping, by resampling with replacement $(T_i,X_i,\dd_i)$, or by generating bootstrap samples from the MLE, is inconsistent for reproducing the distribution of the MLE. 

The method works as follows. We let $ h_0=c_0n^{-1/5}$ be an initial choice of the bandwidth and calculate the plug-in estimates $\hat\b_{n,h_0}$ and $ F_{n,h_{c_0}}$ based on the original sample $(T_i, X_i,\dd_i), i = 1,\ldots,n$. We generate a bootstrap sample  $(X_i, T_i, \dd_i^*),i = 1,\ldots,n$ where the $(T_i,X_i)$ correspond to the $(T_i,X_i)$ in the original sample and where the indicator $\dd_i^*$ is generated from a Bernoulli distribution with probability $ F_{n,h_0}(T_i -\hat\b_{n,h_0}X_i)$, and next estimate  $\hat\b_{n,h_c}^*$ from this bootstrap sample. We repeat this $B$ times and estimate $MSE(c)$ by,
\begin{align}
\label{bootstrap_mse}
\widehat{MSE_B(c)} =B^{-1}\sum_{b=1}^B ( \hat\beta_{n,h_c}^{*b} - \hat\beta_{n,h_{c_0}})^2,
\end{align}
where $\hat\beta_{n,h_c}^{*b}$ is the bootstrap estimate in the $b$-th bootstrap run. The optimal bandwidth $\hat h_{opt} = \hat c_{opt}n^{-1/5}$ where $\hat c_{opt}$ is defined as the minimizer of $\widehat{MSE_B(c)}$.

To analyze the behavior of the bootstrap method, we compared the Monte Carlo estimate of $MSE$, defined in (\ref{Monte_carlo_MSE}), (based on $N= 1000$ samples of size $n =1000$)   to the bootstrap $MSE$ defined in (\ref{bootstrap_mse}) (based on a single sample of size $n =1000$) in Figure \ref{fig:mse_plugin}. The figure shows that the Monte Carlo $MSE$ and the bootstrap $MSE$ are in line, which illustrates the consistency of the method. The choice of the initial bandwidth does effect the size of the estimated $MSE$ but not the behavior of the estimate and we conclude that this bootstrap algorithm can be used to select an optimal bandwidth parameter in the described method above.

\begin{figure}[!ht]
	\centering
	\includegraphics[width=0.5\textwidth]{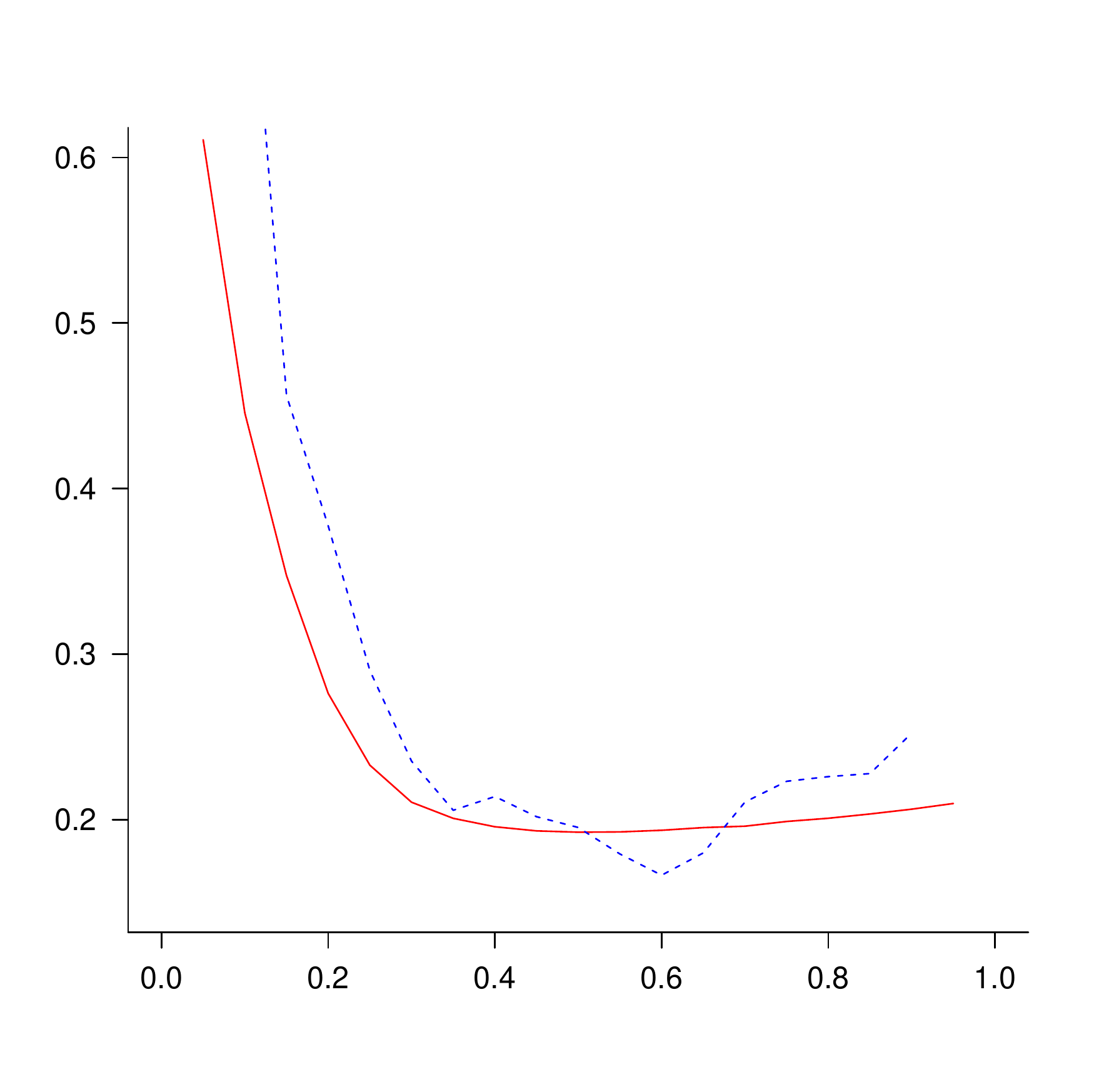}
	\caption{Estimated $MSE(c)$ plot of $\hat\b_n$ obtained from $N= 1000$ Monte Carlo simulations (red, solid) versus the bootstrap $MSE$ for $c_0=0.25$ (blue, dashed)  with $B = 10000$, $n=1000$ and $\ee=0.001$.}
	\label{fig:mse_plugin}
\end{figure}

\section{Discussion}
\label{section:discussion}
In this paper we propose a simple $\sqrt n$-consistent estimate for the finite dimensional regression parameter in the semi-parametric current status linear regression model with unspecified error distribution. The estimate has an asymptotically normal limiting distribution but does not attain the efficiency bounds. We also consider two different methods to obtain $\sqrt n$-consistent and asymptotic normal estimates with an  asymptotic variance that is arbitrarily close to the efficient variance. The first approach uses the MLE for the distribution function $F$ for fixed $\b$, the second approach does not depend on the behavior of this MLE but uses a kernel estimate for the distribution function. 
All proposed estimates are defined as a root of a score function as a function of $\b$.

 We introduced a truncation device to avoid theoretical and numerical difficulties caused by unbounded score functions. The truncation is carried out by considering a subsample of the data depending on the $\ee$ and $1-\ee$ quantiles of the distribution function estimate. Although our estimates do not attain the efficiency bound, the proposed method allows for easy computation of the estimates without the need for selecting a truncation sequence converging to zero. Achieving efficiency at the cost of additional computational complexities associated with smoothing procedures and truncation sequence selection results in only a small asymptotic efficiency gain and does not seem to improve the performance of our simple methods.

The estimates based on the efficient score function depending on the MLE for $F_0$ for fixed $\b$ have a slightly better performance than the estimates based on the smooth score function depending on the plug-in estimates for $F_0$ when the sample size is large. For small samples none of the MLE-based estimates comes out as uniformly best.

\section{Appendix}
\label{section:Appendix}
In this section, we include the derivation of the efficient information bound for the current status linear regression model.  The proofs of the results given in Section \ref{subsection:MLE}, Section \ref{subsection:MLE2} and Section \ref{subsection:plugin} are deferred to the supplemental article \cite{GroeneboomHendrickx-supplement}.

\subsection{Efficient information in the current status linear regression model}
\label{section:bounds}
The density of one observation in the current status linear regression model is,
\begin{align*}
p_{\b,F}(t,x,\d) = F(t-\b'x)^{\d}\left\{1-F(t-\b'x) \right\}^{1-\d}f_{T,X}(t,x).
\end{align*}
We assume that the distribution of $(T,X)$ does not depend on $(\b, F)$ which implies that the relevant part of the log likelihood is given by: 
\begin{align} 
l_n(\b,F)&=\sum_{i=1}^n\left[\dd_i\log F(T_i-\b' X_i)+(1-\dd_i)\log\{1-F(T_i- \b'X_i)\}\right]. \nonumber
\end{align}
If the distribution $F$ is known (parametric case), the information for $\b$ is given by,
\begin{align*}
&I_P(\b) = E\left(\left(\frac{\partial}{\partial \b} \log p_{\b,F}(T,X,\dd)\right)'\left(\frac{\partial}{\partial \b} \log p_{\b,F}(T,X,\dd)\right) \right). 
\end{align*}
Straightforward calculations yield,
\begin{align*}
&I_P(\b)_{ij}=\int \frac{\E(X_iX_j|T-\b'X=u)}{F(u)\{1-F(u)\}}f(u)^2f_{T-\b'X}(u)\,du,
\end{align*} 
where $f=F'$ and where $f_{T-\b'X}$ is the density of $T-\b'X$. When $F$ is unknown, we need to calculate the efficient score function.
Let  $F$ and $P_0$ be the probability measures of $\varepsilon$ and $(T,X,\dd)$ respectively and let $L_2^0(Q)$ be the Hilbert space of square integrable functions $a$ with respect to the measure $dQ$ satisfying $\int a dQ = 0$. The score operator $l_{F}: L_2^0(F) \mapsto L_2^0(P_0)$ is defined by,
\begin{align*}
[l_{F}a](t,x,\d) &= E\bigl(a(\varepsilon) | (T,X,\dd) =(t,x,\d)\bigr) \\
&=  \frac{\d \int_{-\infty}^{t-\b'x}a(s)dF(s)}{F(t-\b'x)} - \frac{(1-\d) \int_{-\infty}^{t-\b'x}a(s)dF(s)}{1-F(t-\b'x)}, 
\end{align*}
with adjoint,
\begin{align*}
[l_F^*b](e) = E\left( b(T,X,\dd)| \varepsilon= e\right). 
\end{align*}
The information for $\b$ in the semi-parametric model is defined by,
\begin{align*}
I(\b) &= E\left(\tilde\ell_{\b,F}(T,X,\dd)' \tilde\ell_{\b,F}(T,X,\dd)\right),
\end{align*} 
where $\tilde\ell_{\b,F}(t,x,\d)$ is the efficient score function defined by,
\begin{align*}
\tilde\ell_{\b,F}(t,x,\d) = \ell_\b(t,x,\d) - [\ell_{F}a_*](t,x,\d),
\end{align*}
where 
\begin{align*}
\ell_\b(t,x,\d) = \frac{\partial}{\partial \b} \log p_{\b,F}(t,x,\d) = \frac{-\d x f(t-\b'x)}{F(t-\b'x)} +	 \frac{(1-\d) x f(t-\b'x)}{1-F(t-\b'x)},
\end{align*}
and $\ell_{F}a_*$ satisfies,
\begin{align}
\label{condition-efficient-a}
\ell_{F}^* \ell_{F}a_* = \ell_{F}^* \ell_\b.
\end{align}
The efficient score $\tilde\ell_{\b,F}$ can be interpreted as the residual of $\ell_\b$ projected in the space spanned by $\ell_{F} a$ for $a \in L_2^0(F)$. Note that, as a consequence of  (\ref{condition-efficient-a}), the efficient information is
\begin{align*}
I(\b) &= E\left(\tilde\ell_{\b,F}(T,X,\dd)' \ell_{\b}(T,X,\dd)\right).
\end{align*} 
To find $a_*$, we have to solve (\ref{condition-efficient-a}),
\begin{align}
\label{fi_condition}
\ell_{F}^* \ell_{F}a_*(e)& = \int _e^{\infty}\frac{\f(u)}{F(u)}f_{T-\b X}(u)\,du-\int_{-\infty}^e \frac{\f(u)}{1-F(u)}f_{T-\b X}(u)\,du \nonumber\\
&=-\int _e^{+\infty}\frac{\E(X|T-\b X=u)f(u)}{1-F(u)}f_{T-\b X}(u)\,du\nonumber \\
&\qquad+\int_{-\infty}^e\frac{\E(X|T-\b X=u)f(u)}{1-F(u)}f_{T-\b X}(u)\,du\nonumber\\
&= \ell_{F}^* \ell_\b(e),
\end{align}
where $\f(t) = \int_{-\infty}^{t}a(s)dF(s)$. Equation (\ref{fi_condition}) is satisfied with 
\begin{align*}
\f(u) = -\E(X|T-\b X=u)f(u).
\end{align*}
Any $a_*$ that satisfies the above equation satisfies (\ref{condition-efficient-a}) and we get,
\begin{align*}
&\tilde\ell_{\b,F}(t,x,\d) = \\
&\quad \left\{E(X|T-\b'X=t-\b' x)-x\right\}f(t-\b' x)\left\{\frac{\d}{F(t-\b' x)}-\frac{1-\d}{1-F(t-\b' x)}\right\},
\end{align*}
and,
\begin{equation}
I(\b)_{ij}=\int \frac{\text{\rm Cov}(X_i,X_j|T-\b'X=u)}{F(u)\{1-F(u)\}}f(u)^2f_{T-\b'X}(u)\,du.
\end{equation} 
Note that $I_P(\b)^{-1}-I(\b)^{-1}$ equals the minimal increase of the variance of an estimator for $\b$ based on an unknown $F$ (semi-parametric case) compared to the situation where $F$ is known (parametric). 
In our simulation example $I_P(\b) = 26.3667$ and $I(\b) =  6.5917$.

\section{Acknowledgements}
\label{section:acknowledgements}
The authors want to thank Richard Nickl for useful comments and his reference to relevant theory in the book \cite{nickl:15} and Fadoua Balabdaoui for incisive questions and references to relevant literature. 
We also thank the associate editor and two referees for their valuable remarks.
The research of the second author was supported by the Research Foundation Flanders (FWO) [grant number 11W7315N]. 
Support from the IAP Research Network P7/06 of the Belgian State (Belgian Science Policy) is gratefully acknowledged.

\vspace{2cm}
Below we give the proofs of the results stated in  Sections \ref{section:MLE} and \ref{section:estimates} of the manuscript.
 Entropy results are used in our proofs. Before we prove the results we first give some definitions and an equicontinuity lemma needed in the proofs.
 
Consider a class of functions ${\cal F}$ on ${\cal R}$ and let $L_2(Q)$ be the $L_2-$norm defined by a probability measure $Q$ on ${\cal R}$, i.e. for $g\in {\cal F}$,
$$ \|g\|_{L_2} = \int g^2dQ. $$
For any probability measure $Q$ on ${\cal R}$ let $N_{B}(\zeta, {\cal F}, L_2(Q))$ be the minimal number $N$ for which there exists pairs of functions $\{[g_j^L,g_j^U], j=1,\ldots, N\}$ such that $\|g_j^U-g_j^L\|_{L_2}\le \zeta$ for all $j=1,\ldots,N$ and such that for each $g \in {\cal F}$ there is a $j\in \{1,\ldots,N\}$ such that $g_j^L\le g \le g_j^U$. The $\zeta-$entropy with bracketing of ${\cal F}$ (for the $L_2(Q)-$distance) is defined as $H_{B}(\zeta, {\cal F}, L_2(Q)) = \log(N_{B}(\zeta, {\cal F}, L_2(Q)))$.

 \begin{lemma}[Equicontinuity Lemma, Theorem 5.12, p.77 in \cite{geer:00}]
 	\label{lemma:equi}
 	Let ${\cal F}$ be a fixed class of functions  with envelope $F$ in $L_2(P) = \{f: \int f^2 dP < \infty\}$. Suppose that 
 	$$\int_{0}^1  H_{B}^{1/2}(u, {\cal F},L_2(P) ) \,du \leq \infty,$$
 	where $H_{B}$ is the entropy with bracketing of ${\cal F}$ for the $L_2$-norm.
 	Then, for all $\eta>0$ there exists a $\d > 0$ such that
 	\begin{align*}
 	\limsup_{n \to \infty} P\left( sup_{[\d]}|\sqrt n \int (f-g)d(\P_n-P_0)| > \eta\right) <\eta,
 	\end{align*}
 	where,
 	\begin{align*}
 	[\d] = \{ (f,g): \|f- g\|\leq \d\}.
 	\end{align*}
 \end{lemma}
 
 \section{Behavior of the maximum likelihood estimator}
 In this section we prove Lemma \ref{lemma:MLE_misspecified}. We first prove in Lemma \ref{lemma:bracketing-entropy} some entropy bounds needed in the proofs.
 
 \begin{lemma}
 	\label{lemma:bracketing-entropy}
 	Let 
 	$${\cal F} =\{(t,x)\mapsto F(t-\b'x) : F\in {\cal F}_0, \b \in \Theta \}, $$
 	where ${\cal F}_0$ is the set of subdistribution functions on $[a,b]$, where $[a,b]$ contains all values $t-\b'x$ for $\b\in\Theta$, and $(t,x)$ in the compact neighborhood over which we let them vary. Then,
 	\begin{equation}
 	\label{entropy_condition2}
 	\sup_{\e>0}\e H_B\left(\e,{\cal F},L_2(P_0)\right)=O(1),
 	\end{equation}
 	Furthermore, let 
 	$${\cal G} =\{(t,x)\mapsto g(t-\b'x) : g\in {\cal G}_0, \b \in \Theta \}, $$
 	where ${\cal G}_0$ is a set functions of uniformly bounded variation, then 
 	\begin{equation}
 	\label{entropy_condition3}
 	\sup_{\e>0}\e H_B\left(\e,{\cal G},L_2(P_0)\right)=O(1).
 	\end{equation}
 \end{lemma}
 \begin{proof}
 	We only prove the result for the class $\cal F$ since the proof for the class $\cal G$ can be obtained similarly.
 	
 	Fix $\e >0$. We first note that $\Theta$ can be covered by $N$ neighborhoods with diameter at most $\e^2$ where N is of order $\e^{-2d}$. 
 	Let $\{\b_1,\ldots, \b_N\}$ denote elements of each of these neighborhoods. 
 	Consider an $\e$-bracket $[F_j^L,F_j^U], j=1,\ldots, N'$ covering the class ${\cal F}_0$ such that,
 	\begin{align*}
 	\left\{\int\{F_j^U(u)-F_j^L(u)\}^2f_{T-\b' X}(u)\,du\right\}^{1/2}<\e.
 	\end{align*}
 	for $ j=1,\ldots, N'$. The existence of such an $\e$-net is assured by the fact that $f_{T-\b'X}$ is bounded above (uniformly in $\b$). The number $N'$ is of order $\exp(C/\e)$ for some constant $C$ (See e.g. \cite{geer:00}, p.\,18). Let $\b_j$ be chosen such that $\|\b_j-\b\|<\e^2$, where $\|\cdot\|$ denotes the Euclidean norm. Then:
 	\begin{align*}
 	t-\b_j' x-\e^2 R\le t-\b' x=t-\b_j' x-(\b-\b_j)' x\le t-\b_j' x+\e^2 R,
 	\end{align*}
 	where $R$ is the maximum of the values $\|x\|$. This implies that for each $F \in {\cal F}_0$ and $\b \in \Theta,$
 	\begin{align*}
 	F_i^L(t-\b_j' x-\e^2 R)\le F(t-\b' x) \le F_i^U(t-\b_j' x+\e^2 R),
 	\end{align*}
 	for some $ i=1,\ldots, N'$ and $ j=1,\ldots, N$. The result of Lemma \ref{lemma:bracketing-entropy} follows if we can show that,
 	\begin{align}
 	\label{term:proof-bracketing}
 	\left\{\int\{F_i^U(t-\b_j' x+\e^2 R)-F_i^L(t-\b_j' x-\e^2 R)\}^2dG(t,x)\right\}^{1/2}\lesssim\e.
 	\end{align}
 	By the triangle inequality we get that the left-hand side of the above equation is bounded by:
 	\begin{align*}
 	&\left\{\int\{F(t-\b_j' x-\e^2 R)-F_i^L(t-\b_j' x-\e^2 R)\}^2\,dG(t,x)\right\}^{1/2}\\
 	&\qquad+ \left\{\int\{F_i^U(t-\b_j' x+\e^2 R)-F(t-\b_j' x+\e^2 R)\}^2\,dG(t,x)\right\}^{1/2}
 	\\
 	&\qquad+\left\{\int\{F(t-\b_j' x+\e^2 R) - F(t-\b_j' x-\e^2 R)\}^2\,dG(t,x)\right\}^{1/2}\\
 	&\lesssim \e+\left\{\int\{F(u+\e^2 R)-F(u-\e^2 R)\}^2f_{T-\b_j'x}(u)\right\}^{1/2}.
 	\end{align*}
 	Let $u_0=a-\e^2 R<u_1,\dots<u_m=b+\e^2 R$, be points such that
 	$u_k-u_{k-1}=\e^2$, $k=1,\dots,m-1$, $u_m-u_{m-1}\le\e^2$. Then:
 	\begin{align*}
 	&\int\{F(u+\e^2 R)-F(u-\e^2 R)\}^2f_{T-\b_jX}(u)\,du\\
 	&\le\int\{F(u+\e^2 R)-F(u-\e^2 R)\}f_{T-\b_jX}(u)\,du\le M\int\{F(u+\e^2 R)-F(u-\e^2 R)\}\,du\\
 	&=M\int_{a+\e^2R}^{b+\e^2R}F(u)\,du-M\int_{a-\e^2R}^{b-\e^2R}F(u)\,du\\
 	&\le M\int_{a-\e^2R}^{a+\e^2R}F(u)\,du+M\int_{b-\e^2R}^{b+\e^2R}F(u)\,du\lesssim \e^2,
 	\end{align*}
 	where $M$ is an upper bound for $f_{T-\b_j'X}$, and where we extend the function $F$ by a constant value outside $[a,b]$. This completes the proof of (\ref{term:proof-bracketing}) since we have shown that there exist positive constants $A_1$, $A_2,A_3$ and $C$ such that
 	\begin{align*}
 	H_B\left(\e,{\cal F},L_2(P_0)\right) &\le \log N + \log N' \le d\log(A_1/\e^2) + A_2\log(\exp(C/\e))\le A_3/\e\\
	&=O(\log(1/\e))+O(1/\e)=O(1/\e),\qquad\e\downarrow0.
 	\end{align*}
 \end{proof}

 \begin{proof}[Proof of Lemma \ref{lemma:MLE_misspecified}]
   Let $h$ denote the Hellinger distance on the class of densities ${\cal P}$ defined by
 \begin{align*}
 {\cal P} =\left\{p_{\b,F}(t,x,\d) = \d F(t-\b'x)+(1-\d)\{1-F(t-\b'x)\}: F \in {\cal F}_0, \b \in \Theta\right\},
 \end{align*}
 w.r.t.\ the product of counting measure on $\{0,1\}$ and the measure $dG$ of $(T,X)$, where ${\cal F}_0$ is the class of right-continuous subdistribution functions.
 
We have (see, e.g., the ``basic inequality'' Lemma 4.5, p.\ 51 of \cite{geer:00}):
 \begin{align*}
 h^2\left(p_{\b,\hat F_{n,\b}},p_{_{\b,F_{\b}}}\right)\le \int \frac{2p_{\b,\hat F_{n,\b}}}{p_{\b,\hat F_{n,\b}}+p_{_{\b,F_{\b}}}}\,d\left(\P_n-P_0\right),
 \end{align*}
where we use the convexity of the set of densities of this type for (temporarily) fixed $\b$.
Hence we get, for $\ee\in(0,1]$:
\begin{align*}
&\P\left\{ \sup_{\b\in\Theta}h\left(p_{\b,\hat F_{n,\b}},p_{_{\b,F_{\b}}}\right)\ge\ee\right\}\\
&=\P\left\{\sup_{\b\in\Theta,\,h\bigl(p_{\b,\hat F_{n,\b}},\,p_{_{\b,F_{\b}}}\bigr)\ge\ee}\left\{
\int \left\{\frac{2p_{\b,\hat F_{n,\b}}}{p_{\b,\hat F_{n,\b}}+p_{_{\b,F_{\b}}}}-1\right\}\,d\left(\P_n-P_0\right)-h^2\left(p_{\b,\hat F_{n,\b}},p_{_{\b,F_{\b}}}\right)\right\}\ge0\right.\\
&\left.\phantom{\sup_{\b\in\Theta,\,h\bigl(p_{\b,\hat F_{n,\b}},\,p_{_{\b,F_{\b}}}\bigr)\ge\ee}\left\{
\int \left\{\frac{2p_{\b,\hat F_{n,\b}}}{p_{\b,\hat F_{n,\b}}+p_{_{\b,F_{\b}}}}-1\right\}\,d\left(\P_n-P_0\right)\right.}\qquad\qquad
,\,\sup_{\b\in\Theta}h\left(p_{\b,\hat F_{n,\b}},p_{_{\b,F_{\b}}}\right)\ge\ee\right\}\\
&\le\P\left\{\sup_{\b\in\Theta,\,F\in{\cal F}_0,\,h\bigl(p_{_{\b,F}},\,p_{_{\b,F_{\b}}}\bigr)\ge\ee}
\left\{\int \left\{\frac{2p_{_{\b, F}}}{p_{_{\b, F}}+p_{_{\b,F_{\b}}}}-1\right\}\,d\left(\P_n-P_0\right)-h^2\bigl(p_{_{\b,F}},\,p_{_{\b,F_{\b}}}\bigr)\right\}\ge0\right\}\\
\end{align*}
\begin{align*}
&\le\sum_{s=0}^\infty\P\left\{\sup_{\substack{\b\in\Theta,\,F\in{\cal F}_0, \\ 2^{2s}\ee \le h\bigl(p_{_{\b, F}},\,p_{_{\b,F_{\b}}}\bigr)\le2^{s+1}\ee}}
\sqrt{n}\int \left\{\frac{2p_{_{\b, F}}}{p_{_{\b, F}}+p_{_{\b,F_{\b}}}}-1\right\}\,d\left(\P_n-P_0\right)\ge\sqrt{n}\,2^{2s}\ee^2\right\},
\end{align*}
We can now use Theorem 5.13 in \cite{geer:00}, taking $\ee=Mn^{-1/3}, \a = 1,\b = 0$ and $T =\sqrt{n}\,2^{2s}\ee^2 = M2^{2s}n^{-1/6}$, together with Lemma \ref{lemma:bracketing-entropy} for the entropy of the set of densities to conclude:
\begin{align*}
&\sum_{s=0}^\infty\P\left\{\sup_{\substack{\b\in\Theta,\,F\in{\cal F}_0, \\ 2^{2s}\ee \le h\bigl(p_{_{\b, F}},\,p_{_{\b,F_{\b}}}\bigr)\le2^{s+1}\ee}}
\sqrt{n}\int \left\{\frac{2p_{_{\b, F}}}{p_{_{\b, F}}+p_{_{\b,F_{\b}}}}-1\right\}\,d\left(\P_n-P_0\right)\ge\sqrt{n}\,2^{2s}\ee^2\right\}\\ &\quad \le\sum_{s=0}^\infty c_1 \exp(-c_2 M2^{2s})
\end{align*}
for constants $c_1,c_2>0$. Since the sum can be made arbitrarily small for $M$ sufficiently large, we find:
\begin{align*}
\sup_{\b\in\Theta}h\left(p_{\b,\hat F_{n,\b}},p_{_{\b,F_{\b}}}\right)=O_p\left(n^{-1/3}\right).
\end{align*}
We have:
\begin{align*}
&h\left(p_{\b,\hat F_{n,\b}},p_{_{\b,F_{\b}}}\right)^2\\
&=\tfrac12\int\left\{p^{1/2}_{\b,\hat F_{n,\b}}(t,x,1)-p^{1/2}_{_{\b,F_{\b}}}(t,x,1)\right\}^2\,dG(t,x)
+\tfrac12\int\left\{p^{1/2}_{\b,\hat F_{n,\b}}(t,x,0)-p^{1/2}_{_{\b,F_{\b}}}(t,x,0)\right\}^2\,dG(t,x)\\
&=\tfrac12\int\left\{\hat F_{n,\b}(t-\b'x)^{1/2}-F_{\b}(t-\b'x)^{1/2}\right\}^2\,dG(t,x)\\
&\qquad\qquad\qquad\qquad+\tfrac12\int\left\{\left(1-\hat F_{n,\b}(t-\b'x)\right)^{1/2}-\left(1-F_{\b}(t-\b'x)\right)^{1/2}\right\}^2\,dG(t,x),
\end{align*}
and
\begin{align*}
&\int\left\{\hat F_{n,\b}(t-\b'x)-F_{\b}(t-\b'x)\right\}^2\,dG(t,x)\\
&=\int\left\{\hat F_{n,\b}(t-\b'x)^{1/2}-F_{\b}(t-\b'x)^{1/2}\right\}^2\left\{\hat F_{n,\b}(t-\b'x)^{1/2}+F_{\b}(t-\b'x)^{1/2}\right\}^2\,dG(t,x)\\
&\le4\int\left\{\hat F_{n,\b}(t-\b'x)^{1/2}-F_{\b}(t-\b'x)^{1/2}\right\}^2\,dG(t,x)\le 8h\left(p_{\hat F_{n,\b}},p_{_{F_{\b}}}\right)^2
\end{align*}
So we find
\begin{align*}
\sup_{\b\in\Theta}\int\left\{\hat F_{n,\b}(t-\b'x)-F_{\b}(t-\b'x)\right\}^2\,dG(t,x)=O_p\left(n^{-2/3}\right).
\end{align*}
\end{proof}

\section{Asymptotic behavior of the simple estimate based on the MLE $\hat F_{n,\b}$, avoiding any smoothing}
This section contains the proof of Theorem \ref{theorem:method1} stated in Section \ref{subsection:MLE} of the manuscript. The proof is decomposed into three parts: (a) proof of existence of a root of $\psi_{1,n}$, (b) proof of consistency of $\hat \b_n$ and (c) proof of asymptotic normality of $\sqrt n (\hat \b_n- \b_0)$. We first prove the properties given in Lemma \ref{lemma:population_score} of the population version of the statistic $\psi_{1,n}^{(\ee)}$ defined by,
\begin{align}
\label{score_population}
\psi_{1,\ee}(\b) &= \int_{F_\b(t-\b'x) \in [\ee,1-\ee]} x \{\d -F_\b(t-\b'x)\}\,dP_0(t,x,\d)\nonumber\\
&=\int_{F_\b(t-\b'x) \in [\ee,1-\ee]} x \{F_0(t-\b_0x) - F_\b(t-\b'x)\}\,dG(t,x).
\end{align}

\begin{proof}[Proof of Lemma \ref{lemma:population_score} ]
We first note that,
\begin{align*}
\psi_{1,\ee}(\b_0) =\int_{F_{0}(t-\b_0'x)\in[\ee,1-\ee]} x\Bigl\{\E\{\dd |(T,X) = (t,x)\} - F_0(t-\b_0' x) \Bigr\}\,dG(t,x) = 0,
\end{align*} 
since $\E\{\dd |(T,X) = (t,x)\} =  F_{0}(t-\b_0' x)$. We next continue by showing result (i). Since
$$\E\left( \dd| T-\b'X=t-\b'x \right) = F_\b(t-\b'x), $$
we get,
\begin{align*}
&\E_{\ee,\b}\left[\text{Cov}\left( \dd, X | T-\b'X \right)\right]\\
&\stackrel{\mbox{\small def}}=\int_{F_\b(t-\b'x) \in [\ee,1-\ee]}\text{Cov}\left( \dd, X | T-\b'X =u \right)f_{T-\b'X}(u)\,du \\
&=\int_{F_{\b}(u)\in[\ee,1-\ee]}\hspace{-0.3cm}\text{Cov}\left\{X,F_0(u+(\b-\b_0)'X)\Bigm|T-\b'X=u\right\}f_{T-\b'X}(u)\,du\\
&=\int_{F_{\b}(t-\b'x)\in[\ee,1-\ee]} x\Bigl\{F_0(t-\b'x+(\b-\b_0)'x)-F_{\b}(t-\b' x)\Bigr\}\,dG(t,x)\\
&=\int_{F_{\b}(t-\b'x)\in[\ee,1-\ee]} x\left\{F_0(t-\b_0' x) - F_{\b}(t-\b' x)\right\}\,dG(t,x) =\psi_{1,\ee}(\b).
\end{align*}

For the second result (ii), we write
\begin{align*}
&(\b -\b_0)'\text{Cov}\left( \dd, X | T-\b'X = u \right) \\
&\quad  =  \text{Cov}\left( F_0( T-\b'X +(\b -\b_0)'X), (\b -\b_0)'X | T-\b'X = u \right)
\end{align*} 
which is positive for all $\b$, following from the fact that $F_0$ is an increasing function. Indeed, using Fubini's theorem, one can prove that for any random variables $X$ and $Y$ such that $XY,X$ and $Y$ are integrable, we have
\begin{align*}
\text{\rm Cov}\left\{X,Y \right\} = EXY-EXEY=\int\{\P(X\ge s,Y\ge t)-\P(X\ge s)\P(Y\ge t)\}\,ds\,dt.
\end{align*}
Denote $Z_1 =(\b-\b_0)' X$ and $Z_2 =F_0( u + (\b -\b_0)'X) = F_0(u+Z_1)$. For simplicity of notation we no longer write the condition $T-\b'X = u$ but note that the results below hold conditional on $T-\b'X = u $. Using the monotonicity of the function $F_0$, we have
\begin{align*}
\P(Z_1\ge z_1,Z_2\ge z_2) &= \P(Z_1 \ge \max\{z_1, \tilde z_2 \}) \ge \P(Z_1 \ge \max\{z_1, \tilde z_2 \})\P(Z_1 \ge \min\{z_1, \tilde z_2 \})\\
& = \P(Z_1 \ge z_1)\P(Z_2 \ge z_2),
\end{align*}
where 
\begin{align*}
\tilde z_2 = F_0^{-1}(z_2) - u.
\end{align*}
We conclude that,
\begin{align*}
&\text{Cov}\left( F_0( T-\b'X +(\b -\b_0)'X), (\b -\b_0)'X | T-\b'X = u \right)\\
&\qquad = \int\{\P(Z_1\ge z_1,Z_2\ge z_2)-\P(Z_1\ge z_1)\P(Z_2\ge z_2)\}\,dz_1\,dz_2 \ge 0,
\end{align*}
and hence (ii) follows from the assumption that the covariance $\text{\rm Cov}(X,F_0(u+(\b-\b_0)'X)|T-\b'X=u)$ is not identically zero for $u$ in the region $A_{\ee,\b}$, for each $\b\in \Theta$, implying:
\begin{align*}
&\E_{\ee,\b}\left[\text{Cov}\left( \dd, X | T-\b'X \right)\right]\\
&=\int_{F_\b(u) \in [\ee,1-\ee]}\text{Cov}\left( F_0( T-\b'X +(\b -\b_0)'X), (\b -\b_0)'X | T-\b'X = u \right)f_{T-\b'X }(u)du\ge 0.
\end{align*}
\noindent
[Uniqueness of $\b_0$:]
	
We next show that $\b_0$ is the only value $\b_* \in \Theta$ such that $ \E_{\ee,\b}\left[(\b -\b_*)'\text{Cov}\left( \dd, X | T-\b'X \right)\right] \ge 0 \text{ for all } \b \in \Theta$.
We start by assuming that, on the contrary, there exists $\b_1 \ne \b_0 $ in $\Theta$ such that
$$(\b -\b_0 )'\psi_{1,\ee}(\b)\ge 0 \qquad \text{ and } (\b - \b_1 )'\psi_{1,\ee}(\b)\ge 0 \qquad \text{for all } \b \in \Theta,$$
and we consider the point $\tilde \b \in \Theta$ given by 
$$
\tilde \b = \tfrac12\{\b_0 + \b_1\}.
$$
The existence of the point $\tilde \b$ is ensured by the convexity of the set $\Theta$.
For this point, we have,
$$(\tilde \b -\b_0)'\psi_{1,\ee}(\tilde\b)=-(\tilde \b -\b_1 )'\psi_{1,\ee}(\tilde\b), $$
which is not possible since both terms should be positive and $\psi_{1,\ee}(\tilde\b)$ is not equal to zero (since, by the assumption that the covariance $\text{\rm Cov}(X,F_0(u+(\b-\b_0)'X)|T-\b'X=u)$ is not identically zero for $u$ in the region $A_{\ee,\b}$,  $\psi_{1,\ee}(\b)$ is only zero at $\b =\b_0$.)

We now calculate the derivative of $\psi_{1,\ee}$ at $\b = \b_0$. We have,
\begin{align*}
&\psi_{1,\ee}'(\b)=  \frac{\partial}{\partial \b}\int_{ F_\b^{-1}(\ee)\leq t-\b'x \leq F_\b^{-1}(1-\ee) } x \{\d -F_\b(t-\b'x)\}dP_0(t,x,\d) \\
&=  \frac{\partial}{\partial \b}\int_{ F_\b^{-1}(\ee)\leq t-\b'x \leq F_\b^{-1}(1-\ee) } x \{F_0(t-\b_0'x) - F_\b(t-\b'x) \}\,dG(t,x)\\
&=\frac{\partial}{\partial \b}\int_{u=F_\b^{-1}(\ee)}^{F_\b^{-1}(1-\ee)}\int x\left\{F_0(u+(\b-\b_0)'x)- F_\b(u)\right\} f_{X|T-\b'X}(x|u)f_{T-\b'X}(u)\,dx\,du\\
\end{align*}
\begin{align*}
&=\int_{u=F_\b^{-1}(\ee)}^{F_\b^{-1}(1-\ee)} \int\frac{\partial}{\partial \b}\left\{x\left\{F_0(u+(\b-\b_0)'x)-F_\b(u)\right\} f_{X|T-\b'X}(x|u)f_{T-\b'X}(u)\right\}\,dx\,du\\
&\qquad+\left\{\frac{\partial}{\partial \b}F_\b^{-1}(1-\ee)\right\}\int x\left\{F_0(F_\b^{-1}(1-\ee)+(\b-\b_0)'x)- (1-\ee)\right\}\\
&\qquad\qquad\qquad\qquad\qquad\qquad\qquad\cdot f_{X|T-\b' X}(x|F_\b^{-1}(1-\ee))f_{T-\b'X}(F_\b^{-1}(1-\ee))\,dx\\
&\qquad-\left\{\frac{\partial}{\partial \b}F_\b^{-1}(\ee)\right\}\int x\left\{F_0(F_\b^{-1}(\ee)+(\b-\b_0)x) - \ee\right\}\\
&\qquad\qquad\qquad\qquad\qquad\qquad\qquad \cdot f_{X|T-\b' X}(x|F_\b^{-1}(\ee))f_{T-\b'X}(F_\b^{-1}(\ee))\,dx.
\end{align*}
Note that if $\b=\b_0$, we get:
\begin{align*}
&\psi_{1,\ee}'(\b_0)=\\
&\int_{F_0^{-1}(\ee)}^{F_0^{-1}(1-\ee)}\int\left.\frac{\partial}{\partial \b}\left\{x\left\{ F_0(u+(\b-\b_0)'x)- F_\b(u)\right\} f_{X|T-\b'X}(x|u)f_{T-\b'X}(u)\right\}\right|_{\b=\b_0}\hspace{-0.4cm}\,du\,dx.
\end{align*}
since the last two terms vanish because the integrands become zero in that case. Note that,
\begin{align*}
\frac{\partial}{\partial \b}F_\b(u) &= \int y f_0(u+(\b-\b_0)'y)f_{X|T-\b' X}(y|u)\,dy  \\
&\qquad+ \int F_0(u+(\b-\b_0)'y) \frac{\partial}{\partial \b}f_{X|T-\b' X}(y|u)\,dy,
\end{align*}
implying that, at $\b = \b_0$,
\begin{align*}
\frac{\partial}{\partial \b}F_\b(u)\Bigm|_{\b = \b_0} &=f_0(u) \E(X |T-\b_0' X = u).
\end{align*}
Since
\begin{align*}
&\int_{u=F_0^{-1}(\ee)}^{F_0^{-1}(1-\ee)}\int x \left\{ x - \E(X |T-\b_0' X = u) \right\}'f_{X|T-\b_0'X}(x|u)\,f_0(u)f_{T-\b_0'X}(u)\,dx\,du\\
&= \E_{\ee}\left[ X\{X-\E(X| T-\b_0'X)\}'f_0(T-\b_0'X)\right],
\end{align*}
(\ref{derivative_score}) now follows.
\end{proof}

\begin{proof}[Proof of Theorem \ref{theorem:method1}, Part 1 (Existence of a root)]
Consider the score function
\begin{align*}
\psi_{1,n}^{(\ee)}(\b)&=\int_{\hat F_{n,\b}(t-\b'x) \in [\ee,1-\ee]} x\bigl\{\d - \hat F_{n,\b}(t-\b'x)\bigr\}\,d\P_n(t,x,\d),
\end{align*}
where $\hat F_{n,\b}$ is the nonparametric maximum likelihood estimator (MLE) of the error distribution.
According to the discussion in Section \ref{subsection:MLE} we have to show that there exists a point $\hat\beta_n$ such that 
\begin{align*}
\psi_{1,n}^{(\ee)}(\b) =&\int_{\hat F_{n,\b}(t-\b'x)\in[\ee,1-\ee]} x\bigl\{\d-\hat F_{n,\b}(t-\b'x)\bigr\}\,d\P_n(t,x,\d)
\end{align*}
has a zero-crossing at $\b =\hat\b_n$. We have:
\begin{align}
	\label{zeroing}
	\\
	\psi_{1,n}^{(\ee)}(\b) &=  \int_{\hat F_{n,\b}(t-\b'x) \in [\ee,1-\ee]} x\bigl\{ \d - F_\b(t-\b'x)\bigr\}\,d\P_n(t,x,\d)\nonumber\\
	&\qquad+ \int_{\hat F_{n,\b}(t-\b'x) \in [\ee,1-\ee]} x\bigl\{F_\b(t-\b'x) - \hat F_{n,\b}(t-\b'x)\bigr\}\,d\P_n(t,x,\d)\nonumber\\
	&=  \int_{\hat F_{n,\b}(t-\b'x) \in [\ee,1-\ee]} x\bigl\{\d -  F_\b(t-\b'x)\bigr\}\,d\P_n(t,x,\d)\nonumber\\
	&\qquad+ \int_{\hat F_{n,\b}(t-\b'x) \in [\ee,1-\ee]} x\bigl\{F_\b(t-\b'x) - \hat F_{n,\b}(t-\b'x)\bigr\}\,d(\P_n-P_0)(t,x,\d)\nonumber\\
	&\qquad+ \int_{\hat F_{n,\b}(t-\b'x) \in [\ee,1-\ee]} x\bigl\{F_\b(t-\b'x) -\hat F_{n,\b}(t-\b'x)\bigr\}dP_0(t,x,\d)\nonumber.
\end{align}
	Let ${\cal F}$ be the set of piecewise constant distribution functions with finitely many jumps (like the MLE $\hat F_{n,\hat\b_n}$), and let, for $\b\in\Theta$,  ${\cal K}$ be the set of functions
\begin{align}
	\label{cal_K1}
	\K&=\left\{(t,x,\d)\mapsto x\bigl\{\d - F_{\b}(t-\b'x)\bigr\}1_{[\ee,1-\ee]}\left(F(t-\b'x)\right): F\in{\cal F},\b\in\Theta\right\}.
\end{align}
We add the function
\begin{align*}
(t,x,\d)\mapsto x\bigl\{\d - F_{\b}(t-\b'x)\bigr\}1_{[\ee,1-\ee]}\left(F_{\b}(t-\b'x)\right)
\end{align*}
to $\cal K$. 
We denote by $H_B(\zeta,\K,L_2(P_0))$ the bracketing $\zeta$-entropy w.r.t.\ the $L_2$-distance $d$, defined by
\begin{align}
\label{distance_cal_K}
d(k_1,k_2)^2=\int \left\|k_1-k_2\right\|^2\,dP_0,\qquad k_1,k_2\in\cal K.
\end{align}	
Note that
\begin{align*}
x\bigl\{\d -F_{\b}(t-\b'x)\bigr\}1_{[\ee,1-\ee]}\left(F(t-\b'x)\right)=f_{1,\b}(t,x,\d)f_{2,\b}(t,x,\d),
\end{align*}
where
\begin{align*}
	f_{1,\b}(t,x,\d)=x\bigl\{\d -F_{\b}(t-\b'x)\bigr\},
\end{align*}
and
\begin{align*}
f_{2,\b}(t,x,\d)=1_{[\ee,1-\ee]}\left(F(t-\b'x)\right).
\end{align*}
Since $t$ and $x$ vary over a bounded region and, by (A4), $F_{\b}$ is of bounded variation, $f_{1,\b}$ is of bounded variation.
Moreover,
\begin{align*}
	f_{2,\b}(t,x,\d)=1_{[\ee,1-\ee]}\left(F(t-\b'x)\right)=1_{[\ee,1]}\left(F(t-\b'x)\right)-1_{(1-\ee,1]}\left(F(t-\b'x)\right).
\end{align*}
Since $F$ is monotone, we have:
\begin{align}
	\label{BV_indicator_function}
	1_{[\ee,1]}\left(F(t-\b'x)\right)-1_{(1-\ee,1]}\left(F(t-\b'x)\right)=1_{[a_{\ee,F},M]}(t-\b'x)-1_{(b_{\ee,F},M]}(t-\b'x)
\end{align}
for points $a_{\ee,F}\le b_{\ee,F}$, where $M$ is an upper bound for the values of $t-\b'x$. Hence $f_{2,\b}$ is also a function of uniformly bounded variation.
	
We therefore get, using Lemma \ref{lemma:bracketing-entropy}
\begin{align*}
	\sup_{\zeta>0} \zeta H_B\left(\zeta,{\cal K},L_2(P_0)\right)=O(1),
\end{align*}
which implies:
\begin{align*}
\int_0^{\zeta} H_B\left(u,{\cal K},L_2(P_0)\right)^{1/2}\,du=O\left(\zeta^{1/2}\right),\qquad\zeta>0.
\end{align*}
This implies
\begin{align*}
&\int_{\hat F_{n,\b}(t-\b'x) \in [\ee,1-\ee]} x\bigl\{\d- F_\b(t-\b'x)\bigr\}\,d\P_n(t,x,\d)\\
	&=\int_{\hat F_{n,\b}(t-\b'x) \in [\ee,1-\ee]} x\bigl\{ \d-F_\b(t-\b'x)\bigr\}\,dP_0(t,x,\d)\\
	&\qquad +\int_{\hat F_{n,\b}(t-\b'x) \in [\ee,1-\ee]} x\bigl\{\d- F_\b(t-\b'x)\bigr\}\,d\bigl(\P_n-P_0\bigr)(t,x,\d)\\
	&=\int_{F_{\b}(t-\b'x) \in [\ee,1-\ee]} x\bigl\{\d- F_\b(t-\b'x)\bigr\}\,dP_0(t,x,\d)\\
	&\qquad +\int_{F_{\b}(t-\b'x) \in [\ee,1-\ee]} x\bigl\{\d- F_\b(t-\b'x)\bigr\}\,d\bigl(\P_n-P_0\bigr)(t,x,\d) +o_p(1)\\
	&=\psi_{1,\ee}(\b)+o_p(1),
\end{align*}
uniformly in $\b$ in $\Theta$, by the convergence in probability (and almost surely) of $\hat F_{n,\b}$ to $F_{\b}$, where we use Lemma \ref{lemma:equi}
for the second term on the right-hand side of the first equality to make the transition of the integration region $\hat F_{n,\b}(t-\b'x) \in [\ee,1-\ee]$ to
$F_{\b}(t-\b'x) \in [\ee,1-\ee]$.
	
For the second term of (\ref{zeroing}) we argue similarly, this time using the function class
\begin{align}
	\label{cal_K1_{beta}'}
	\K'&=\left\{(t,x,\d)\mapsto x\bigl\{F_{\b}(t-\b'x) - F(t-\b'x)\bigr\}1_{[\ee,1-\ee]}\left(F(t-\b'x)\right): F\in{\cal F},\,\b\in\Theta\right\}.
\end{align}
to which we add the function that is identically zero. This implies that these terms are $o_p(1)$.
For the third term of (\ref{zeroing}) we get by an application of the Cauchy-Schwarz inequality that, uniformly in $\b$,
\begin{align*}
&\int_{\hat F_{n,\b}(t-\b'x) \in [\ee,1-\ee]} x\bigl\{F_{\b}(t-\b'x) -\hat F_{n,\b}(t-\b'x)\bigr\}dP_0(t,x,\d)\\
&\le \left( \int_{\hat F_{n,\b}(t-\b'x) \in [\ee,1-\ee]} x^2 dP_0(t,x,\d) \int_{\hat F_{n,\b}(t-\b'x) \in [\ee,1-\ee]} \bigl\{F_{\b}(t-\b'x) - \hat F_{n,\b}(t-\b'x)\bigr\}^2dP_0(t,x,\d)\right)^{1/2}\\
&=O_p(n^{-1/3}).
\end{align*}
The conclusion is that,
\begin{align}
\label{relation_Score}
\psi_{1,n}^{(\ee)}(\b) = \psi_{1,\ee}(\b)+o_p(1),
\end{align}
uniformly in $\b \in \Theta$.

[Existence of $\hat\b_n$:]
	Let $\psi_{1,\ee}$ be the population version of the statistic $\psi_{1,n}^{(\ee)}$ defined by,
	\begin{align}
	\label{score_population}
	\psi_{1,\ee}(\b) = \int_{F_\b(t-\b'x) \in [\ee,1-\ee]} x \{\d-F_\b(t-\b'x)\}\,dP_0(t,x,\d).
	\end{align}
	We have
	\begin{align*}
	\psi_{1,\ee}(\b_0) = 0.
	\end{align*}
Furthermore, 
	\begin{align}
\label{expansion_psi_n}
	& \psi_{1,n}^{(\ee)}(\b) = \psi_{1,\ee}'(\b_0)(\b-\b_0) +R_n(\b),
	\end{align}
	where $R_n(\b) = o_p(1)+o(\b-\b_0)$, and where the $o_p(1)$ term is uniform in $\b\in\Theta$. Note that $\psi_{1,\ee}'(\b_0)$  is by definition non-singular.

We now define, for $h>0$, the functions
\begin{align*}
\tilde R_{n,h}(\b)=h^{-d}\int K_h(u_1-\b_1)\dots K_h(u_d-\b_d)\,R_n(u_1,\dots,u_d)\,du_1\dots\,du_d,
\end{align*}	
where $d$ is the dimension of $\Theta$ and
\begin{align*}
K_h(x)=h^{-1}K(x/h),\qquad x\in\R,
\end{align*}
letting $K$ be one of the usual smooth kernels with support $[-1,1]$, like the Triweight kernel that we used in the simulations.

Furthermore, we define:
\begin{align*}
\tilde \psi_{1,n,h}^{(\ee)}(\b)= \psi_{1,\ee}'(\b_0)(\b-\b_0) +\tilde R_{nh}(\b).
\end{align*}
Clearly:
\begin{align*}
\lim_{h\downarrow0}\tilde \psi_{1,n,h}^{(\ee)}(\b))=\psi_{1,n}^{(\ee)}(\b)\qquad\text{and}\qquad \lim_{h\downarrow0}\tilde R_{nh}(\b)=R_n(\b)
\end{align*}
for each continuity point $\b$ of $\psi_{1,n}^{(\ee)}$.

We now reparametrize, defining
\begin{align*}
\g=\psi_{1,\ee}'(\b_0)\b,\qquad \g_0=\psi_{1,\ee}'(\b_0)\b_0.
\end{align*}
This gives:
\begin{align*}
\psi_{1,\ee}'(\b_0)(\b-\b_0) +\tilde R_{nh}(\b)=\g-\g_0+\tilde R_{nh}\left(\psi_{1,\ee}'(\b_0)^{-1}\g\right).
\end{align*}
By (\ref{expansion_psi_n}), the mapping
\begin{align*}
\g\mapsto \g_0-R_n\left(\psi_{1,\ee}'(\b_0)^{-1}\g\right),
\end{align*}
maps, for each $\eta>0$, the ball $B_{\eta}(\g_0)=\{\g:\|\g-\g_0\}\le\eta\}$ into $B_{\eta/2}(\g_0)=\{\g:\|\g-\g_0\}\le\eta/2\}$ for all large $n$, with probability tending to one, where $\|\cdot\|$ denotes the Euclidean norm, implying that the {\it continuous} map
\begin{align*}
\g\mapsto \g_0-\tilde R_{nh}\left(\psi_{1,\ee}'(\b_0)^{-1}\g\right),
\end{align*}
maps $B_{\eta}(\g_0)=\{\g:\|\g-\g_0\}\le\eta\}$ into itself for all large $n$ and small $h$. So for large $n$ and small $h$
there is, by Brouwer's fixed point theorem a point $\g_{nh}$ such that
\begin{align*}
\g_{nh}=\g_0-\tilde R_{nh}\left(\psi_{1,\ee}'(\b_0)^{-1}\g_{nh}\right).
\end{align*}
Defining $\b_{nh}=\psi_{1,\ee}'(\b_0)^{-1}\g_{nh}$, we get:
\begin{align}
\label{Brouwer_relation}
\tilde \psi_{1,n,h}^{(\ee)}(\b_{nh})= \psi_{1,\ee}'(\b_0)(\b_{nh}-\b_0) +\tilde R_{nh}(\b_{nh})=0.
\end{align}
By compactness, $(\b_{n,1/k})_{k=1}^{\infty}$ must have a subsequence $(\b_{n,1/k_i})$ with a limit $\tilde\b_n$, as $i\to\infty$.
We show that each component of $\psi_{1,n}^{(\ee)}$ has a crossing of zero at $\tilde\b_n$.

Suppose that the $j$th component $\psi_{1,n,j}^{(\ee)}$ of $\psi_{1,n}^{(\ee)}$ does not have a crossing of zero at $\tilde\b_n$.
Then there must be an open ball $B_{\d}(\tilde\b_n)=\{\b:\|\b-\tilde\b_n\|<\d\}$ of $\tilde\b_n$ such that $\psi_{1,n,j}^{(\ee)}$ has a constant sign in $B_{\d}(\tilde\b_n)$,
say $\psi_{1,n,j}^{(\ee)}(\b)>0$ for $\b\in B_{\d}(\tilde\b_n)$. Since $\psi_{1,n,j}^{(\ee)}$ only has finitely many values, this means that
\begin{align*}
\psi_{1,n,j}^{(\ee)}(\b)\ge c>0,\qquad\text{for all }\b\in B_{\d}(\tilde\b_n),
\end{align*}
for some $c>0$. This means that the $j$th component $\tilde \psi_{1,n,h,j}^{(\ee)}$ of $\tilde \psi_{1,n,h}^{(\ee)}$ satisfies
\begin{align*}
&\tilde \psi_{1,n,h,j}^{(\ee)}(\b)= \psi_{1,\ee}'(\b_0)(\b-\b_0) +\tilde R_{nh}(\b)\\
&=h^{-d}\int \left\{\left[\psi_{1,\ee}'(\b_0)(\b-\b_0)\right]_j+R_{nj}(u_1,\dots,u_d)\right\}K_h(u_1-\b_1)\dots K_h(u_d-\b_d)\,du_1\dots\,du_d\\
&\ge
h^{-d}\int \left\{\left[\psi_{1,\ee}'(\b_0)(u-\b_0)\right]_j+R_{nj}(u_1,\dots,u_d)\right\}K_h(u_1-\b_1)\dots K_h(u_d-\b_d)\,du_1\dots\,du_d
-c/2\\
&\ge c\,h^{-d}\int K_h(u_1-\b_1)\dots K_h(u_d-\b_d)\,du_1\dots\,du_d-c/2\\
&=c/2,
\end{align*}
for $\b\in B_{\d/2}(\tilde\b_n)$ and sufficiently small $h$, contradicting (\ref{Brouwer_relation}), since $\b_{nh}$, for $h=1/k_i$, belongs to
$B_{\d/2}(\tilde\b_n)$ for large $k_i$.

\end{proof}
\begin{proof}[Proof of Theorem \ref{theorem:method1}, Part 2 (Consistency)]
	We assume that $\hat\b_n$ is contained in the compact set $\Theta$, and hence the sequence $(\hat\b_n)$ has a subsequence $(\hat\b_{n_k}=\hat\b_{n_k}(\omega))$, converging to an element $\b_*$.
	If $\hat\b_{n_k}=\hat\b_{n_k}(\omega)\longrightarrow \b_*$, we get by Lemma \ref{lemma:MLE_misspecified},
	\begin{align*}
	\hat F_{n_k,\hat\b_{n_k}}(t-\hat\b_{n_k}'x)\longrightarrow  F_{\b_*}(t-\b_*'x),
	\end{align*}
	where $F_\b$ is defined in (\ref{def_F_beta}).
	In the limit  we get therefore the relation
	\begin{align}
	&\lim_{k\to\infty}\int_{\hat F_{n_k,\hat\b_{n_k}}(t-\hat\b_{n_k}'x)\in[\ee,1-\ee]} x\bigl\{\d -F_{n_k,\hat\b_{n_k}}(t-\hat\b_{n_k}'x)\bigr\}d\P_{n_k}(t,x,\d)
	\nonumber\\
	&=\int_{F_{\b_*}(t-\b_*'x)\in[\ee,1-\ee]}x\bigl\{F_0(t-\b_0'x)-F_{\b_*}(t-\b_*'x)\bigr\}\,dG(t,x) = 0,\nonumber
	\end{align}
using that, in the limit, the crossing of zero becomes a root of the continuous limiting function.
	Consider
	\begin{align*}
	&\int_{F_{\b_*}(t-\b_*'x)\in[\ee,1-\ee]} x\bigl\{F_0(t-\b_0'x)-F_{\b_*}(t-\b_*' x)\bigr\}\,dG(t,x)\\
	&=\int_{F_{\b_*}(t-\b_*'x)\in[\ee,1-\ee]} x\bigl\{F_0(t-\b_*'x+(\b_*-\b_0)'x)-F_{\b_*}(t-\b_*' x)\bigr\}\,dG(t,x).
	\end{align*}
	Since
	\begin{align*}
	F_{\b_*}(t-\b_*'x)=\int F_0(t-\b_*'x+(\b_*-\b_0)'y)f_{X|T-\b_*'X}(y|T-\b_*'X=t-\b_*'x)\,dy,
	\end{align*}
	we get:
	\begin{align*}
	&(\b_*-\b_0)'\int_{F_{\b}(t-\b'x)\in[\ee,1-\ee]} x\Bigl\{F_0(t-\b_*'x+(\b_*-\b_0)'x)-F_{\b_*}(t-\b_*' x)\Bigr\}\,dG(t,x)\\
	&=\int_{F_{\b_*}(t-\b_*'x)\in[\ee,1-\ee]} (\b_*-\b_0)'x\biggl\{F_0(t-\b_*'x+(\b_*-\b_0)'x) \\
	&\qquad\left.-\int F_0(t-\b_*'x+(\b_*-\b_0)'y)f_{X|T-\b_*'X}(y|T-\b_*'X=t-\b_*'x)\,dy\right\}\,dG(t,x)\\
	&=\int_{F_{\b_*}(u)\in[\ee,1-\ee]}\hspace{-0.3cm}\text{Cov}\left\{(\b_*-\b_0)'X,F_0(u+(\b_*-\b_0)'X)\Bigm|T-\b_*'X=u\right\}f_{T-\b_*'X}(u)\,du\\
	& =0.
	\end{align*}
	We first note that by Lemma \ref{lemma:population_score} the integrand is positive for all $\b_* \in \Theta$. Suppose that $\b_* \ne \b_0$, then this integral can only be zero if $\text{Cov}((\b_*-\b_0)'X,F_0(u+(\b_*-\b_0)'X)|T-\b_*'X=u)$ is zero for all $u$ such that $F_{\b_*}(u)\in[\ee,1-\ee]$, if $f_{T-\b_*'X}(u)$ stays away from zero on this region (Assumptions (A3)), using continuity of the functions in the integrand (Assumptions (A5)) and the non negativity of the conditional covariance function (see also Remark \ref{remark:Hoeffding}). Since this is excluded by the condition that the covariance $\text{\rm Cov}(X,F_0(u+(\b-\b_0)'X)|T-\b'X=u)$ is continuous in $u$ and not identically zero for $u$ in the region $\{u:\ee\le F_{\b}(u)\le 1-\ee\}$, for each $\b\in \Theta$, we must have: $\b_*=\b_0$.
\end{proof}

\begin{proof}[Proof of Theorem \ref{theorem:method1}, Part 3 (Asymptotic Normality)]
	Before working out the details, we give a kind of ``road map'' for the proof of Theorem \ref{theorem:method1}, Part 3.
\begin{enumerate}
\item

We define $\psi_{1,n}^{(\ee)}$ at $\hat\b_n$ by putting
\begin{align}
\label{estimator1_crossing_def}
\psi_{1,n}^{(\ee)}(\hat\b_n)=0.
\end{align}
Note that, with this definition, $\psi_{1,n}^{(\ee)}(\hat\b_n)$ is in dimension 1 just the convex combination of the left and right limit at $\hat\b_n$:
\begin{align}
\label{convex_comb_def}
\psi_{1,n}^{(\ee)}(\hat\b_n)=\a\psi_{1,n}^{(\ee)}(\hat\b_n-)+(1-\a)\psi_{1,n}^{(\ee)}(\hat\b_n+)=0,
\end{align}
where we can choose $\a\in[0,1]$ in such a way that (\ref{convex_comb_def}) holds. In dimension $d$ higher than one, we can also define $\psi_{1,n}^{(\ee)}$ at $\hat\b_n$ by (\ref{estimator1_crossing_def}) and use the representation of the components as a convex combination
since we have a crossing of zero componentwise. Since the following asymptotic representations are also valid for one-sided limits as used in (\ref{convex_comb_def}) we can use Definition (\ref{estimator1_crossing_def})  and assume $\psi_{1,n}^{(\ee)}(\hat\b_n)=0$.

We show:
\begin{align}
		\label{fundamental_relation}
		&\psi_{1,n}^{(\ee)}(\hat\b_n)\nonumber\\
		&=\int_{F_0(t-\b_0'x)\in[\ee,1-\ee]}\Bigl\{x-\f_0(t-\b_0'x)\Bigr\}\Bigl\{F_0(t-\b_0'x) - F_{\hat\b_n}(t-\hat\b_n'x)\Bigr\}\,dP_0(t,x,\d)\nonumber\\
		&\qquad+\int_{F_0(t-\b_0'x)\in[\ee,1-\ee]}\Bigl\{x-\f_0(t-\b_0'x)\Bigr\}\Bigl\{\d - F_0(t-\b_0'x)\Bigr\}\,d\bigl(\P_n-P_0\bigr)(t,x,\d)\nonumber\\
		&\qquad+o_p\left(n^{-1/2} + \hat\b_n - \b_0\right),
		\end{align}
		where
		\begin{align*}
		\f_0(u)=\f_{\b_0}(u),
		\end{align*}
		and where $\f_{\b}$ is defined by:
		\begin{align}
		\label{fi}
		\f_{\b}(u)=\E\left\{X|T-\b'X=u\right\}.
		\end{align}
		Since $\hat\b_n\stackrel{p}\longrightarrow\b_0$ and 
		\begin{align*}
		&\int_{F_0(t-\b_0'x)\in[\ee,1-\ee]}\Bigl\{x-\f_0(t-\b_0'x)\Bigr\}\Bigl\{F_0(t-\b_0'x) - F_{\hat\b_n}(t-\hat\b_n'x)\Bigr\}\,dP_0(t,x,\d)\\
		&=\psi_{1,\ee}'(\b_0)\left(\hat\b_n-\b_0\right)+o_p\left(\hat\b_n-\b_0\right),
		\end{align*}
		this yields, using the invertibility of $\psi_{1,\ee}'(\b_0)$,
		\begin{align*}
		&\sqrt n (\hat \b_n  - \b_0 ) \\
		&= - \psi_{1,\ee}'(\b_0)^{-1} \biggl\{\sqrt n \int_{F_0(t-\b_0'x)\in[\ee,1-\ee]}\Bigl\{x-\f_{\b_0}(t-\b_0'x)\Bigr\}\Bigl\{\d -F_0(t-\b_0'x)\Bigr\}\\
		&\qquad\qquad\qquad\qquad\qquad\qquad\qquad\qquad\,d\bigl(\P_n-P_0\bigr)(t,x,\d) \biggr\}+ o_p\left(1+ \sqrt n (\hat\b_n - \b_0)\right).
		\end{align*}
		As a consequence, the result of Theorem \ref{theorem:method1} follows, since
		\begin{align*}
		&\sqrt n \int_{F_0(t-\b_0'x)\in[\ee,1-\ee]}\Bigl\{x-\f_0(t-\b_0'x)\Bigr\}\Bigl\{\d - F_0(t-\b_0'x)\Bigr\}\,d\bigl(\P_n-P_0\bigr)(t,x,\d)\\
		&\stackrel{d}{\rightarrow} N(0, B).
		\end{align*}
		\item To show that (\ref{fundamental_relation}) holds, we need entropy results for the functions $u\mapsto \hat F_{n,\b}(u)$ and $u\mapsto \bar\f_{\b,\hat F_{n,\b}}(u)$ (see (\ref{def_bar_phi_{beta,F}}) below). We also have to deal with the simpler parametric functions $F_{\b}$ and $\f_{\b}$, parametrized by the finite dimensional parameter $\b$, which are the population equivalents of $\hat F_{n,\b}$ and $\bar \f_{\b,\hat F_{n,\b}}$.
		\item The result will then follow from the properties of $F_{\b}$ and $\f_{\b}$, together with the closeness of $\hat F_{n,\b}$ to $F_{\b}$ and $\bar\f_{\b,\hat F_{n,\b}}$ to $\f_{\b}$, respectively, and the convergence of $\hat\b_n$ to $\b_0$.
	\end{enumerate}
	
	Let $\bar\f_{\hat\b_n,\hat F_{n,\hat\b_n}}$ be a (random) piecewise constant version of $\f_{\hat\b_n}$, where, for a piecewise constant distribution function $F$ with finitely many jumps at $\t_1<\t_2<\dots$, the function $\bar\f_{\b,F}$ is defined in the following way.
	\begin{align}
	\label{def_bar_phi_{beta,F}}
	\bar\f_{\b,F}(u)=
	\left\{\begin{array}{lll}
	\f_{\b}(\t_i),\,&\mbox{ if }F_{\b}(u)>F(\t_i),\,u\in[\t_i,\t_{i+1}),\\
	\f_{\b}(s),\,&\mbox{ if }F_{\b}(u)=F(s),\mbox{ for some }s\in[\t_i,\t_{i+1}),\\
	\f_{\b}(\t_{i+1}),\,&\mbox{ if }F_{\b}(u)<F(\t_i),\,u\in[\t_i,\t_{i+1}).
	\end{array}
	\right.
	\end{align}
	We can write:
	\begin{align*}
	&\psi_{1,n}^{(\ee)}(\hat\b_n)=\int_{\hat F_{n,\hat\b_n}(t-\hat\b_n'x)\in[\ee,1-\ee]} x\bigl\{\d -\hat F_{n,\hat\b_n}(t-\hat\b_n'x)\bigr\}\,d\P_n(t,x,\d)\\
	&=\int_{\hat F_{n,\hat\b_n}(t-\hat\b_n'x)\in[\ee,1-\ee]}\bigl\{x-\f_{\hat\b_n}(t-\hat\b_n'x)\bigr\}\bigl\{\d -\hat F_{n,\hat\b_n}(t-\hat\b_n'x)\bigr\}\,d\P_n(t,x,\d)\\
	&\qquad+\int_{\hat F_{n,\hat\b_n}(t-\hat\b_n'x)\in[\ee,1-\ee]}\bigl\{\f_{\hat\b_n}(t-\hat\b_n'x)-\bar\f_{\hat\b_n,\hat F_{n,\hat\b_n}}(t-\hat\b_n'x)\bigr\}\\
	&\qquad\qquad\qquad\qquad\qquad\qquad\qquad\qquad\qquad\cdot\bigl\{\d -\hat F_{n,\b}(t-\hat\b_n'x)\bigr\}\,d\P_n(t,x,\d)\\
	&= I + II,
	\end{align*}
	using
	$$
	\int_{\hat F_{n,\hat\b_n}(t-\hat\b_n'x)\in[\ee,1-\ee]} \bar\f_{\hat \b_n,\hat F_{n,\hat\b_n}}(t-\hat\b_n'x)\bigl\{\d -\hat F_{n,\hat\b_n}(t-\hat\b_n'x)\bigr\}\,d\P_n(t,x,\d)=0,
	$$
	by the definition of the MLE $\hat F_{n,\hat\b_n}$ as the slope of the greatest convex minorant of the corresponding cusum diagram, based on the values of the $\dd_i$ in the ordering of the $T_i-\hat\b_n'X_i$ (see also Lemma A.5 on p.380 of \cite{piet_geurt_birgit:10}). 
	
	Since the function $u\mapsto \f_{\b}(u)$ has a totally bounded derivative (as a consequence of (\ref{fi}) and assumption (A5)), we can bound the Euclidean norm of the differences $\f_{\b}(u)-\bar\f_{\b,\hat F_{n,\b}}(u)$ above by a constant times $|\hat F_{n,\b}(u)-F_{\b}(u)|$, for $u\in A_{\ee,\b}$ (see (A2)), i.e.,
	\begin{align*}
	\|\f_{\b}(u)-\bar\f_{\b,\hat F_{n,\b}}(u)\|\leq K_\b |\hat F_{n,\b}(u)-F_{\b}(u)|,
	\end{align*}
	for some constant $K_\b>0$ where the constant $K_\b$ depends on $\b$ through $f_\b$ (see for this technique for example (10.64) in \cite{piet_geurt:14}). By Assumption (A2) we know that $f_\b$ is continuous for all $\b \in \Theta$ such that we can find a constant $K>0$ not depending on $\b$, satisfying,
	\begin{align}
	\label{difference_f-fbar}
	\|\f_{\b}(u)-\bar\f_{\b,\hat F_{n,\b}}(u)\|\leq K|\hat F_{n,\b}(u)-F_{\b}(u)|,
	\end{align}
	uniformly in $\b \in \Theta$. 	Note that we also need $f_{\b}(u)>0$ for applying this, which is ensured by (A2).
	
	
	We have:
	\begin{align*}
	&II =\int_{\hat F_{n,\hat\b_n}(t-\hat\b_n'x)\in[\ee,1-\ee]}\Bigl\{\f_{\hat\b_n}(t-\hat\b_n'x)-\bar\f_{\hat\b_n,\hat F_{n,\hat\b_n}}(t-\hat\b_n'x)\Bigr\}\\
	&\qquad\qquad\qquad\qquad\qquad\qquad\qquad\qquad\qquad\cdot\Bigl\{\d - \hat F_{n,\hat\b_n}(t-\hat\b_n'x)\Bigr\}\,d\P_n(t,x,\d)\\
	&=\int_{\hat F_{n,\hat\b_n}(t-\hat\b_n'x)\in[\ee,1-\ee]}\Bigl\{\f_{\hat\b_n}(t-\hat\b_n'x)-\bar\f_{\hat\b_n,\hat F_{n,\hat\b_n}}(t-\hat\b_n'x)\Bigr\}\\
	&\qquad\qquad\qquad\qquad\qquad\qquad\qquad\qquad\qquad\cdot \Bigl\{\d - \hat F_{n,\hat\b_n}(t-\hat\b_n'x)\Bigr\}\,d(\P_n-P_0)(t,x,\d)\\
	&\quad+\int_{\hat F_{n,\hat\b_n}(u)\in[\ee,1-\ee]}\Bigl\{\f_{\hat\b_n}(u)-\bar\f_{\hat\b_n,\hat F_{n,\hat\b_n}}(u)\Bigr\}\Bigl\{F_{\hat\b_n}(u)-\hat F_{n,\hat\b_n}(u)\Bigr\}f_{T-\hat\b_n'X}(u)\,du\\
	&\quad+\int_{\hat F_{n,\hat\b_n}(t-\hat\b_n'x)\in[\ee,1-\ee]}\Bigl\{\f_{\hat\b_n}(t-\hat\b_n'x)-\bar\f_{\hat\b_n,\hat F_{n,\hat\b_n}}(t-\hat\b_n'x)\Bigr\}\\
	&\qquad\qquad\qquad\qquad\qquad\qquad\qquad\qquad\qquad\cdot\Bigl\{F_0(t-\b_0'x) -F_{\hat\b_n}(t-\hat\b_n'x)\Bigr\}\,dP_0(t,x,\d)\\
	&= II_a + II_b+II_c.
	\end{align*}
		
	First consider $II_a$. Let ${\cal F}$ be the set of piecewise constant distribution functions with finitely many jumps (like the MLE $\hat F_{n,\hat\b_n}$), and let ${\cal K}_1$ be the set of functions
	\begin{align}
	\label{cal_K}
	{\cal K}_1&=\left\{(t,x,\d)\mapsto\bigl(\f_{\b}(t-\b'x)-\bar\f_{\b,F}(t-\b'x)\bigr)\bigl(\d - F(t-\b'x)\bigr)\right.\nonumber\\
	&\left.\qquad\qquad\qquad\qquad\qquad\qquad\qquad\cdot1_{[\ee,1-\ee]}(F(t-\b'x)): F\in{\cal F},\,\b\in\Theta\right\},
	\end{align}
	where $\bar\f_{\b,F}$ is again defined by (\ref{def_bar_phi_{beta,F}}). We add the function which is identically zero to ${\cal K}_1$. 
	
	The functions $u\mapsto F(u)$, for $F\in{\cal F}$ and (as argued above) $u\mapsto \bar\f_{\b,F}(u)$ are bounded functions of uniformly bounded variation.
	Note that, for $F_1,F_2\in\cal F$,
	\begin{align*}
	&F_1(t-\b_1'x)-F_2(t-\b_2'x)\\
	&=F_1(t-\b_1'x)-F_{\b_1}(t-\b_1'x)+F_{\b_1}(t-\b_1'x)-F_{\b_2}(t-\b_2'x)\\
	&\qquad\quad\qquad\quad\qquad\quad\qquad\quad\qquad\quad+F_{\b_2}(t-\b_2'x)-F_2(t-\b_2'x),
	\end{align*}
	and that (see (\ref{def_F_beta})):
	\begin{align*}
	&\left|F_{\b_1}(t-\b_1'x)-F_{\b_2}(t-\b_2'x)\right|\\
	&=\left|\int F_0(t-\b_0'x+(\b_1-\b_0)'(y-x))f_{X|T-\b_1'X}(y|t-\b_1'x)\,dy\right.\\
	&\left.\qquad\qquad\qquad\qquad-\int F_0(t-\b_0'x+(\b_2-\b_0)'(y-x))f_{X|T-\b_2'X}(y|t-\b_2'x)\,dy\right|\\
	&=O\left(|\b_1-\b_2|\right),
	\end{align*}
	by (A2) and (A5).
	
	For the indicator function $1_{[\ee,1]}\left(F(t-\b'x)\right)$ we get, as in (\ref{BV_indicator_function}), using the monotonicity of $F$,
	\begin{align*}
	&1_{[\ee,1]}\left(F(t-\b'x)\right)\\
	&=1_{[\ee,1]}\left(F(t-\b'x)\right)-1_{(1-\ee,1]}\left(F(t-\b'x)\right)=1_{[a_{\ee,F},M]}(t-\b'x)-1_{(b_{\ee,F},M]}(t-\b'x),
	\end{align*}
	for points $a_{\ee,F}\le b_{\ee,F}$, where $M$ is an upper bound for the values of $t-\b'x$, implying that the function
	\begin{align*}
	(t,x)\mapsto 1_{[\ee,1]}\left(F(t-\b'x)\right),
	\end{align*}
is of uniformly bounded variation.
	So the functions in ${\cal K}_1$ are products of functions of uniformly bounded variation, and we therefore get, using Lemma \ref{lemma:bracketing-entropy}
	\begin{align*}
	\sup_{\zeta>0} \zeta H_B\left(\zeta,{\K_1},L_2(P_0)\right)=O(1),
	\end{align*}
	which implies:
	\begin{align*}
	\int_0^{\zeta} H_B\left(u,{\K_1},L_2(P_0)\right)^{1/2}\,du=O\left(\zeta^{1/2}\right),\qquad\zeta>0.
	\end{align*}

	Defining
	\begin{align*}
	k_{\b,F}(t,x,\d)=&\bigl(\f_{\b}(t-\b'x)-\bar\f_{\b,F}(t-\b'x)\bigr)\bigl(\d - F(t-\b'x)\bigr)\\
	&\qquad\qquad\qquad\qquad\qquad\qquad \cdot1_{[\ee,1-\ee]}(F(t-\b'x))
	\end{align*}
	for $F\in{\cal F}$, we get, using (\ref{difference_f-fbar}),
	\begin{align*}
	&\left\{\int \left\|k_{\hat\b_n,\hat F_{n,\hat\b_n}}(t,x,\d)\right\|^2\,dP_0(t,x,\d)\right\}^2\\
	&\le \int_{\hat F_{n,\hat\b_n}(t-\hat\b_n'x)\in[\ee,1-\ee]} \left\|\f_{\hat\b_n}(t-\hat\b_n'x)-\bar\f_{\hat\b_n,\hat F_{n,\hat\b_n}}(t-\hat\b_n'x)\right\|^2\,dP_0(t,x,\d)\\
	&\le K\int_{\hat F_{n,\hat\b_n}(t-\hat\b_n'x)\in[\ee,1-\ee]} \left\{\hat F_{n,\hat\b_n}(t-\hat\b_n'x)-F_{\hat\b_n}(t-\hat\b_n'x)\right\}^2\,dP_0(t,x,\d)\\
	&\le K'\int_{\hat F_{n,\hat\b_n}(u)\in[\ee,1-\ee]} \left\{\hat F_{n,\hat\b_n}(u)-F_{\hat\b_n}(u)\right\}^2\,du\\
	&\stackrel{p}\longrightarrow0,
	\end{align*}
	for constants $K,K'>0$.
	This implies 
	\begin{align}
	\label{II_a}
	\sqrt n II_a =\sqrt n\int k_{\hat\b_n,\hat F_{n,\hat\b_n}}(t,x,\d)\,d(\P_n-P_0)(t,x,\d)=o_p(1),
	\end{align}
	by an application of Lemma \ref{lemma:equi}.

	Using (\ref{difference_f-fbar}),  $\|F_{\hat\b_n}-\hat F_{n,\hat\b_n}\|_2=O_p(n^{-1/3})$ and the Cauchy-Schwarz inequality on the second term we get,
	$$
	II_b = O_p(n^{-2/3}).
	$$
	
	The functions $\f_{\b}$ and $F_{\b}$ are of a simple parametric nature, since
	\begin{align*}
	\f_{\b}=\E(X|T-\b'X),
	\end{align*}
	and
	\begin{align*}
	F_{\b}(u)=\int F_0(u+(\b-\b_0)'x)f_{X|T-\b'X}(x|T-\b'X=u)\,dx,
	\end{align*}
	see (\ref{def_F_beta}).
	Moreover, since:
	\begin{align}
	\label{expansion_F_beta}
	F_{\hat\b_n}(u)&=F_0(u)+(\hat\b_n-\b_0)'\int x f_0(u)f_{X|T-\hat\b_n'X}(x|u)\,dx+o_p(\hat\b_n-\b_0)\nonumber\\
	&=F_0(u)+(\hat\b_n-\b_0)'f_0(u)\E\{X|T-\hat\b_n'X=u\}+o_p(\hat\b_n-\b_0),
	\end{align}
	and since the difference $\f_{\hat\b_n}-\bar\f_{\hat\b_n,\hat F_{n,\hat\b_n}}$ converges to zero, we get for the third term $II_c$:
	\begin{align*}
	II_c&= \int_{\hat F_{n,\hat\b_n}(t-\hat\b_n'x)\in[\ee,1-\ee]}\Bigl\{\f_{\hat\b_n}(t-\hat\b_n'x)-\bar\f_{\hat\b_n,\hat F_{n,\hat\b_n}}(t-\hat\b_n'x)\Bigr\}\\
	&\qquad\qquad\qquad\qquad\qquad\qquad\qquad\qquad\cdot\Bigl\{F_0(t-\b_0'x)-F_{\hat\b_n}(t-\hat\b_n'x)\Bigr\}\,dP_0(t,x,\d)\\
	&=o_p\left(\hat\b_n-\b_0\right).
	\end{align*}
	We therefore conclude,
	\begin{align*}
	\psi_{1,n}^{(\ee)}(\hat\b_n)&= I + o_p\left(n^{-1/2}+\hat\b_n-\b_0\right).
	\end{align*}
	We now write,
	\begin{align*}
	I&=\int_{\hat F_{n,\hat\b_n}(t-\hat\b_n'x)\in[\ee,1-\ee]}\Bigl\{x-\f_{\hat\b_n}(t-\hat\b_n'x)\Bigr\}\Bigl\{\d - \hat F_{n,\hat\b_n}(t-\hat\b_n'x)\Bigr\}\,d\P_n(t,x,\d)\\
	&=\int_{\hat F_{n,\hat\b_n}(t-\hat\b_n'x)\in[\ee,1-\ee]}\Bigl\{x-\f_{\hat\b_n}(t-\hat\b_n'x)\Bigr\}\Bigl\{\d - F_{\hat\b_n}(t-\hat\b_n'x)\Bigr\}\,d\P_n(t,x,\d)\\
	&\qquad+\int_{\hat F_{n,\hat\b_n}(t-\hat\b_n'x)\in[\ee,1-\ee]}\Bigl\{x-\f_{\hat\b_n}(t-\hat\b_n'x)\Bigr\}\\
	&\qquad\qquad\qquad\qquad\qquad\qquad\quad\cdot\Bigl\{F_{\hat\b_n}(t-\hat\b_n'x)-\hat F_{n,\hat\b_n}(t-\hat\b_n'x)\Bigr\}\,d\P_n(t,x,\d)\\
	&= I_a + I_b.
	\end{align*}

	We get:
	\begin{align*}
	&I_a\\
	&=\int_{\hat F_{n,\hat\b_n}(t-\hat\b_n'x)\in[\ee,1-\ee]}\Bigl\{x-\f_{\hat\b_n}(t-\hat\b_n'x)\Bigr\}\Bigl\{\d - F_{\hat\b_n}(t-\hat\b_n'x)\Bigr\}\,d\P_n(t,x,\d)\\
	&=\int_{\hat F_{n,\hat\b_n}(t-\hat\b_n'x)\in[\ee,1-\ee]}\Bigl\{x-\f_{\hat\b_n}(t-\hat\b_n'x)\Bigr\}\Bigl\{\d -F_{\hat\b_n}(t-\hat\b_n'x)\Bigr\}\,d(\P_n-P_0)(t,x,\d)\\
	&\quad+\int_{\hat F_{n,\hat\b_n}(t-\hat\b_n'x)\in[\ee,1-\ee]}\Bigl\{x-\f_{\hat\b_n}(t-\hat\b_n'x)\Bigr\}\Bigl\{F_0(t-\b_0'x) -F_{\hat\b_n}(t-\hat\b_n'x)\Bigr\}\,dP_0(t,x,\d).
	\end{align*}
	For the second integral on the right-hand side we get:
	\begin{align*}
	&\int_{\hat F_{n,\hat\b_n}(t-\hat\b_n'x)\in[\ee,1-\ee]}\Bigl\{x-\f_{\hat\b_n}(t-\hat\b_n'x)\Bigr\}\Bigl\{F_0(t-\b_0'x)-F_{\hat\b_n}(t-\hat\b_n'x)\Bigr\}\,dP_0(t,x,\d)\\
	&=\int_{\hat F_{n,\hat\b_n}(t-\hat\b_n'x)\in[\ee,1-\ee]}\Bigl\{x-\f_{\hat\b_n}(t-\hat\b_n'x)\Bigr\}\Bigl\{F_0(t-\hat\b_n'x) -F_{\hat\b_n}(t-\hat\b_n'x)\Bigr\}\,dP_0(t,x,\d) \\
	&\quad +\int_{\hat F_{n,\hat\b_n}(t-\hat\b_n'x)\in[\ee,1-\ee]}\Bigl\{x-\f_{\hat\b_n}(t-\hat\b_n'x)\Bigr\}\Bigl\{F_0(t-\b_0'x)-F_0(t-\hat\b_n'x)\Bigr\}\,dP_0(t,x,\d),
	\end{align*}
	and next we get, using the definition of $\f_\b$ given in (\ref{fi}), for the first integral on the right-hand side of the last display:
	\begin{align*}
	&\int_{\hat F_{n,\hat\b_n}(t-\hat\b_n'x)\in[\ee,1-\ee]}\Bigl\{x-\f_{\hat\b_n}(t-\hat\b_n'x)\Bigr\}\Bigl\{F_0(t-\hat\b_n'x)-F_{\hat\b_n}(t-\hat\b_n'x)\Bigr\}\,dP_0(t,x,\d) \\
	&=\int_{\hat F_{n,\hat\b_n}(u)\in[\ee,1-\ee]}\Bigl\{x-\f_{\hat\b_n}(u)\Bigr\}\Bigl\{F_0(u) - F_{\hat\b_n}(u)\Bigr\}
	f_{T-\hat\b_n'X}(u)\,f_{X|T-\hat\b_n'X}(x|u)\,du\,dx\\
	&=0.
	\end{align*}
	
	Furthermore, we get by expanding $F_0(t-\b'x)$ and by the continuity of $\b\mapsto\f_{\b}(u)$ at $\b=\b_0$
	\begin{align*}
	&\int_{\hat F_{n,\hat\b_n}(t-\hat\b_n'x)\in[\ee,1-\ee]}\Bigl\{x-\f_{\hat\b_n}(t-\hat\b_n'x)\Bigr\}\Bigl\{F_0(t-\b_0'x)-F_0(t-\hat\b_n'x)\Bigr\}\,dP_0(t,x,\d)\\
	&=\int_{\hat F_{n,\hat\b_n}(t-\hat\b_n'x)\in[\ee,1-\ee]}\Bigl\{x-\f_{\hat\b_n}(t-\hat\b_n'x)\Bigr\}(\hat\b_n-\b_0)'xf_0(t-\b_0'x)\,dP_0(t,x,\d)\\
	&\qquad
	+o_p\left(\hat\b_n-\b_0\right)\\
	&=\left\{\int_{\hat F_{n,\hat\b_n}(t-\hat\b_n'x)\in[\ee,1-\ee]}\Bigl\{x-\f_0(t-\b_0'x)\Bigr\}x'f_0(t-\b_0'x)\,dP_0(t,x,\d)\right\}(\hat\b_n-\b_0)\\
	&\qquad+o_p\left(\hat\b_n-\b_0\right).
	\end{align*}
	Finally we get from the consistency of $\hat F_{n,\hat\b_n}$:
	\begin{align*}
	&\left\{\int_{\hat F_{n,\hat\b_n}(t-\hat\b_n'x)\in[\ee,1-\ee]}\Bigl\{x-\f_0(t-\b_0'x)\Bigr\}x'f_0(t-\b_0'x)\,dP_0(t,x,\d)\right\}\left(\hat\b_n-\b_0\right)\\
	&=\left\{\int_{F_0(t-\b_0'x)\in[\ee,1-\ee]}\Bigl\{x-\f_0(t-\b_0'x)\Bigr\}x'f_0(t-\b_0'x)\,dP_0(t,x,\d)\right\}\left(\hat\b_n-\b_0\right)\\
	&\qquad+o_p\left(\hat\b_n-\b_0\right)\\
	&= \psi_{1,\ee}'(\b_0)(\hat\b_n - \b_0)+o_p\left(\hat\b_n-\b_0\right).
	\end{align*}
	So we obtain:
	\begin{align*}
	I_a&=\int_{\hat F_{n,\hat\b_n}(t-\hat\b_n'x)\in[\ee,1-\ee]}\Bigl\{x-\f_{\hat\b_n}(t-\hat\b_n'x)\Bigr\}\Bigl\{\d - F_{\hat\b_n}(t-\hat\b_n'x)\Bigr\}\,d(\P_n-P_0)(t,x,\d)\\
	&\qquad+ \psi_{1,\ee}'(\b_0)(\hat\b_n - \b_0)+o_p\left(\hat\b_n-\b_0\right).
	\end{align*}
	
	We now proceed again as before, and define 
	${\cal K}_1'$ to be the set of functions
	\begin{align*}
	{\cal K}_1'&=\left\{(t,x,\d)\mapsto\bigl(x-\f_{\b}(t-\b'x)\bigr)\bigl(\d -F_{\b}(t-\b'x)\bigr)1_{[\ee,1-\ee]}(F(t-\b'x))\right.\\
	&\left.\qquad\qquad\qquad\qquad\qquad\qquad\qquad\qquad\qquad\qquad\qquad\ : F\in{\cal F},\,\b\in\Theta\right\}.
	\end{align*}
	We add the function
	\begin{align*}
	(t,x,\d)\mapsto \bigl(x-\f_0(t-\b_0'x)\bigr)\bigl(\d -F_0(t-\b_0'x)\bigr)1_{[\ee,1-\ee]}(F_0(t-\b_0'x))\
	\end{align*}
	to the set ${\cal K}_1'$.	 The distance $d$ is defined by (\ref{distance_cal_K}) again, with ${\cal K}$ replaced by ${\cal K}_1'$.
	We therefore get, similarly as before, using Lemma \ref{lemma:bracketing-entropy},
	\begin{align*}
	\sup_{\zeta>0} \zeta H_B\left(\zeta,{\K_1'},L_2(P_0)\right)=O(1),
	\end{align*}
	which implies:
	\begin{align*}
	\int_0^{\zeta} H_B\left(u,{\K_1'},L_2(P_0)\right)^{1/2}\,du=O\left(\zeta^{1/2}\right),\qquad\zeta>0.
	\end{align*}
	Moreover, we get:
	\begin{align*}
	&\bigl(x-\f_{\hat\b_n}(t-\hat\b_n'x)\bigr)\bigl(\d -F_{\hat\b_n}(t-\hat\b_n'x)\bigr)1_{[\ee,1-\ee]}\left(\hat F_{n,\hat\b_n}(t-\hat\b_n'x)\right)\\
	&\qquad\qquad-\bigl(x-\f_0(t-\b_0'x)\bigr)\bigl(\d -F_0(t-\b_0'x)\bigr)1_{[\ee,1-\ee]}\left(F_0(t-\b_0'x)\right)\\
	&=\Bigl\{\bigl(x-\f_{\hat\b_n}(t-\hat\b_n'x)\bigr)\bigl(\d -F_{\hat\b_n}(t-\hat\b_n'x)\bigr)\\
	&\qquad\qquad-\bigl(x-\f_0(t-\b_0'x)\bigr)\bigl(\d -F_0(t-\b_0'x)\bigr)\Bigr\}1_{[\ee,1-\ee]}\left(\hat F_{n,\hat\b_n}(t-\hat\b_n'x)\right)\\
	&\qquad +\bigl(x-\f_0(t-\b_0'x)\bigr)\bigl(\d -F_0(t-\b_0'x)\bigr)\\
	&\qquad\qquad\qquad\qquad\qquad\cdot\left\{
	1_{[\ee,1-\ee]}\left(F_0(t-\b_0'x)\right)-1_{[\ee,1-\ee]}\left(\hat F_{n,\hat\b_n}(t-\hat\b_n'x)\right)\right\}\\
	&=A_n(t,x,\d)+B_n(t,x,\d),
	\end{align*}
	implying
	\begin{align*}
	&\int\Bigl\{\bigl(x-\f_{\hat\b_n}(t-\hat\b_n'x)\bigr)\bigl(\d -F_{\hat\b_n}(t-\hat\b_n'x)\bigr)1_{[\ee,1-\ee]}\left(\hat F_{n,\hat\b_n}(t-\hat\b_n'x)\right)\\
	&\qquad-\bigl(x-\f_0(t-\b_0'x)\bigr)\bigl(\d -F_0(t-\b_0'x)\bigr)1_{[\ee,1-\ee]}\left(F_0(t-\b_0'x)\right)\Bigr\}^2\,dP_0(t,x,\d)\\
	&\le 2\int\left\{A_n(t,x,\d)^2+B_n(t,x,\d)^2\right\}\,dP_0(t,x,\d)=o_p(1),
	\end{align*}
	since the integrals w.r.t.\ $A_n^2$ and $B_n^2$ tends to zero using the consistency of $\hat \b_n$ and $\hat F_{n,\hat\b_n}$.

	Hence we get from Lemma \ref{lemma:equi}:
	\begin{align*}
	I_a&=\int_{ F_{0}(t-\b_0'x)\in[\ee,1-\ee]}\Bigl\{x-\f_0(t-\b_0'x)\Bigr\}\Bigl\{\d -F_0(t-\b_0'x)\Bigr\}\,d(\P_n-P_0)(t,x,\d)\\
	&\qquad+ \psi_{1,\ee}'(\b_0)(\hat\b_n - \b_0)+o_p\left(\hat\b_n-\b_0\right)+o_p\left(n^{-1/2}\right).
	\end{align*}
	This means that we get the conclusion
	\begin{align}
	\label{conclusion_Th4.1}
	&\int_{F_0(t-\b_0'x)\in[\ee,1-\ee]}\Bigl\{x-\f_0(t-\b_0'x)\Bigr\}\Bigl\{\d -F_0(t-\b_0'x)\Bigr\}\,d\left(\P_n-P_0\right)(t,x,\d)\nonumber\\
	&=-\psi_{1,\ee}'(\b_0)(\hat\b_n - \b_0) + o_p\left(\hat\b_n - \b_0 \right)+  o_p\left( n^{-1/2}\right),
	\end{align}
	if we can show that $I_b$ is negligible.
	
	Since, by definition of $\f_\b$ given in (\ref{fi}),
	$$ 	\int_{\hat F_{n,\hat\b_n}(t-\hat\b_n'x)\in[\ee,1-\ee]}\Bigl\{x-\f_{\hat\b_n}(t-\hat\b_n'x)\Bigr\}\,f_{X|T-\hat\b_n'X}(x|t-\hat\b_n'x)\,dx=0, 	$$
	we have
	\begin{align*}
	I_b&= \int_{\hat F_{n,\hat\b_n}(t-\hat\b_n'x)\in[\ee,1-\ee]}\Bigl\{x-\f_{\hat\b_n}(t-\hat\b_n'x)\Bigr\}\\
	&\qquad\qquad\qquad\qquad\qquad\cdot\Bigl\{F_{\hat\b_n}(t-\hat\b_n'x)-\hat F_{n,\hat\b_n}(t-\hat\b_n'x)\Bigr\}\,d(\P_n-P_0)(t,x,\d).
	\end{align*}
	The negligibility of $I_b$ now follows in the same way as (\ref{II_a}), using the parametric nature of the function $\f_{\b}$ and the entropy properties of the class of functions
	$$
	u\mapsto \hat F_{n,\hat\b_n}(u)-F_{\hat\b_n}(u).
	$$
	The conclusion now follows from (\ref{conclusion_Th4.1}).
\end{proof}

\begin{remark}
	{\rm Note that the proof above yields the representation
		\begin{align*}
		&\hat\b_n-\b_0\\
		&\sim n^{-1}\psi_{1,\ee}'(\b_0)^{-1}\sum_{i=1}^n (X_i-\E(X_i|T-\b_0'X))\left\{\dd_i-F_0(T_i-\b_0'X_i\right\},
		\end{align*}
		where $\psi_{1,\ee}'(\b_0)$ is given by (\ref{derivative_score}).
	}
\end{remark}

\section{Asymptotic behavior of the efficient estimate based on  the MLE $\hat F_{n,\b}$}
In this section we prove the asymptotic efficiency of the score estimator defined in Section \ref{subsection:MLE2}. The proof of existence of the root and the consistency proof for the score estimator is similar to the proof of existence and consistency of the first score estimator defined in Section \ref{subsection:MLE}, thus omitted.

\subsection{Asymptotic normality of the efficient score estimator}
\begin{proof}[Proof of Theorem \ref{theorem:method2} (Asymptotic Normality)]
	Since the proof is very similar to the proof of Theorem \ref{theorem:method1}, we only give the main steps of the proof. 
	As in the proof of Theorem \ref{theorem:method1}, we can define $\psi_{2,nh}^{(\ee)}$ at $\hat\b_n$ by
	\begin{align*}
	\psi_{2,nh}^{(\ee)}(\hat\b_n)=0,
	\end{align*}
	and $\psi_{2,nh}^{(\ee)}(\hat\b_n)$ is then a combination of one-sided limits at $\hat\b_n$.
	
	We prove that:
	\begin{align}
	\label{decomposition2}
	&\psi_{2,nh}^{(\ee)}(\hat \b_n)  \\
	&= \int_{F_0(t-\b_0'x)\in[\ee,1-\ee]}\frac{\left\{xf_0(t-\b_0'x)-\varphi_{\b_0}(t-\b_0'x) \right\}\left\{\d - F_0(t-\b_0'x)\right\}}{F_0(t-\b_0'x)\{1-F_0(t-\b_0'x)\}}\,d\P_n(t,x,\d) \nonumber\\
	&\quad+ \psi_{2,\ee}'(\b_0)(\hat \b_n - \b_0) + o_p\left(n^{-1/2} + (\hat \b_n - \b_0)\right),\nonumber
	\end{align}
	where $\varphi_\b$ is defined by
	\begin{align}
	\label{phi2}
	\varphi_{\b}(t-\b'x) = \E(X|T-\b'X = t- \b'x)f_\b(t-\b'x),
	\end{align}
	and $\psi_{2,\ee}$ is defined by,
	\begin{align}
	\label{psi_2_pop}
	&\psi_{2,\ee}(\b) \\
	&= \int_{F_\b(t-\b'x)\in[\ee,1-\ee]}\frac{\Bigl\{xf_\b(t-\b'x)-\varphi_{\b}(t-\b'x) \Bigr\}\Bigl\{\d -F_\b(t-\b'x)\Bigr\}}{F_\b(t-\b'x)\{1-F_\b(t-\b'x)\}}\,dP_0(t,x,\d).  \nonumber
	\end{align}
	Straightforward calculations show that,
	\begin{align*}
	\psi_{2,\ee}'(\b_0) &= \int_{F_0(t-\b_0'x)\in[\ee,1-\ee]}\frac{\left\{xf_0(t-\b_0'x)-\varphi_{\b_0}(t-\b_0'x) \right\}^2}{F_0(t-\b_0'x)\{1-F_0(t-\b_0'x)\}}\,dP_0(t,x,\d)\\
	&=\E_\ee\left\{  \frac{ f_0(T- \b_0'X)^2\,\left\{ X - \E(X| T- \b_0'X) \right\}\left\{ X - \E(X| T- \b_0'X) \right\}'}{F_0(T- \b_0'X)\{1-F_0(T- \b_0'X)\}} \right\} \\
	&= I_\ee(\b_0).
	\end{align*}
	(See also the derivation of the derivative $\psi_\ee'$ for the first score equation in the proof of Theorem \ref{theorem:method1}, Part 1).
	Since
	\begin{align*}
	&\sqrt n \int_{F_0(t-\b_0'x)\in[\ee,1-\ee]}\frac{\left\{xf_0(t-\b_0'x)-\varphi_{\b_0}(t-\b_0'x) \right\}\left\{\d - F_0(t-\b_0'x)\right\}}{F_0(t-\b_0'x)\{1-F_0(t-\b_0'x)\}}\,d\P_n(t,x,\d) \\
	&\stackrel{d}{\rightarrow} N(0,I_\ee(\b_0)),
	\end{align*}	
	(\ref{decomposition2}) implies, using the non-singularity of $\psi_{2,\ee}'(\b_0)$ and the consistency of $\hat \b_n$, 
	\begin{align*}
	&\sqrt n (\hat \b_n - \b_0)\\
	& = -\psi_{2,\ee}'(\b_0)^{-1}\Bigl\{ \sqrt n \int_{F_0(t-\b_0'x)\in[\ee,1-\ee]}\frac{xf_0(t-\b_0'x)-\varphi_{\b_0}(t-\b_0'x)}{F_0(t-\b_0'x)\{1-F_0(t-\b_0'x)\}}\\
	&\qquad\qquad\qquad\qquad\qquad\qquad\qquad\qquad\qquad \cdot\left\{\d -F_0(t-\b_0'x)\right\}\,d\P_n(t,x,\d)\Bigr\} \\
	&\qquad + o_p(1+ \sqrt n (\hat \b_n - \b_0))  \\
	& \stackrel{d}{\rightarrow} N\left(0,I_\ee(\b_0)^{-1}\right).
	\end{align*}
	
	Let, analogously to the start of the proof of Theorem \ref{theorem:method1}, $\bar\varphi_{\hat\b_n,\hat F_{n,\hat\b_n}}$ be a (random) piecewise constant version of $\varphi_{\hat\b_n}$, where, for a piecewise constant distribution function $F$ with finitely many jumps at $\t_1<\t_2<\dots$, the function $\bar\varphi_{\b,F}$ is defined in the following way.
	\begin{align}
	\label{def_bar_phi2}
	\bar\varphi_{\b,F}(u)=
	\left\{\begin{array}{lll}
	\varphi_{\b}(\t_i),\,&\mbox{ if }F_{\b}(u)>F(\t_i),\,u\in[\t_i,\t_{i+1}),\\
	\varphi_{\b}(s),\,&\mbox{ if }F_{\b}(u)=F(s),\mbox{ for some }s\in[\t_i,\t_{i+1}),\\
	\varphi_{\b}(\t_{i+1}),\,&\mbox{ if }F_{\b}(u)<F(\t_i),\,u\in[\t_i,\t_{i+1}).
	\end{array}
	\right.
	\end{align}
	We now have:
	\begin{align*}
	&\psi_{2,nh}^{(\ee)}(\b)\\
	&= \int_{\hat F_{n,\hat\b_n}(t-\hat\b_n'x)\in[\ee,1-\ee]}x f_{nh,\hat\b_n}(t-\hat\b_n'x)\frac{\d -\hat F_{n,\hat\b_n}(t-\hat\b_n'x)}{\hat F_{n,\hat\b_n}(t-\hat\b_n'x)\{1-\hat F_{n,\hat\b_n}(t-\hat\b_n'x)\}}\,d\P_n(t,x,\d) \nonumber
	\end{align*}
	\begin{align*}
	&= \int_{\hat F_{n,\hat\b_n}(t-\hat\b_n'x)\in[\ee,1-\ee]}\left\{x f_{nh,\hat\b_n}(t-\hat\b_n'x)-\varphi_{\hat\b_n}(t-\hat\b_n'x)\right\}\nonumber\\
	&\qquad\qquad\qquad\qquad\quad\qquad\cdot\frac{\d - \hat F_{n,\hat\b_n}(t-\hat\b_n'x)}{\hat F_{n,\hat\b_n}(t-\hat\b_n'x)\{1-\hat F_{n,\hat\b_n}(t-\hat\b_n'x)\}}\,d\P_n(t,x,\d) \nonumber\\
	&\qquad+ \int_{\hat F_{n,\hat\b_n}(t-\hat\b_n'x)\in[\ee,1-\ee]}\left\{\varphi_{\hat\b_n}(t-\hat\b_n'x)-\bar\varphi_{n,\hat F_{n,\hat\b_n}}(t-\hat\b_n'x)\right\}\nonumber\\
	&\qquad\qquad\qquad\qquad\quad\qquad\cdot\frac{\d - \hat F_{n,\hat\b_n}(t-\hat\b_n'x)}{\hat F_{n,\hat\b_n}(t-\hat\b_n'x)\{1-\hat F_{n,\hat\b_n}(t-\hat\b_n'x)\}}\,d\P_n(t,x,\d)\nonumber\\
	&=I+II.
	\end{align*}
	Let ${\cal F}$ be the set of piecewise constant distribution functions with finitely many jumps (like the MLE $\hat F_{n,\hat\b_n}$), and let ${\cal K}_2$ be the set of functions
	\begin{align}
	\label{cal_K2}
	{\cal K}_2&=\left\{(t,x,\d)\mapsto\big\{\varphi_{\b}(t-\b'x)-\bar\varphi_{\b,F}(t-\b'x)\bigr\}\frac{\d -F(t-\b'x)\bigr)}{F(t-\b'x)\{1-F(t-\b'x)\}}\right.\nonumber\\
	&\left.\qquad\qquad\qquad\qquad\qquad\qquad\qquad\qquad\qquad\qquad\qquad\cdot1_{[\ee,1-\ee]}(F(t-\b'x)): F\in{\cal F},\,\b\in\Theta\right\},
	\end{align}
	where $\bar\varphi_{\b,F}$ is defined by (\ref{def_bar_phi2}). We add the function which is identically zero to ${\cal K}_2$. As in the proof of Theorem \ref{theorem:method1}, the functions are uniformly bounded and also of uniformly bounded variation, using conditions (A4) and (A5). For $k_1$ and $k_2$ in ${\cal K}_2$, we define
	\begin{align}
	\label{distance_cal_K2}
	d(k_1,k_2)^2=\int \left\|k_1-k_2\right\|^2\,dP_0,\qquad k_1,k_2\in{\cal K}_2.
	\end{align}
	For this distance, we therefore get, similarly as before, using Lemma \ref{lemma:bracketing-entropy},
	\begin{align*}
	\sup_{\zeta>0} \zeta H_B\left(\zeta,{\K_2},L_2(P_0)\right)=O(1),
	\end{align*}
	which implies:
	\begin{align*}
	\int_0^{\zeta} H_B\left(u,{\K_2},L_2(P_0)\right)^{1/2}\,du=O\left(\zeta^{1/2}\right),\qquad\zeta>0.
	\end{align*}
	Note that the indicator function keeps $F(t-\b'x)$ away from zero and one, which is essential for getting the bounded variation property.
	
	Following the same steps as in the proof of Theorem \ref{theorem:method1}, we get:
	\begin{align*}
	II=o_p\left(n^{-1/2}+\hat\b_n-\b_0\right).
	\end{align*}
	We now write,
	\begin{align*}
	I&=\int_{\hat F_{n,\hat\b_n}(t-\hat\b_n'x)\in[\ee,1-\ee]}\Bigl\{xf_{nh,\hat\b_n}(t-\hat\b_n'x)-\varphi_{\hat\b_n}(t-\hat\b_n'x)\Bigr\}\\
	&\qquad\qquad\qquad\qquad\qquad\qquad\qquad\qquad\qquad\cdot\frac{\d -\hat F_{n,\hat\b_n}(t-\hat\b_n'x)}{\hat F_{n,\hat\b_n}(t-\hat\b_n'x)\{1-\hat F_{n,\hat\b_n}(t-\hat\b_n'x)\}}\,d\P_n(t,x,\d)\\
	&=\int_{\hat F_{n,\hat\b_n}(t-\hat\b_n'x)\in[\ee,1-\ee]}\Bigl\{xf_{nh,\hat\b_n}(t-\hat\b_n'x)-\varphi_{\hat\b_n}(t-\hat\b_n'x)\Bigr\}\\
	&\qquad\qquad\qquad\qquad\qquad\qquad\qquad\qquad\qquad\cdot\frac{\d - F_{\hat\b_n}(t-\hat\b_n'x)}{\hat F_{n,\hat\b_n}(t-\hat\b_n'x)\{1-\hat F_{n,\hat\b_n}(t-\hat\b_n'x)\}}\,d\P_n(t,x,\d)\\
	&\qquad+\int_{\hat F_{n,\hat\b_n}(t-\hat\b_n'x)\in[\ee,1-\ee]}\Bigl\{xf_{nh,\hat\b_n}(t-\hat\b_n'x)-\varphi_{\hat\b_n}(t-\hat\b_n'x)\Bigr\}\\
	&\qquad\qquad\qquad\qquad\qquad\qquad\qquad\qquad\qquad\cdot\frac{F_{\hat\b_n}(t-\hat\b_n'x) - \hat F_{n,\hat\b_n}(t-\hat\b_n'x)}{\hat F_{\hat\b_n}(t-\hat\b_n'x)\{1-\hat F_{n,\hat\b_n}(t-\hat\b_n'x)\}}\,dP_0(t,x,\d)\\
	&\qquad+\int_{\hat F_{n,\hat\b_n}(t-\hat\b_n'x)\in[\ee,1-\ee]}\Bigl\{xf_{nh,\hat\b_n}(t-\hat\b_n'x)-\varphi_{\hat\b_n}(t-\hat\b_n'x)\Bigr\}\\
	&\qquad\qquad\qquad\qquad\qquad\qquad\qquad\qquad\qquad\cdot\frac{F_{\hat\b_n}(t-\hat\b_n'x) - \hat F_{n,\hat\b_n}(t-\hat\b_n'x)}{\hat F_{\hat\b_n}(t-\hat\b_n'x)\{1-\hat F_{n,\hat\b_n}(t-\hat\b_n'x)\}}\,d(\P_n-P_0)(t,x,\d)\\
	&= I_a + I_b +I_c.
	\end{align*}
	
For the term $I_b$ we get:
\begin{align*}
&\int_{\hat F_{n,\hat\b_n}(t-\hat\b_n'x)\in[\ee,1-\ee]}\Bigl\{xf_{nh,\hat\b_n}(t-\hat\b_n'x)-\varphi_{\hat\b_n}(t-\hat\b_n'x)\Bigr\}\\
	&\qquad\qquad\qquad\qquad\qquad\qquad\qquad\qquad\qquad\cdot\frac{F_{\hat\b_n}(t-\hat\b_n'x) - \hat F_{n,\hat\b_n}(t-\hat\b_n'x)}{\hat F_{\hat\b_n}(t-\hat\b_n'x)\{1-\hat F_{n,\hat\b_n}(t-\hat\b_n'x)\}}\,dP_0(t,x,\d)\\
&=\int_{\hat F_{n,\hat\b_n}(t-\hat\b_n'x)\in[\ee,1-\ee]}\left\{x-\E(X|T-\hat\b_n'X = t- \hat\b_n'x)\right\}f_{nh,\hat\b_n}(t-\hat\b_n'x)\\
	&\qquad\qquad\qquad\qquad\qquad\qquad\qquad\qquad\qquad\cdot\frac{F_{\hat\b_n}(t-\hat\b_n'x) - \hat F_{n,\hat\b_n}(t-\hat\b_n'x)}{\hat F_{\hat\b_n}(t-\hat\b_n'x)\{1-\hat F_{n,\hat\b_n}(t-\hat\b_n'x)\}}\,dP_0(t,x,\d)\\
&\quad+\int_{\hat F_{n,\hat\b_n}(t-\hat\b_n'x)\in[\ee,1-\ee]}\Bigl\{f_{nh,\hat\b_n}(t-\hat\b_n'x)-f_{\hat\b_n}(t-\hat\b_n'x)\Bigr\}\E(X|T-\hat\b_n'X = t- \hat\b_n'x)\\
&\qquad\qquad\qquad\qquad\qquad\qquad\qquad\qquad\qquad\cdot\frac{F_{\hat\b_n}(t-\hat\b_n'x) - \hat F_{n,\hat\b_n}(t-\hat\b_n'x)}{\hat F_{\hat\b_n}(t-\hat\b_n'x)\{1-\hat F_{n,\hat\b_n}(t-\hat\b_n'x)\}}\,dP_0(t,x,\d)\\
&=\int_{\hat F_{n,\hat\b_n}(t-\hat\b_n'x)\in[\ee,1-\ee]}\Bigl\{f_{nh,\hat\b_n}(t-\hat\b_n'x)-f_{\hat\b_n}(t-\hat\b_n'x)\Bigr\}\E(X|T-\hat\b_n'X = t- \hat\b_n'x)\\
&\qquad\qquad\qquad\qquad\qquad\qquad\qquad\qquad\qquad\cdot\frac{F_{\hat\b_n}(t-\hat\b_n'x) - \hat F_{n,\hat\b_n}(t-\hat\b_n'x)}{\hat F_{\hat\b_n}(t-\hat\b_n'x)\{1-\hat F_{n,\hat\b_n}(t-\hat\b_n'x)\}}\,dP_0(t,x,\d)\\
&=\int_{\hat F_{n,\hat\b_n}(u)\in[\ee,1-\ee]}\Bigl\{f_{nh,\hat\b_n}(u)-f_{\hat\b_n}(u)\Bigr\}\E(X|T-\hat\b_n'X =u)
\frac{F_{\hat\b_n}(u) - \hat F_{n,\hat\b_n}(u)}{\hat F_{\hat\b_n}(u)\{1-\hat F_{n,\hat\b_n}u)\}}\,f_{T-\hat\b_n'X}(u)\,du.
\end{align*}
Furthermore,
\begin{align*}
&=\int_{\hat F_{n,\hat\b_n}(u)\in[\ee,1-\ee]}\Bigl\{f_{nh,\hat\b_n}(u)-f_{\hat\b_n}(u)\Bigr\}\E(X|T-\hat\b_n'X =u)
\frac{F_{\hat\b_n}(u) - \hat F_{n,\hat\b_n}(u)}{\hat F_{\hat\b_n}(u)\{1-\hat F_{n,\hat\b_n}u)\}}\,f_{T-\hat\b_n'X}(u)\,du.\\
&=h^{-2}\int_{\hat F_{n,\hat\b_n}(u)\in[\ee,1-\ee]}\Bigl\{\int K'((u-v)/h)\hat F_{n,\hat\b_n}(v)\,dv-f_{\hat\b_n}(u)\Bigr\}\E(X|T-\hat\b_n'X =u)\\
&\qquad\qquad\qquad\qquad\qquad\qquad\qquad\qquad\qquad\qquad\qquad\cdot\frac{F_{\hat\b_n}(u) - \hat F_{n,\hat\b_n}(u)}{\hat F_{\hat\b_n}(u)\{1-\hat F_{n,\hat\b_n}(u)\}}\,f_{T-\hat\b_n'X}(u)\,du\\
&=h^{-2}\int_{\hat F_{n,\hat\b_n}(u)\in[\ee,1-\ee]}\int K'((u-v)/h)\left\{\hat F_{n,\hat\b_n}(v)-F_{\hat\b_n}(v)\right\}\,dv\,\E(X|T-\hat\b_n'X =u)\\
&\qquad\qquad\qquad\qquad\qquad\qquad\qquad\qquad\qquad\qquad\qquad\cdot\frac{F_{\hat\b_n}(u) - \hat F_{n,\hat\b_n}(u)}{\hat F_{\hat\b_n}(u)\{1-\hat F_{n,\hat\b_n}(u)\}}\,f_{T-\hat\b_n'X}(u)\,du\\
&\qquad\qquad+\int_{\hat F_{n,\hat\b_n}(u)\in[\ee,1-\ee]}\left\{\int K_h(u-v)\,dF_{\hat\b_n}(v)-f_{\hat\b_n}(u)\right\}\E(X|T-\hat\b_n'X =u)\\
&\qquad\qquad\qquad\qquad\qquad\qquad\qquad\qquad\qquad\qquad\qquad\qquad\cdot\frac{F_{\hat\b_n}(u) - \hat F_{n,\hat\b_n}(u)}{\hat F_{\hat\b_n}(u)\{1-\hat F_{n,\hat\b_n}(u)\}}\,f_{T-\hat\b_n'X}(u)\,du.
\end{align*}
The last term on the right-hand side has an upper bound of order $O_p(n^{-2/7-1/3})=O_p(n^{-13/21})$ $=o_p(n^{-1/2})$, since
$$
\left\{\int_{\hat F_{n,\hat\b_n}(u)\in[\ee,1-\ee]}\left\{\int K_h(u-v)\,dF_{\hat\b_n}(v)-f_{\hat\b_n}(u)\right\}^2\,du\right\}^{1/2}=O_p\left(n^{-2/7}\right),
$$
and
\begin{align}
\label{CS_Lemma3.1}
\left\{\int_{\hat F_{n,\hat\b_n}(u)\in[\ee,1-\ee]}\left\{\hat F_{n,\hat\b_n}(u)-F_{\hat\b_n}(u)\right\}^2\,du\right\}^{1/2}=O_p\left(n^{-1/3}\right),
\end{align}
using Lemma \ref{lemma:MLE_misspecified} for the last relation. We also use the Cauchy-Schwarz inequality.

The first term on the right is of order $O_p(n^{1/7-2/3})=O_p(n^{-11/21})
=o_p(n^{-1/2})$ by (\ref{CS_Lemma3.1}) and using
\begin{align*}
&\left|h^{-2}\int_{\hat F_{n,\hat\b_n}(u)\in[\ee,1-\ee]}\int K'((u-v)/h)\left\{\hat F_{n,\hat\b_n}(v)-F_{\hat\b_n}(v)\right\}\,dv\,\E(X|T-\hat\b_n'X =u)\right.\\
&\left.\qquad\qquad\qquad\qquad\qquad\qquad\qquad\qquad\qquad\qquad\qquad\cdot\frac{F_{\hat\b_n}(u) - \hat F_{n,\hat\b_n}(u)}{\hat F_{\hat\b_n}(u)\{1-\hat F_{n,\hat\b_n}(u)\}}\,f_{T-\hat\b_n'X}(u)\,du\right|\\
&=h^{-1}\left|\int_{\hat F_{n,\hat\b_n}(u)\in[\ee,1-\ee]}\int K'(w)\left\{\hat F_{n,\hat\b_n}(u-hw)-F_{\hat\b_n}(u-hw)\right\}\,dw\,\E(X|T-\hat\b_n'X =u)\right.\\
&\left.\qquad\qquad\qquad\qquad\qquad\qquad\qquad\qquad\qquad\qquad\qquad\cdot\frac{F_{\hat\b_n}(u) - \hat F_{n,\hat\b_n}(u)}{\hat F_{\hat\b_n}(u)\{1-\hat F_{n,\hat\b_n}(u)\}}\,f_{T-\hat\b_n'X}(u)\,du\right|\\
&\le ch^{-1}\int_{\hat F_{n,\hat\b_n}(u)\in[\ee/2,1-\ee/2]}\left\{\hat F_{n,\hat\b_n}(u)-F_{\hat\b_n}(u)\right\}^2\,du,
\end{align*}
for small $h$ and a constant $c>0$, where we first use Fubini's theorem and next the Cauchy-Schwarz inequality in the last inequality, together with
\begin{align*}
&\int_{\hat F_{n,\hat\b_n}(u)\in[\ee,1-\ee]} \left\{\hat F_{n,\hat\b_n}(u-hw)-F_{\hat\b_n}(u-hw)\right\}^2\,du\\
&\le \int_{\hat F_{n,\hat\b_n}(u)\in[\ee/2,1-\ee/2]} \left\{\hat F_{n,\hat\b_n}(u)-F_{\hat\b_n}(u)\right\}^2\,du,
\end{align*}
for small $h>0$, together with $w\in[-1,1]$.
Finally we use Lemma  \ref{lemma:MLE_misspecified} again.

	For the term $I_c$ we argue similarly as before using Lemma \ref{lemma:equi} that,
	\begin{align*}
		I_c=o_p\left(n^{-1/2}\right).
	\end{align*}
	Finally,
	\begin{align*}
	I_a&=\int_{\hat F_{n,\hat\b_n}(t-\hat\b_n'x)\in[\ee,1-\ee]}\Bigl\{xf_{nh,\hat\b_n}(t-\hat\b_n'x)-\varphi_{\hat\b_n}(t-\hat\b_n'x)\Bigr\}\\
	&\qquad\qquad\qquad\qquad\qquad\cdot\frac{\d - \hat F_{n,\hat\b_n}(t-\hat\b_n'x)}{\hat F_{n,\hat\b_n}(t-\hat\b_n'x)\{1-\hat F_{n,\hat\b_n}(t-\hat\b_n'x)\}}\,d\bigl(\P_n-P_0\bigr)(t,x,\d)\\
	&\qquad+\int_{\hat F_{n,\hat\b_n}(t-\hat\b_n'x)\in[\ee,1-\ee]}\Bigl\{xf_{nh,\hat\b_n}(t-\hat\b_n'x)-\varphi_{\hat\b_n}(t-\hat\b_n'x)\Bigr\}\\
	&\qquad\qquad\qquad\qquad\qquad\qquad\cdot\frac{F_0(t-\b_0'x)-\hat F_{n,\hat\b_n}(t-\hat\b_n'x)}{\hat F_{n,\hat\b_n}(t-\hat\b_n'x)\{1-\hat F_{n,\hat\b_n}(t-\hat\b_n'x)\}}\,dP_0(t,x,\d).
	\end{align*}
	This time we consider the class of functions
	\begin{align*}
	{\cal K}_2'&=\left\{(t,x,\d)\mapsto\bigl(xf(t-\b'x)-\varphi_{\b}(t-\b'x)\bigr)\frac{\d -F(t-\b'x)}{F(t-\b'x)\{1-F(t-\b'x)\}}1_{[\ee,1-\ee]}(F(t-\b'x))\right.\\
	&\left.\qquad\qquad\qquad\qquad\qquad\qquad\qquad\qquad\qquad\qquad\qquad\qquad\qquad\qquad : F\in{\cal F},\,f\in{\cal F}',\,\b\in\Theta\right\},
	\end{align*}
	where ${\cal F}'$ is a class of uniformly bounded functions of uniformly bounded variation (which have the interpretation of estimates of $F'_{\b}$), to which we add the function
	\begin{align*}
	(t,x,\d)\mapsto \bigl(xf_0(t-\b_0'x)-\varphi_{\b_0}(t-\b_0'x)\bigr)\frac{\d -F_0(t-\b_0'x)}{F_0(t-\b_0'x)\{1-F_0(t-\b_0'x)\}}1_{[\ee,1-\ee]}(F_0(t-\b_0'x)).
	\end{align*}
	So we get,  using Lemma \ref{lemma:bracketing-entropy},
		\begin{align*}
		\sup_{\zeta>0} \zeta H_B\left(\zeta,{\K_2'},L_2(P_0)\right)=O(1),
		\end{align*}
		which implies:
		\begin{align*}
		\int_0^{\zeta} H_B\left(u,{\K_2'},L_2(P_0)\right)^{1/2}\,du=O\left(\zeta^{1/2}\right),\qquad\zeta>0.
		\end{align*}
	As before, we now get:
	\begin{align*}
	&\int_{\hat F_{n,\hat\b_n}(t-\hat\b_n'x)\in[\ee,1-\ee]}\Bigl\{xf_{nh,\hat\b_n}(t-\hat\b_n'x)-\varphi_{\hat\b_n}(t-\hat\b_n'x)\Bigr\}\\
	&\qquad\qquad\qquad\qquad\qquad\qquad\cdot\frac{\d -\hat F_{n,\hat\b_n}(t-\hat\b_n'x)}{\hat F_{n,\hat\b_n}(t-\hat\b_n'x)\{1-\hat F_{n,\hat\b_n}(t-\hat\b_n'x)\}}\,d\bigl(\P_n-P_0\bigr)(t,x,\d)\\
	&= \int_{F_0(t-\b_0'x)\in[\ee,1-\ee]}\Bigl\{xf_0(t-\b_0'x)-\varphi_{\b_0}(t-\b_0'x)\Bigr\}\\
	&\qquad\qquad\qquad\qquad\qquad\qquad\cdot\frac{\d - F_0(t-\b_0'x)}{F_0(t-\b_0'x)\{1-F_0(t-\b_0'x)\}}\,d\bigl(\P_n-P_0\bigr)(t,x,\d)\\
	&\qquad\qquad+o_p\left(n^{-1/2}+\hat\b_n-\b_0\right)
	\end{align*}
	and
	\begin{align*}
	&\int_{\hat F_{n,\hat\b_n}(t-\hat\b_n'x)\in[\ee,1-\ee]}\Bigl\{xf_{nh,\hat\b_n}(t-\hat\b_n'x)-\varphi_{\hat\b_n}(t-\hat\b_n'x)\Bigr\}\\
	&\qquad\qquad\qquad\qquad\qquad\qquad\qquad\qquad\cdot\frac{F_0(t-\b_0'x)-\hat F_{n,\hat\b_n}(t-\hat\b_n'x)}{\hat F_{n,\hat\b_n}(t-\hat\b_n'x)\{1-\hat F_{n,\hat\b_n}(t-\hat\b_n'x)\}}\,dP_0(t,x,\d)\\
	&=\Biggl\{ \int_{F_0(t-\b_0'x)\in[\ee,1-\ee]}\Bigl\{xf_0(t-\b_0'x)-\varphi_{\b_0}(t-\b_0'x)\Bigr\}\\
	&\qquad\qquad\qquad\qquad\qquad\qquad\cdot\frac{f_0(t-\b_0'x)x'}{F_0(t-\b_0'x)\{1-F_0(t-\b_0'x)\}}\,dP_0(t,x,\d)\Biggr\}\left(\hat\b_n-\b_0\right)\\
	&\qquad\qquad+o_p\left(n^{-1/2}+\hat\b_n-\b_0\right).
	\end{align*}
	The result now follows.
\end{proof}

\section{Asymptotic behavior of the plug-in estimator}
 In this section we first sketch in Section \ref{section:normality} the proof of consistency of the plug-in estimator, denoted by $\hat \b_n$.  This is the second result stated in Theorem \ref{th:asymptotic_normality}. The proof of existence of a root is similar to the proof of existence of a root of the simple score estimator defined in Section \ref{subsection:MLE} and omitted. We next prove the asymptotic normality result of the plug-in estimator, which is the third result given in Theorem \ref{th:asymptotic_normality}. The proof of Theorem \ref{th:alternative_MLE-expansion} on the asymptotic representation of the plug-in estimator as a sum of i.i.d. random variables follows from the proof of \ref{th:asymptotic_normality}. The asymptotic distribution of the estimator of the intercept, given in Theorem \ref{th:intercept}, is proved in Section \ref{sectionA:intercept}.

Before we start the proofs, we give some auxiliary results on the $L_2$-distance between the plug-in estimate $F_{nh,\b}$ and $F_{\b}$ and between the partial derivative of the plug-in estimate $\partial_\b F_{nh,\b}$ and $\partial_\b F_{\b}$ in Lemma \ref{lemma:distance_F{nh}-F_0}. For simplicity, we derive the proof of Lemma \ref{lemma:distance_F{nh}-F_0} for the one-dimensional case and let $\Theta = [\b_0-\eta,\b_0+\eta]$ for some $\eta >0$. The higher-dimensional extension of the one-dimensional proof is straightforward. Next, we follow the arguments used to prove the asymptotic normality of the estimators defined in Theorem \ref{theorem:method1} and Theorem \ref{theorem:method2} and give a similar proof for the limiting distribution of the plug-in estimator.  

\begin{lemma}
\label{lemma:distance_F{nh}-F_0}
Let the conditions of Theorem \ref{th:asymptotic_normality} be satisfied and let $k=1$. Let the function $F_{\b}$ be defined by (\ref{def_F_beta}). Then we have, for the estimate $F_{nh,\b}$, defined by (\ref{plug_in_estimate}),
	\begin{align}
	\label{bound_F_{nh,beta}1}
	&\int_{F_{nh,\b}(t-\b x)\in[\ee,1-\ee]}\left\{F_{nh,\b}(t-\b x)-F_{\b}(t-\b x)\right\}^2\,dG(t,x)=O_p\left(\frac1{nh}\right)+O_p\left(h^4\right),
	\\ \nonumber \\
	&\int_{F_{nh,\b}(t-\b x)\in[\ee,1-\ee]}\left\{\partial_\b F_{nh,\b}(t-\b x)-\partial_\b F_\b(t-\b x)\right\}^2\,dG(t,x)= O_p\left(\frac1{nh^3}\right)+O_p\left(h^2\right)	\label{derivative_F_{nh,beta}_bound1}
	\end{align}
	uniformly in $\b\in[\b_0-\eta,\b_0+\eta]$.
	The results remain valid when $dG$ in (\ref{bound_F_{nh,beta}1}) or (\ref{derivative_F_{nh,beta}_bound1}) is replaced by $d\G_n$.
\end{lemma}

\begin{proof}[proof of Lemma \ref{lemma:distance_F{nh}-F_0}]
We first prove the first part and show that (\ref{bound_F_{nh,beta}1}) holds.
Recall that,
\begin{align*}
	F_{nh,\b}(t-\b x)=\frac{g_{nh,1,\b}(t-\b x)}{g_{nh,\b}(t-\b x)}
\end{align*}
where
\begin{align*}
g_{nh,1,\b}(t-\b x)=\int \d K_h(t-\b x-u+\b y)\,d\P_n(u,y,\d),
\end{align*}
and
\begin{align*}
g_{nh,\b}(t-\b x)=\int K_h(t-\b x-u+\b y)\,d\P_n(u,y,\d).
\end{align*}
Moreover,
\begin{align*}
F_{\b}(t-\b x)=\int F_0(t-\b_0 x+(\b-\b_0)(y-x))f_{X|T-\b X}(y|t-\b x)\,dy.
\end{align*}
We first investigate the bias part.
\begin{align*}
&\E g_{nh,1,\b}(t-\b x) = \int F_0(u-\b_0 y)K_h(t-\b x-u+\b y)\,dG(u,y)\\
&=\int F_0(v+(\b-\b_0)y)K_h(t-\b x-v)\,f_{T-\b X}(v)\,f_{X|T-\b X}(y|v)\,dy\,dv\\
&=\int F_0(t-\b x+(\b-\b_0)y-hw)K(w)\,f_{T-\b X}(t-\b x-hw)\,f_{X|T-\b X}(y|t-\b x-hw)\,dy\,dw\\
&=f_{T-\b X}(t-\b x)\int F_0(t-\b_0 x+(\b-\b_0)(y-x))f_{X|T-\b X}(y|t-\b x)\,dy +O\left(h^2\right),
\end{align*}
uniformly in $\b\in[\b_0-\eta,\b_0+\eta]$ and $t,x$ varying over a finite interval, due to the assumptions of Theorem 
\ref{th:asymptotic_normality}. In a similar way, we get
\begin{align*}
&\E g_{nh,\b}(t-\b x)=f_{T-\b X}(t-\b x)+O\left(h^2\right),
\end{align*}
uniformly in $\b\in[\b_0-\eta,\b_0+\eta]$ and $t,x$ varying over a finite interval.
So we find:
\begin{align*}
\frac{\E g_{nh,1,\b}(t-\b x)}{\E g_{nh,\b}(t-\b x)}=F_{\b}(t-\b x)+O\left(h^2\right).
\end{align*}
uniformly in $\b\in[\b_0-\eta,\b_0+\eta]$ and $t,x$ varying over a finite interval, such that $\E g_{nh,1,\b}(t-\b x)$ stays away from zero.
	
So we obtain
\begin{align*}
&F_{nh,\b}(t-\b x)-F_{\b}(t-\b x)\nonumber\\
&=\frac{g_{nh,1,\b}(t-\b x)-\E g_{nh,1,\b}(t-\b x)}{g_{nh,\b}(t-\b x)}
+\E g_{nh,1,\b}(t-\b x)\frac{\E g_{nh,\b}(t-\b x)-g_{nh,\b}(t-\b x)}{g_{nh,\b}(t-\b x)\E g_{nh,\b}(t-\b x)}+O\left(h^2\right),
\end{align*}
and
\begin{align}
\label{first_upbound_F_{nh}}
&\left\{F_{nh,\b}(t-\b x)-F_{\b}(t-\b x)\right\}^2\nonumber\\
&\le 3\left\{\frac{g_{nh,1,\b}(t-\b x)-\E g_{nh,1,\b}(t-\b x)}{g_{nh,\b}(t-\b x)}\right\}^2
+3\left\{\E g_{nh,1,\b}(t-\b x)\frac{\E g_{nh,\b}(t-\b x)-g_{nh,\b}(t-\b x)}{g_{nh,\b}(t-\b x)\E g_{nh,\b}(t-\b x)}\right\}^2\nonumber\\
&\qquad\qquad\qquad\qquad\qquad\qquad\qquad\qquad\qquad\qquad\qquad\qquad\qquad\qquad\qquad\qquad\qquad\qquad\qquad+O\left(h^4\right).
\end{align}
uniformly in $\b\in[\b_0-\eta,\b_0+\eta]$ and $t,x$ varying over a finite interval, such that $\E g_{nh,1,\b}(t-\b x)$ stays away from zero.
		
Since $\eta>0$ is chosen in such a way that $a_1(\b)=F_{\b}^{-1}(\ee)>a$, $b_1(\b)=F_{\b}^{-1}(1-\ee)<b$, for each $\b\in[\b_0-\eta,\b_0+\eta]$ and	
since $g_{nh,\b}$ stays away from zero with probability tending to one if $\ee<F_{nh,\b}(t-\b x)<1-\ee$ we get

\begin{align*}
&\int_{F_{nh,\b}(t-\b x)\in[\ee,1-\ee]}\left\{\frac{g_{nh,1,\b}(t-\b x)-\E g_{nh,1,\b}(t-\b x)}{g_{nh,\b}(t-\b x)}\right\}^2\,dG(t,x)\\
&\lesssim\int_{F_{nh,\b}(t-\b x)\in[\ee,1-\ee]}\left\{g_{nh,1,\b}(t-\b x)-\E g_{nh,1,\b}(t-\b x)\right\}^2\,dG(t,x)\\
\end{align*}
Furthermore
\begin{align*}
\E\left\{g_{nh,1,\b}(t-\b x)-\E g_{nh,1,\b}(t-\b x)\right\}^2
&=\E\left\{\int \d K_h(t-\b x-u+\b y)\,d(\P_n-P_0)(u,y,\d)\right\}^2 \\
&=O\left(\frac1{nh}\right),
\end{align*}
uniformly for $(t,x)$ in a bounded region, so we get
\begin{align*}
\E\int_{F_{nh,\b}(t-\b x)\in[\ee,1-\ee]}\left\{g_{nh,1,\b}(t-\b x)-\E g_{nh,1,\b}(t-\b x)\right\}^2\,dG(t,x)=O\left(\frac1{nh}\right).
\end{align*}
Hence
\begin{align*}
\int_{F_{nh,\b}(t-\b x)\in[\ee,1-\ee]}\left\{\frac{g_{nh,1,\b}(t-\b x)-\E g_{nh,1,\b}(t-\b x)}{g_{nh,\b}(t-\b x)}\right\}^2\,dG(t,x)
=O_p\left(\frac1{nh}\right).
\end{align*}
The second term on the right-hand side of (\ref{first_upbound_F_{nh}}) can be treated in a similar way. So we get (\ref{bound_F_{nh,beta}1}).
This proves (\ref{bound_F_{nh,beta}1}).  

We next replace $dG$ in part (\ref{bound_F_{nh,beta}1}) by $d\G_n$ and we get
\begin{align*}
&\int_{F_{nh,\b}(t-\b x)\in[\ee,1-\ee]}\left\{\frac{g_{nh,1,\b}(t-\b x)-\E g_{nh,1,\b}(t-\b x)}{g_{nh,\b}(t-\b x)}\right\}^2\,d\G_n(t,x)\\
&\lesssim\int_{F_{nh,\b}(t-\b x)\in[\ee,1-\ee]}\left\{g_{nh,1,\b}(t-\b x)-\E g_{nh,1,\b}(t-\b x)\right\}^2\,d\G_n(t,x)\\
&=\frac1n\sum_{i=1}^n\left\{g_{nh,1,\b}(T_i-\b X_i)-\E g_{nh,1,\b}(T_i-\b X_i)\right\}^21_{\{\ee<F_{nh,\b}(T_i-\b X_i)<1-\ee\}}.
\end{align*}
Moreover,
\begin{align*}
&\E\frac1n\sum_{i=1}^n\left\{g_{nh,1,\b}(T_i-\b X_i)-\E g_{nh,1,\b}(T_i-\b X_i)\right\}^21_{\{\ee<F_{nh,\b}(T_i-\b X_i)<1-\ee\}}\\
&=\E\left\{g_{nh,1,\b}(T_1-\b X_1)-\E g_{nh,1,\b}(T_1-\b X_1)\right\}^21_{\{\ee<F_{nh,\b}(T_1-\b X_1)<1-\ee\}}\\
&\lesssim\E\int_{\ee/2<F_{\b}(t-\b x)<1-\ee/2}\left\{g_{nh,1,\b}(t-\b x)-\E g_{nh,1,\b}(t-\b x)\right\}^2\,dG(t,x)\\
&=O\left(\frac1{nh}\right).
\end{align*}
This implies by the Markov inequality,
\begin{align*}
&\int_{F_{nh,\b}(t-\b x)\in[\ee,1-\ee]}\left\{\frac{g_{nh,1,\b}(t-\b x)-\E g_{nh,1,\b}(t-\b x)}{g_{nh,\b}(t-\b x)}\right\}^2\,d\G_n(t,x)
=O_p\left(\frac1{nh}\right).
\end{align*}
The other term on the right-hand side of (\ref{first_upbound_F_{nh}}) is treated similarly; and the result of (\ref{bound_F_{nh,beta}1}) also follows when we replace $dG$ by $d\G_n$.
\\
\\
We next continue with the proof of (\ref{derivative_F_{nh,beta}_bound1}).
		
We have:
\begin{align}
\label{partial_beta}
\partial_\b F_{nh,\b}(t-\b x)
=\frac{\int(y-x)\{\d-F_{nh,\b}(t-\b x)\}K_h'(t-\b x-u+\b y)\,d\P_n(u,y,\d)}{g_{nh,\b}(t-\b x)}\,.
\end{align}    
We consider the numerator of (\ref{partial_beta}). It can be rewritten as
\begin{align*}
&\int(y-x)\{\d-F_0(u-\b_0y)\}K_h'(t-\b x-u+\b y)\,d\P_n(u,y,\d)\\
& \qquad +\int(y-x)\{F_0(u-\b_0y)-F_{\b}(t-\b x)\}K_h'(t-\b x-u+\b y)\,d\G_n(u,y)\\
& \qquad +\{F_{\b}(t-\b x)-F_{nh,\b}(t-\b x)\}\int(y-x)K_h'(t-\b x-u+\b y)\,d\G_n(u,y).  
\end{align*}
The first term can be written as
\begin{align*}
A_n(t,x,\b)\stackrel{\text{\small def}}=\int(y-x)\{\d-F_0(u-\b_0y)\}K_h'(t-\b x-u+\b y)\,d\bigl(\P_n-P_0\bigr)(u,y,\d),
\end{align*}
and we have:
\begin{align*}
&\E\int_{F_{nh,\b}(t-\b x)\in[\ee,1-\ee]}A_n(t,x,\b)^2\,dG(t,x)\le \E\int A_n(t,x,\b)^2\,dG(t,x)\\
&\sim\frac1{nh^3}\int\text{var}(X|v)F_0(v)\{1-F_0(v)\}f_{T-\b X}(v)\,dv\int K'(u)^2\,du,\,n\to\infty.
\end{align*}
In the second term we must compare $F_0(u-\b_0y)$ with
$$
F_{\b}(t-\b x)=\int F_0(t-\b_0 x+(\b-\b_0)(z-x))f_{X|T-\b X}(z|t-\b x)\,dz.
$$
We can write
\begin{align*}
&F_0(u-\b_0y)-F_{\b}(t-\b x)\\
&=\int \left\{F_0(u-\b_0y)-F_0(t-\b_0 x+(\b-\b_0)(z-x))\right\}f_{X|T-\b X}(z|t-\b x)\,dz.
\end{align*}
So we find for the second term
\begin{align*}
&B_n(t,x,\b)\stackrel{\text{\small def}}=\int(y-x)\left\{F_0(u-\b_0y)-F_{\b}(t-\b x)\right\}K_h'(t-\b x-u+\b y)\,d\G_n(u,y)\\
&=\int \int(y-x)\left\{F_0(u-\b_0y)-F_0(t-\b_0 x+(\b-\b_0)(z-x))\right\}f_{X|T-\b X}(z|t-\b x)\,dz\\
&\qquad\qquad\qquad\qquad\qquad\qquad\qquad\qquad\qquad\qquad\qquad\qquad\cdot K_h'(t-\b x-u+\b y)\,d\G_n(u,y)\\
&=\int (y-x)\int\left\{F_0(u-\b_0y)-F_0(t-\b_0 x+(\b-\b_0)(z-x))\right\}f_{X|T-\b X}(z|t-\b x)\,dz\\
&\qquad\qquad\qquad\qquad\qquad\qquad\qquad\qquad\qquad\qquad\qquad\qquad\cdot K_h'(t-\b x-u+\b y)\,dG(u,y)\\
&\qquad+\int (y-x)\int\left\{F_0(u-\b_0y)-F_0(t-\b_0 x+(\b-\b_0)(z-x))\right\}f_{X|T-\b X}(z|t-\b x)\,dz\\
&\qquad\qquad\qquad\qquad\qquad\qquad\qquad\qquad\qquad\qquad\qquad\qquad\cdot K_h'(t-\b x-u+\b y)\,d\bigl(\G_n-G\bigr)(u,y)\\
&=f_{T-\b X}(t-\b x)\partial_\b F_{\b}(t-\b x) + O(h) + O_p\left(\frac1{nh^3}\right).
\end{align*}
where, using integration by parts, the last line follows by straightforward calculation.
Since
\begin{align*}
g_{nh,\b}(t-\b x) = f_{T-\b X}(t - \b x) + O_p(h^2),
\end{align*}
we get, 
\begin{align*}
&\int_{F_{nh,\b}(t-\b x)\in[\ee,1-\ee]}\left\{ \frac{B_n(t,x,\b)}{g_{nh,\b}(t-\b x)}-\partial_\b F_\b(t-\b x)\right\}^2\,dG(t,x)=O_p\left(\frac1{nh^3}\right)+O_p\left(h^2\right).
\end{align*}
Finally, defining
\begin{align*}
C_n(t,x,\b)\stackrel{\text{\small def}}=\{F_{\b}(t-\b x)-F_{nh,\b}(t-\b x)\}\int(y-x)K_h'(t-\b x-u+\b y)\,d\G_n(u,y),
\end{align*}
we get, using,
\begin{align*}
&\int(y-x)K_h'(t-\b x-u+\b y)\,d\G_n(u,y)\\
&=\int(y-x)K_h'(t-\b x-u+\b y)\,dG(u,y)+\int(y-x)K_h'(t-\b x-u+\b y)\,d(\G_n-G)(u,y)\\
&=\int (y-x)K_h'(t-\b x-v)f_{T-\b X}(v)f_{X|T-\b X}(y|v)\,dv\,dy+O_p\left(\frac1{nh^3}\right)\\
&=\int (y-x)K_h(t-\b x-v)\frac{d}{dv}\left\{f_{T-\b X}(v)f_{X|T-\b X}(y|v)\right\}\,dv\,dy+O_p\left(\frac1{nh^3}\right)\\
&=O_p(1),
\end{align*}
and using the first part of Lemma \ref{lemma:distance_F{nh}-F_0} for the factor $F_{\b}(t-\b x)-F_{nh,\b}(t-\b x)$ that
	\begin{align*} 
	\int_{F_{nh,\b}(t-\b x)\in[\ee,1-\ee]}C_n(t,x,\b)^2\,dG(t,x)=O_p\left(\frac1{nh}\right)+O_p\left(h^4\right).
	\end{align*}
	This proves  (\ref{derivative_F_{nh,beta}_bound1}). The second part of the result, replacing $dG$ by $d\G_n$ in (\ref{derivative_F_{nh,beta}_bound1}) is proved in the same way as the second part of (\ref{bound_F_{nh,beta}1}).
\end{proof}

\subsection{Consistency and asymptotic normality of the plug-in estimator}
\label{section:normality}
We first prove that $\hat \b_n$ is a consistent estimate of $\b_0$.
\begin{proof}[Proof of Theorem \ref{th:asymptotic_normality}, Part 1 (Consistency of the plug-in estimator)]
	We assume that $\hat\b_n$ is contained in the compact set $\Theta$, and hence the sequence $(\hat\b_n)$ has a subsequence $(\hat\b_{n_k}=\hat\b_{n_k}(\omega))$, converging to an element $\b_*$. It is easily seen that, if $\hat\b_{n_k}=\hat\b_{n_k}(\omega)\longrightarrow \b_*$, we get:
	\begin{align*}
	F_{n_kh,\hat\b_{n_k}}(t-\hat\b_{n_k}'x)\longrightarrow  F_{\b_*}(t-\b_*'x)\stackrel{\text{\small def}}=\int F_0(t-\b_*'x+(\b_*-\b_0)'y)f_{X|T-\b_*'X}(y|t-\b_*'x)\,dy.
	\end{align*}
	
	In the limit  we get therefore the relation
	\begin{align}
	&\lim_{k\to\infty}-(\hat \b_{n_k}-\b_0)'\\
	&\qquad\int_{F_{n_k h,\hat\b_{n_k}}(t-\hat\b_{n_k}'x)\in[\ee,1-\ee]}
	\frac{\bigl\{\d-F_{n_k h,\hat\b_{n_k}}(t-\hat\b_{n_k}'x)\bigr\}\partial_{\b}F_{{n_k} h,\b}(t-\b'x)\bigr|_{\b=\hat\b_{n_k}}}{F_{n_k,\hat\b_{n_k}}(t-\hat\b_{n_k}'x)\{1-F_{{n_k},\hat\b_{n_k}}(t-\hat\b_{n_k}'x)\}}\,\,d\P_{n_k}(t,x,\d)
	\nonumber\\
	&=-\left(\b_*-\b_0\right)'\int_{F_{\b_*}(t-\b_*'x)\in[\ee,1-\ee]}\frac{\bigl\{F_0(t-\b_0'x) -F_{\b_*}(t-\b_*'x)\bigr\}\partial_{\b}F_{\b}(t-\b'x)|_{\b=\b_*}}{F_{\b_*}(t-\b_*'x)\{1-F_{\b_*}(t-\b_*'x)\}}\,dG(t,x)=0,
	\label{consistency_ineq}
	\end{align}	
	which can only mean $\b_*=\b_0$ by condition (\ref{identifiability_plugin}).
	
\end{proof}
We next continue with the proof of the asymptotic normality of the plug-in estimator.
\begin{proof}[Proof of Theorem \ref{th:asymptotic_normality}, Part 2 (Asymptotic Normality of the plug-in estimator)]
To prove the asymptotic normality of the plug-in estimator, we follow the reasoning of the corresponding proofs of the simple score estimator and the efficient score estimator described in  Section \ref{subsection:MLE} and  Section \ref{subsection:MLE2}. We prove that,
\begin{align}
\label{decomposition3}
&\psi_{3,nh}^{(\ee)}(\hat \b_n)  \\
&= \int_{F_0(t-\b_0'x)\in[\ee,1-\ee]}\frac{\left\{E(X|T-\b_0'X = t- \b_0'x) -x\right\}f_0(t-\b_0'x)\left\{\d - F_0(t-\b_0'x)\right\}}{F_0(t-\b_0'x)\{1-F_0(t-\b_0'x)\}}\,d\P_n(t,x,\d) \nonumber\\
&\quad+ \psi_{3,\ee}'(\b_0)(\hat \b_n - \b_0) + o_p\left(n^{-1/2} + (\hat \b_n - \b_0)\right),\nonumber
\end{align}
where $\psi_{3,\ee}$ is defined by,
\begin{align}
\label{psi_3-pop}
&\psi_{3,\ee}(\b) = \int_{F_\b(t-\b'x)\in[\ee,1-\ee]}\frac{\partial_\b F_{\b}(t-\b x)\Bigl\{\d - F_\b(t-\b'x)\Bigr\}}{F_\b(t-\b'x)\{1-F_\b(t-\b'x)\}}\,dP_0(t,x,\d),
\end{align}
and
\begin{align*}
\psi_{3,\ee}'(\b_0) 
&=-\E_\ee\left\{  \frac{ f_0(T- \b_0'X)^2\,\left\{ X - \E(X| T- \b_0'X) \right\}\left\{ X - \E(X| T- \b_0'X) \right\}'}{F_0(T- \b_0'X)\{1-F_0(T- \b_0'X)\}} \right\} =- I_\ee(\b_0),
\end{align*}
which follows by straightforward calculations after noting that,
\begin{align*}
\partial_\b F_{\b}(t-\b' x) &= \int(y-x)f_0(t- \b_0' x + (\b-\b_0)'(y-x) )f_{X|T- \b' X}(y|T-\b' X=t-\b' x)dy\\
&\qquad+\int F_0(t- \b_0' x + (\b-\b_0)'(y-x) ) \partial_\b f_{X|T-\b' X}(y|T-\b'X=t-\b' x)\,dG(t,x)
\end{align*}
is, at $\b =\b_0$ equal to
$$f_0(t- \b_0' x)\E\left\{X-x|T- \b_0' X=t- \b_0 'x\right\} .$$

We have,
\begin{align*}
&\psi_{3,nh}^{(\ee)}(\hat\b_n) 
\\&=\int_{F_{nh,\hat\b_n}(t-\hat\b_n'x)\in[\ee,1-\ee]}\partial_\b F_{nh,\b}(t-\b'x)\mid_{\b = \hat\b_n}\frac{\d - F_{nh,\hat\b_n}(t-\hat\b_n'x)}{F_{nh,\hat\b_n}(t-\hat\b_n'x)\{1-F_{nh,\hat\b_n}(t-\hat\b_n'x)\}}\,d\P_n(t,x,\d)\\
&=\int_{F_{nh,\hat\b_n}(t-\hat\b_n'x)\in[\ee,1-\ee]}\partial_\b F_{\b}(t-\b'x)\mid_{\b = \hat\b_n}\frac{\d -F_{nh,\hat\b_n}(t-\hat\b_n'x)}{F_{nh,\hat\b_n}(t-\hat\b_n'x)\{1-F_{nh,\hat\b_n}(t-\hat\b_n'x)\}}\,d\P_n(t,x,\d)
\\
&\qquad\qquad + \int_{F_{nh,\hat\b_n}(t-\hat\b_n'x)\in[\ee,1-\ee]}\left\{\partial_\b F_{nh,\b}(t-\b'x)\mid_{\b =\hat\b_n} - \partial_\b F_{\b}(t-\b'x)\mid_{\b = \hat\b_n}\right\}\\
&\qquad\qquad\qquad\qquad\qquad\qquad\qquad\cdot\frac{\d - F_{nh,\hat\b_n}(t-\hat\b_n'x)}{F_{nh,\hat\b_n}(t-\hat\b_n'x)\{1-F_{nh,\hat\b_n}(t-\hat\b_n'x)\}}\,d\P_n(t,x,\d)
\\
&\qquad = I + II
\end{align*}
 Let ${\cal F}$ be a class of functions with the property that
 $$
 \int_{\ee/2<F_{\b}(u)<1-\ee/2} f'(u)^2\,du\le M.
 $$
 if $f\in{\cal F}$, for a fixed $M>0$. Using Proposition 5.1.9, p.\  393 in \cite{nickl:15}, with $m=1$, $p=2$ and $h\asymp n^{-1/5}$, we may assume that the functions $u\to F_{nh,\b}(u)$ and $u\to  \partial_\b F_{nh,\b}(u) $ belong to ${\cal F}$ . Since the plug-in estimates are monotonically increasing with probability tending to one we get that the function
 \begin{align*}
 (t,x)\mapsto 1_{[\ee,1-\ee]}\left(F_{nh,\b}(t-\b'x)\right),
 \end{align*}
 can be written in the form
 \begin{align*}
 (t,x)\mapsto 1_{\left[a_{\ee,F_{nh,\b}},b_{\ee,F_{nh,\b}}\right]}(t-\b'x)
 =1_{\left[a_{\ee,F_{nh,\b}},\infty\right)}(t-\b'x)-1_{\left(b_{\ee,F_{nh,\b}},\infty\right)}(t-\b'x),
 \end{align*}
 for $a_{\ee,F_{nh,\b}}\le b_{\ee,F_{nh,\b}}$ for large $n$, with probability tending to one.
 The function is therefore of uniformly bounded variation for $n$ sufficiently large (see also the proofs of Theorems \ref{theorem:method1} and \ref{theorem:method2}). It now follows that the bracketing $\zeta-$entropy $H_B(\zeta,{\cal K}_3,L_2(P_0))$ for the class
 ${\cal K}_3$ of functions consisting of the function which is identically zero and the functions
 \begin{align}
 \label{cal_K3}
 {\cal K}_3&=\Biggl\{(t,x,\d)\mapsto\big\{\partial_\b F_{nh,\b}(t-\b'x) - \partial_\b F_{\b}(t-\b'x)\bigr\}\frac{\d - F(t-\b'x)}{F(t-\b'x)\{1-F(t-\b'x)\}}\nonumber\\
 &\qquad\qquad\qquad\qquad\qquad\qquad\qquad\qquad\qquad\qquad\qquad\cdot1_{[\ee,1-\ee]}(F_{nh,\b}(t-\b'x)): F\in{\cal F},\,\b\in\Theta\Biggr\},
 \end{align}
 w.r.t.\ the $L_2$-distance satisfies:
 	\begin{align*}
 	\sup_{\zeta>0} \zeta H_B\left(\zeta,{\K_3},L_2(P_0)\right)=O(1),
 	\end{align*}
 	which implies:
 	\begin{align*}
 	\int_0^{\zeta} H_B\left(u,{\K_3},L_2(P_0)\right)^{1/2}\,du=O\left(\zeta^{1/2}\right),\qquad\zeta>0.
 	\end{align*}
Moreover, by Lemma \ref{lemma:distance_F{nh}-F_0} we also have,
\begin{align*}
&\int_{F_{nh,\b}(t-\b'x)\in[\ee,1-\ee]}\Biggl\{ \left\{\partial_\b F_{nh,\b}(t-\b'x) - \partial_\b F_{\b}(t-\b'x)\right\}\\
&\qquad\qquad\qquad\qquad\qquad\qquad\cdot\frac{\d - F_{nh,\b}(t-\b'x)}{F_{nh,\b}(t-\b'x)\{1-F_{nh,\b}(t-\b'x)\}}\Biggr\}^2\,dP_0(t,x,\d)\stackrel{p}{\to}0.
\end{align*}
This implies by an application of Lemma \ref{lemma:equi}, that,
\begin{align*}
&\int_{F_{nh,\hat \b_n}(t-\hat \b_n'x)\in[\ee,1-\ee]}\Biggl\{ \left\{\partial_\b F_{nh,\b}(t-\b'x)\mid_{\b =\hat \b_n} - \partial_\b F_{\b}(t-\b'x)\mid_{\b=\hat \b_n}\right\}\\
&\qquad\qquad\qquad\qquad\cdot\frac{\d - F_{nh,\hat \b_n}(t-\hat \b_n'x)}{F_{nh,\hat \b_nb}(t-\hat \b_n'x)\{1-F_{nh,\hat \b_n}(t-\hat \b_n'x)\}}\Biggr\}\,d(\P_n-P_0)(t,x,\d) = o_p\left(n^{-1/2}\right)
\end{align*}
Furthermore, an application of the Cauchy-Schwarz inequality and Lemma \ref{lemma:distance_F{nh}-F_0} yield that
\begin{align*}
&\sqrt{n}\int_{F_{nh,\hat \b_n}(t-\b x)\in[\ee, 1-\ee]}\left\{\partial_\b F_\b(t-\b'x)\mid_{\b = \hat \b_n}-\partial_\b F_{nh,\b}(t-\b'x)\mid_{\b = \hat \b_n}\right\}\\
&\qquad\qquad\qquad\qquad\qquad\qquad\qquad\qquad\cdot\left\{\frac{F_0(t-\b_0'x) - F_{nh,\hat \b_n}(t-\hat \b_n' x)}{F_{nh,\hat \b_n}(t-\hat \b_n' x)\{1-F_{nh,\hat \b_n}(t-\hat \b_n' x)\}}\right\}\,dP_0(t,x,\d)\\
&=O_p\left(n^{-1/10}\right)+o_p\left(\sqrt{n}(\hat\b_n-\b_0)\right)
\end{align*}
The conclusion is that
$$ II = o_p\left(n^{-1/2} + (\hat \b_n-\b_0)\right) $$
We now write:
	\begin{align*}
	I&=\int_{ F_{nh,\hat\b_n}(t-\hat\b_n'x)\in[\ee,1-\ee]}\partial_\b F_\b(t-\b'x)\mid_{\b = \hat \b_n}\\
	&\qquad\qquad\qquad\qquad\qquad\qquad\qquad\cdot\frac{\d - F_{nh,\hat\b_n}(t-\hat\b_n'x)}{ F_{nh,\hat\b_n}(t-\hat\b_n'x)\{1- F_{nh,\hat\b_n}(t-\hat\b_n'x)\}}\,d\P_n(t,x,\d)\\
	&=\int_{ F_{nh,\hat\b_n}(t-\hat\b_n'x)\in[\ee,1-\ee]}\partial_\b F_\b(t-\b'x)\mid_{\b = \hat \b_n}\\
	&\qquad\qquad\qquad\qquad\qquad\qquad\qquad\cdot\frac{\d -F_{\hat\b_n}(t-\hat\b_n'x)}{ F_{n,\hat\b_n}(t-\hat\b_n'x)\{1- F_{n,\hat\b_n}(t-\hat\b_n'x)\}}\,d\P_n(t,x,\d)\\
	&\qquad+\int_{ F_{nh,\hat\b_n}(t-\hat\b_n'x)\in[\ee,1-\ee]}\partial_\b F_\b(t-\b'x)\mid_{\b = \hat \b_n}\\
	&\qquad\qquad\qquad\qquad\qquad\qquad\qquad\qquad\qquad\cdot\frac{ F_{\hat\b_n}(t-\hat\b_n'x)-F_{n,\hat\b_n}(t-\hat\b_n'x)}{\hat F_{\hat\b_n}(t-\hat\b_n'x)\{1- F_{n,\hat\b_n}(t-\hat\b_n'x)\}}\,d\P_n(t,x,\d)\\
	&= I_a + I_b.
	\end{align*}
	We now get, using Lemma \ref{lemma:distance_F{nh}-F_0} and
	\begin{align*}
	\partial_\b F_\b(t-\b'x)\mid_{\b = \hat \b_n} = \E(X-x|T-\hat \b_n'X = t- \hat \b_n'x)f_0(t-\hat \b_n'x) + O_p\left(\hat \b_n - \b_0\right),
	\end{align*} 
that $	I_b=o_p\left(n^{-1/2} + \hat \b_n-\b_0\right)$.
The result of Theorem \ref{th:asymptotic_normality} now follows by showing that,
\begin{align}
\label{a}
&\int_{ F_{nh,\hat\b_n}(t-\hat\b_n'x)\in[\ee,1-\ee]}\partial_\b F_\b(t-\b'x)\mid_{\b = \hat \b_n} \nonumber\\
&\qquad\qquad\qquad\qquad\qquad\qquad\qquad\cdot\frac{\d -F_{\hat\b_n}(t-\hat\b_n'x)}{ F_{nh,\hat\b_n}(t-\hat\b_n'x)\{1- F_{nh,\hat\b_n}(t-\hat\b_n'x)\}}\,d(\P_n-P_0)(t,x,\d)\\
&=\int_{ F_0(t-\b_0'x)\in[\ee,1-\ee]}\partial_\b F_\b(t-\b'x)\mid_{\b =  \b_0}\nonumber\\
&\qquad\qquad\qquad\qquad\qquad\qquad\qquad\cdot\frac{\d -F_0(t-\b_0'x)}{ F_0(t-\b_0'x)\{1-F_0(t-\b_0'x)\}}\,d(\P_n-P_0)(t,x,\d)]\nonumber \\
&\qquad  + o_p\left(n^{-1/2} + \hat \b_n-\b_0\right) \nonumber
\end{align}	
and,
\begin{align}
\label{b}
&\int_{ F_{nh,\hat\b_n}(t-\hat\b_n'x)\in[\ee,1-\ee]}\partial_\b F_\b(t-\b'x)\mid_{\b = \hat \b_n}\nonumber\\
&\qquad\qquad\qquad\qquad\qquad\qquad\qquad\cdot\frac{\d - F_{\hat\b_n}(t-\hat\b_n'x)}{ F_{nh,\hat\b_n}(t-\hat\b_n'x)\{1- F_{nh,\hat\b_n}(t-\hat\b_n'x)\}}\,dP_0(t,x,\d)\\
&=\psi_{3,\ee}'(\b_0)( \hat \b_n-\b_0)  + o_p\left(n^{-1/2} + \hat \b_n-\b_0\right).\nonumber
\end{align}
The proof of (\ref{a}) and (\ref{b}) is similar to the proof of the corresponding steps given in the proof of Theorem \ref{theorem:method1} and omitted.
\end{proof}

\begin{remark} {\rm
It follows from the proof of Theorem \ref{th:asymptotic_normality} that
	\begin{align*}
	&\sqrt{n} I_{\ee}(\b_0)(\hat\b_n-\b_0)\\
	&\quad=n^{-1/2}\sum_{i=1}^nf_0(T_i-\b_0'X_i)\{\E(X_i|T_i-\b_0'X_i)-X_i\}\\
	&\qquad\qquad\qquad\qquad\cdot\frac{\dd_i-F_0(T_i-\b_0' X_i)}{F_0(T_i-\b_0' X_i)\{1-F_0(T_i-\b_0' X_i)\}}1_{[\ee,1-\ee]}\left\{F_0(T_i-\b_0'X_i)\right\} +o_p(1).
	\end{align*}
Therefore the result of Theorem \ref{th:alternative_MLE-expansion} follows.}
\end{remark}

\subsection{Estimation of the intercept}
\label{sectionA:intercept}
\begin{proof}[Proof of Theorem \ref{th:intercept}]
We will denote $dx_1\dots dx_k$ by $dx$. We have
\begin{align}
\label{hat_a-a0}
\hat \a_n -\a_0 &= \int u\,d F_{nh,\hat\b_n}(u)-\int u\,d F_0(u)=\int \bigl\{F_0(u)- F_{nh,\hat\b_n}(u)\bigr\}\,du \nonumber\\
&=\int\frac{F_0(t-\hat\b_n'x)- F_{nh,\hat\b_n}(t-\hat\b_n'x)}{f_{T-\hat\b_n'X}(t-\hat\b_n'x)}\,dG(t,x) \nonumber\\
&=\int\frac{F_0(t-\hat\b_n'x)-F_0(t-\b_0'x)}{f_{T-\hat\b_n'X}(t-\hat\b_n'x)}\,dG(t,x)
+\int\frac{F_0(t-\b_0'x)- F_{nh,\hat\b_n}(t-\hat\b_n'x)}{f_{T-\hat\b_n'X}(t-\hat\b_n'x)}\,dG(t,x)
\end{align}
For the first term in the last expression we get
\begin{align*}
&\int\frac{F_0(t-\hat\b_n'x)-F_0(t-\b_0'x)}{f_{T-\hat\b_n'X}(t-\hat\b_n'x)}\,dG(t,x)\\
&\qquad=\int\bigl\{F_0(u)-F_0(u+x'(\hat\b_n-\b_0))\}f_{X|T-\hat\b_n'X}(x|T-\hat\b_n'X=u)\,du\,dx\\
&\qquad\sim-\int x'(\hat\b_n-\b_0)f_0(u)f_{X|T-\b_0'X}(x|T-\b_0'X=u)\,du\,dx\\
&\qquad\sim-\left\{\int \E\{X'|T-\b_0'X=u\}f_0(u)\,du\right\}(\hat\b_n-\b_0)
\end{align*}
This term, multiplied with $\sqrt{n}$, is asymptotically normal, with expectation zero and variance
\begin{align*}
\s_1^2\stackrel{\text{\small def}}=a(\b_0)'I_{\ee}(\b_0)^{-1}\,a(\b_0),
\end{align*}
where $a(\b_0)$ is the $k$-dimensional vector, defined by
\begin{align*}
a(\b_0)=\int \E\{X|T-\b_0'X=u\}f_0(u)\,du.
\end{align*}
For the second term in (\ref{hat_a-a0}), we first note that,
\begin{align}
\label{F:hat F:0}
F_{nh,\hat\b_n}(t-\hat\b_n'x)-F_0(t-\b_0'x)&=\frac{\int\{\d-F_0(t-\b_0'x)\}K_h(t-\hat\b_n'x-u+\hat\b_n'y)\,d\P_n(u,y,\d)}{g_{nh,\hat\b_n}(t-\hat\b_n'x)}\,.
\end{align}
We  write (\ref{F:hat F:0}) as the sum of the integral over $dP_0$ and the integral over $d(\P_n-P_0)$ and show that the contribution of the $dP_0$ integral, evaluated in (\ref{hat_a-a0}) is negligible and that the contribution of the $d(\P_n-P_0)$ integral will yield an asymptotic normal distribution.

We have
\begin{align*}
&\int\{\d-F_0(t-\b_0'x)\}K_h(t-\hat\b_n'x-u+\hat\b_n'y)\,dP_0(u,y,\d)&\\
&\qquad =\int\{F_0(u-\b_0'y)-F_0(t-\b_0'x)\}K_h(t-\hat\b_n'x-u+\hat\b_n'y)\,dG(u,y)\\
&\qquad=\int\{F_0(v+(\hat\b_n-\b_0)y)-F_0(t-\b_0'x)\}\}K_h(t-\hat\b_n'x-v)\\
&\qquad\qquad \qquad\qquad \qquad\qquad \qquad\qquad \cdot f_{T-\hat\b_n'X}(v)f_{X|T-\hat\b_n'X}(y|T-\hat\b_n'X=v)\,dv\,dy\\
&\qquad=f_{T-\hat\b_n'X}(t-\hat\b_n'x)\int\{F_0(t-\hat\b_n'x+(\hat\b_n-\b_0)y)-F_0(t-\b_0'x)\}\\
&\qquad\qquad\qquad\qquad\qquad\qquad\qquad\qquad\cdot f_{X|T-\hat\b_n'X}(y|T-\hat\b_n'X=t-\hat\b_n'x)\,dy
+O_p\left(h^2\right)\\
&\qquad=f_{T-\hat\b_n'X}(t-\hat\b_n'x)f_0(t-\b_0'x)\bigl(\hat\b_n-\b_0\bigr)'\E\{X-x|T-\hat\b_n'X=t-\hat\b_n'x\}\\
&\qquad\qquad\qquad\qquad\qquad\qquad\qquad\qquad\qquad\qquad\qquad\qquad\qquad+O_p\left(h^2\right)+o_p\bigl(\|\hat\b_n-\b_0\|\bigr),
\end{align*}
where $\|x\|$ is the euclidean norm of the vector $x$.  Hence we get
\begin{align*}
&\int\frac{\int\{\d-F_0(t-\b_0'x)\}K_h(t-\hat\b_n'x-u+\hat\b_n'y)\,dP_0(u,y,\d)}{g_{nh,\hat\b_n}(t-\hat\b_n'x)f_{T-\hat\b_n'X}(t-\hat\b_n'x)}\,dG(t,x)\\
&\quad=(\hat\b_n-\b_0)'\int \frac{f_0(t-\b_0'x)E\{X-x|T-\b_0'X=t-\b_0'x\}}{g_{nh}(t-\hat\b_n'x)}\,dG(t,x)+O_p\left(h^2\right)+o_p(\hat\b_n-\b_0)\\
&\quad=(\hat\b_n-\b_0)'\int f_0(v)\E\{X-x|T-\b_0'X=v\}f_{X|T-\b_0'X}(x|T-\b_0'X=v)\,dx\,dv\\
&\qquad\qquad\qquad\qquad\qquad\qquad\qquad\qquad\quad\qquad\qquad\qquad\qquad\qquad\qquad+O_p\left(h^2\right)+o_p\bigl(\|\hat\b_n-\b_0\|\bigr)\\
&\quad=O_p\left(h^2\right)+o_p\bigl(\|\hat\b_n-\b_0\|\bigr),
\end{align*}
which is $o_p(n^{-1/2})$ if $h\ll n^{-1/4}$. 

Finally,
\begin{align*}
&\sqrt{n}\int\frac{\int\{\d-F_0(t-\b_0'x)\}K_h(t-\hat\b_n'x-u+\hat\b_n'y)\,d\bigl(\P_n-P_0\bigr)(u,y,\d)}{g_{nh,\hat\b_n}(t-\hat\b_n'x)f_{T-\hat\b_n'X}(t-\hat\b_n'x)}\,dG(t,x)&\\
&\qquad=\sqrt{n}\iint\frac{\{\d-F_0(t-\b_0'x)\}K_h(t-\hat\b_n'x-u+\hat\b_n'y)}{g_{nh,\hat\b_n}(t-\hat\b_n'x)f_{T-\hat\b_n'X}(t-\hat\b_n'x)}\,dG(t,x)\,d\bigl(\P_n-P_0\bigr)(u,y,\d)\\
&\qquad=\sqrt{n}\int\frac{\{\d-F_0(u-\b_0'y)\}}{f_{T-\b_0'X}(u-\b_0'y)}\,d\bigl(\P_n-P_0\bigr)(u,y,\d)+O_p\left(h^2\right)+O_p\bigl(\|\hat\b_n-\b_0\|\bigr)
\end{align*}
is asymptotically normal, with expectation zero and variance
\begin{align}
\label{curstat_var}
&\int\frac{F_0(v)\{1-F_0(v)\}}{f_{T-\b_0'X}(v)}\,dv,
\end{align}
if $h\ll n^{-1/4}$.

Both terms in the representation on the right of (\ref{hat_a-a0}) are, apart from a negligible contribution, sums of independent variables with expectation zero. By Theorem \ref{th:alternative_MLE-expansion} we have
\begin{align*}
&\sqrt{n}(\hat\b_n-\b_0)\\
&\quad=n^{-1/2} I_{\ee}(\b_0)^{-1}\sum_{i=1}^nf_0(T_i-\b_0X_i)\{\E(X_i|T_i-\b_0'X_i)-X_i\}\\
&\qquad\qquad\qquad\qquad\cdot\frac{\dd_i-F_0(T_i-\b_0' X_i)}{F_0(T_i-\b_0' X_i)\{1-F_0(T_i-\b_0' X_i)\}}1_{[\ee, 1-\ee]}\left\{F_0(T_i-\b_0'X_i)\right\} +o_p(1).
\end{align*}
 and the second term of (\ref{hat_a-a0}) has the representation
 \begin{align*}
 n^{-1/2}\sum_{i=1}^n\frac{\dd_i-F_0(T_i-\b_0'X_i)}{f_{T-\b_0'X}(T_i-\b_0'X_i)}.
 \end{align*}
 By the independence of the summands with indices $i\ne j$, the only contribution to the covariance of the two terms in the representation can come from summands with the same index. But,
 \begin{align*}
 &\E\left\{\frac{ f_0(T_i-\b_0'X_i)\{\E(X_i|T_i-\b_0'X_i)-X_i\}\{\dd_i-F_0(T_i-\b_0' X_i)\}^2}{F_0(T_i-\b_0' X_i)\{1-F_0(T_i-\b_0'X_i)\}f_{T-\b_0'X}(T_i-\b_0'X_i)}1_{[\ee, 1-\ee]}\{F_0(T_i-\b_0'X_i)\}\right\}.
 \\
 &= \int_{F_0(u-\b_0'y)\in[\ee, 1-\ee]}\frac{ f_0(u-\b_0'y)\{\E(X|T-\b_0'X=u -\b_0'y)-y\}\{\d-F_0(u-\b_0' y)\}^2}{F_0(u-\b_0'y)\{1-F_0(u-\b_0'y)\}f_{T-\b_0'X}(u-\b_0'y)}dP_0(u,y,\d)
 \\
 &= \iint_{F_0(v)\in[\ee, 1-\ee]}\frac{ f_0(v)\{\E(X|T-\b_0'X=v)-y\}F_0(v)\{1-F_0(v)\}}{F_0(v)\{1-F_0(v)\}}f_{X|T-\b_0'X}(y|v) dvdy
 \\
 &= \int_{F_0(v)\in[\ee, 1-\ee]} \int \{\E(X|T-\b_0'X=v)-y\}f_{X|T-\b_0'X}(y|v) dy \frac{f_0(v) F_0(v)\{1-F_0(v)\}}{F_0(v)\{1-F_0(v)\}}dv
 \\
 &= 0
 \end{align*} 
 So the covariance is zero and Theorem \ref{th:intercept} follows.

\end{proof}

\bibliographystyle{imsart-nameyear}
\bibliography{cupbook}

\end{document}